\documentclass[11pt]{article}
\usepackage[T1]{fontenc}
\usepackage{lmodern}
\usepackage{amsmath,amsthm,amssymb}
\usepackage[usenames,dvipsnames]{xcolor}
\usepackage{enumerate}
\usepackage{graphicx}
\usepackage{cite}
\usepackage{comment}
\usepackage{oands}
\usepackage{tikz}
\usepackage{changepage}
\usepackage{bbm}
\usepackage{mathtools}
\usepackage[margin=1in]{geometry}
\usepackage[pagewise,mathlines]{lineno}
\usepackage{appendix}
\usepackage{stmaryrd}
\usepackage{multicol}
\usepackage{microtype}
\usepackage[colorinlistoftodos]{todonotes}
\usepackage{dsfont}
\usepackage{mathrsfs}  
\usepackage{multirow}
\usepackage{graphicx,subfigure}
\usepackage{array}
\usepackage{bm}
\newcolumntype{P}[1]{>{\centering\arraybackslash}p{#1}}
\usepackage[pdftitle={Reconstructing SLE-decorated Liouville quantum gravity surfaces from random permutons}, pdfauthor={Jacopo Borga and Ewain Gwynne},
colorlinks=true,linkcolor=NavyBlue,urlcolor=RoyalBlue,citecolor=PineGreen,bookmarks=true,bookmarksopen=true,bookmarksopenlevel=2,unicode=true,linktocpage]{hyperref}
\usepackage[capitalise,noabbrev]{cleveref}

\theoremstyle{plain}
\newtheorem{thm}{Theorem}[section]
\newtheorem{cor}[thm]{Corollary}
\newtheorem{lem}[thm]{Lemma}
\newtheorem{prop}[thm]{Proposition}

\def\@rst #1 #2other{#1}
\newcommand\MR[1]{\relax\ifhmode\unskip\spacefactor3000 \space\fi
  \MRhref{\expandafter\@rst #1 other}{#1}}
\newcommand{\MRhref}[2]{\href{http://www.ams.org/mathscinet-getitem?mr=#1}{MR#2}}

\theoremstyle{definition}
\newtheorem{defn}[thm]{Definition}
\newtheorem{remark}[thm]{Remark}

\newtheorem*{ack}{Acknowledgements}

\numberwithin{equation}{section}

\newcommand{\dsb}{\begin{adjustwidth}{2.5em}{0pt}
\begin{footnotesize}}
\newcommand{\dse}{\end{footnotesize}
\end{adjustwidth}}

\newcommand{\ssb}{\begin{adjustwidth}{2.5em}{0pt}}
\newcommand{\sse}{\end{adjustwidth}}

\newcommand{\aryb}{\begin{eqnarray*}}
\newcommand{\arye}{\end{eqnarray*}}
\def\alb#1\ale{\begin{align*}#1\end{align*}}
\def\allb#1\alle{\begin{align}#1\end{align}}
\newcommand{\eqb}{\begin{equation}}
\newcommand{\eqe}{\end{equation}}
\newcommand{\eqbn}{\begin{equation*}}
\newcommand{\eqen}{\end{equation*}}

\newcommand{\BB}{\mathbbm}
\newcommand{\ol}{\overline}

\newcommand{\op}{\operatorname}

\newcommand{\frk}{\mathfrak}
\newcommand{\eqD}{\overset{d}{=}}
\newcommand{\ep}{\varepsilon}
\newcommand{\rta}{\rightarrow}

\newcommand{\wt}{\widetilde}
\newcommand{\wh}{\widehat} 
\newcommand{\mcl}{\mathcal}

\newcommand{\bdy}{\partial}

\newcommand{\perm}{{\boldsymbol{\pi}}}

\let\originalleft\left
\let\originalright\right
\renewcommand{\left}{\mathopen{}\mathclose\bgroup\originalleft}
\renewcommand{\right}{\aftergroup\egroup\originalright}

\title{Reconstructing SLE-decorated Liouville quantum gravity surfaces from random permutons}
\date{    }
\author{
\begin{tabular}{c} Jacopo Borga\\[-3pt]\small MIT \end{tabular}
\begin{tabular}{c} Ewain Gwynne\\[-3pt]\small University of Chicago \end{tabular} 
}

\begin{document}

\maketitle

\begin{abstract}
	Permutons constructed from a Liouville quantum gravity surface and a pair of space-filling Schramm-Loewner evolutions (SLEs) have been shown -- or are conjectured -- to describe the scaling limit of various natural models of random constrained permutations. We prove that, in two distinct and natural settings,  these permutons uniquely determine, modulo rotation, scaling, translation and reflection, both the Liouville quantum gravity surface and the pair of space-filling SLEs used in their construction. In other words, the Liouville quantum gravity surface and the pair of space-filling SLEs can be deterministically reconstructed from the permuton.
	
	Our results cover the cases of the \emph{skew Brownian permutons}, the universal limits of pattern-avoiding permutations, and the \emph{meandric permuton}, which is the conjectural permuton limit of permutations obtained from uniform meanders.
	
	In the course of the proof, we give a detailed description of how the support of the permuton relates to the multiple points of the two space-filling SLEs.
	
\end{abstract}

\tableofcontents

\section{Introduction}

Permutons are probability measures on the unit square with uniform marginals which have been used to describe the scaling limit of many models of random permutations (see, for instance, \cite{bbfgp-universal,bbfs-tree-sep,borga-strong-baxter} and references therein). 

Liouville quantum gravity surfaces decorated by (multiple) space-filling SLEs are the canonical random surfaces endowed with (multiple) space-filling paths used to describe the scaling limit of many different models of decorated planar maps, that is, planar maps endowed with certain statistical-mechanics/combinatorial models (see, for instance, \cite{legall-brownian-geometry,ghs-mating-survey,bp-lqg-notes} and references therein).

The interplay between permutons and SLE-decorated Liouville quantum gravity surfaces has emerged in a series of recent works \cite{borga-skew-permuton, bhsy-baxter-permuton,bgs-meander}. Let us briefly recall this connection.

\medskip

Fix $\gamma \in (0,2)$. Let $h$ be a random generalized function on $\BB C$ corresponding to a singly marked unit area $\gamma$-Liouville quantum sphere $(\BB C, h, \infty)$, and let $\mu_h$ be its associated $\gamma$-LQG area measure. We will review the definitions of these objects in \cref{sec-intro-lqg}, but for now the reader can just think of $\mu_h$ as a random, non-atomic, Borel probability measure on $\BB C$ which assigns positive mass to every open subset of $\BB C$.

Independently from $h$, let $(\eta_1,\eta_2)$ be a random pair consisting of a whole-plane space-filling SLE$_{\kappa_1}$ curve and a whole-plane space-filling SLE$_{\kappa_2}$ curve, each going from $\infty$ to $\infty$. We will review the definition of these random curves in Section \ref{sect:SLES} and \ref{sec:SLE_as_flow}, but for now the reader can just think of $(\eta_1,\eta_2)$ as a pair of random non-self-crossing space-filling curves in $\BB C$ which each visit almost every point of $\BB C$ exactly once.

We parametrize each of $\eta_1$ and $\eta_2$ by $\mu_h$-mass, i.e.,
\begin{equation*}
	\mu_h(\eta_1([0,t])) = \mu_h(\eta_2([0,t]) = t\quad\text{for each}\quad t\in[0,1].
\end{equation*} 
We emphasize that the coupling of our two SLE$_\kappa$ curves (viewed modulo time parametrization) is arbitrary. We will be interested in cases where the two curves determine each other (in particular, the setting of the \emph{skew Brownian permutons} in~\cite{borga-skew-permuton}) as well as the case where the two curves are independent (in particular, this setting include the \emph{meandric permuton} in \cite{bgs-meander}). Note that the pair $(\eta_1,\eta_2)$, viewed as curves modulo time parametrization, is required to be independent from $h$. 

Let $\psi : [0,1]\rta[0,1]$ be a Lebesgue measurable function such that 
\eqb \label{eqn-psi-property}
\eta_1(t) = \eta_2(\psi(t)),\quad\forall t \in [0,1] .
\eqe 
(Note that the function $\psi$ is not unique due to the multiple points of $\eta_2$.)
We define the random permuton associated with $(\eta_1,\eta_2)$ to be the random probability measure $\perm$ on $[0,1]^2$ defined by\footnote{This definition does not depend on the choice of $\psi$; see \cref{lem-permuton-defined}.}
\eqb \label{eqn-permuton-def}
\perm(A) = \perm_{\eta_1,\eta_2}(A) = \op{Leb}\left\{ t\in [0,1] : (t,\psi(t)) \in A \right\} ,\quad \text{$\forall$ Borel sets $A\subset \BB C$}. 
\eqe

\newpage

There are two special cases of the construction in \eqref{eqn-permuton-def} which are particularly interesting:

\begin{itemize}
	\item The \emph{skew Brownian permutons}  $\mu_{\rho,q}$ indexed by $(\rho,q) \in (-1,1] \times (0,1)$, which are known to be the universal limits of pattern-avoiding permutations \cite{bbfgp-universal,bbfs-tree-sep,bm-baxter-permutation,borga-strong-baxter}. The original construction of the skew Brownian permutons does not involve SLEs and LQG, and is given in terms of a $q$-dependent system of stochastic  differential equations driven by a two-dimensional Brownian excursion in the first quadrant and of correlation $\rho$; see \cite{borga-skew-permuton} for further details. But, in~\cite[Theorem 1.17]{borga-skew-permuton}, the first author of the present paper showed that for each choice of parameters\footnote{Note that here we excluded the case $\rho=1$, i.e.\ the case of the \emph{Brownian separable permuton}.} $(\rho,q) \in (-1,1) \times (0,1)$ for the skew Brownian permutons, there exists a $q$-dependent coupling of two whole-plane space-filling SLE$_{\kappa = 16/\gamma^2}$ curves with $\rho = -\cos(\pi \gamma^2/4)$ such that the permuton~\eqref{eqn-permuton-def} coincides with the skew Brownian permuton. The specific coupling considered for the two space-filling SLE$_{\kappa = 16/\gamma^2}$ curves $\eta_1$ and $\eta_2$ is the \emph{imaginary geometry} coupling of Miller and Sheffield~\cite{ig4}, where $\eta_1$ and $\eta_2$ are the space-filling counterflow lines\footnote{As explained in \cref{sec:SLE_as_flow}, every whole-plane space-filling SLE$_{\kappa}$ can be realized as a space-filling counterflow line of a whole-plane GFF.} of angle $0$ and $\theta(q)\in(-\pi/2,\pi/2)$ of the same whole-plane Gaussian free field (GFF). The relation between $\theta$ and $q$ is not known explicitly. We will review this coupling in \cref{sec:SLE_as_flow}.

	\item The \emph{meandric permuton} is the permuton $\perm$ in~\eqref{eqn-permuton-def} in the case when $\gamma = \sqrt{\frac13 \left( 17 - \sqrt{145} \right)}$, $\kappa_1 = \kappa_2 = 8$, and $\eta_1$ and $\eta_2$ are independent, viewed as curves modulo time parametrization. This is the conjectural permuton limit of uniform meandric permutations \cite{bgs-meander}, i.e., the permutations in bijection with meanders \cite{lacroix-meander-survey}.
\end{itemize}

The permutons introduced in  \eqref{eqn-permuton-def} are defined in terms of a $\gamma$-LQG measure (coming from a singly marked unit area $\gamma$-Liouville quantum sphere) and a pair of space-filling whole-plane SLE curves. A natural question is how much ``information'' about the $\gamma$-Liouville quantum sphere and the pair of SLEs is contained in $\perm$.

The main goal of this paper is to show that in many cases (including the two cases above), the permuton $\perm$ contains \emph{all} the  ``information'' about the $\gamma$-Liouville quantum sphere and the pair of SLEs. That is, one can deterministically recover from $\perm$ both the $\gamma$-Liouville quantum sphere and the pair of SLEs. In the next section, we make this statement precise.

\subsection{Main result: permutons determine SLEs and LQG}

The \textbf{(closed) support} of $\perm$ is the set
\eqb \label{eqn-permuton-supp}
\op{supp} \perm :=  \left( \text{intersection of all closed sets $K\subset [0,1]^2$ with $\perm(K) = 1$} \right) .
\eqe
Recall that $(\BB C, h, \infty)$ denotes the singly marked unit area $\gamma$-Liouville quantum sphere determining the measure $\mu_h$; see \cref{def-sphere} for a precise definition of  $(\BB C, h, \infty)$. 

In the two cases when $\eta_1$ and $\eta_2$ are independent modulo time parametrization -- such as for the meandric permuton -- or $\eta_1$ and $\eta_2$ are the space-filling counterflow lines of the same whole-plane GFF -- for instance, this is the case for the skew Brownian permuton -- it is possible to recover the curve-decorated quantum surface $(\BB C,h,\infty,\eta_1,\eta_2)$, viewed modulo complex conjugation, from just the permuton $\perm$; see \cref{def-curve-decorated} for a precise definition of curve-decorated quantum surfaces.

\begin{thm} \label{thm-permuton-determ-lqg-sles}
	Let $\gamma \in (0,2)$. 
	Let $(\BB C , h , \infty)$ be a singly marked unit area $\gamma$-Liouville quantum sphere.
	Let $(\eta_1,\eta_2)$ be a pair of whole-plane space-filling SLEs from $\infty$ to $\infty$, sampled independently from $h$ and then parametrized by $\gamma$-LQG mass with respect to $h$.  
	Let $\perm$ be the permuton associated with $(\eta_1,\eta_2)$ as in~\eqref{eqn-permuton-def}.  
	Assume that either
	\begin{enumerate}
		\item \label{item-indep} $\eta_1$ and $\eta_2$ are independent (viewed modulo time parametrization) and $\eta_1$ has parameter $\kappa_1 > 4$ and $\eta_2$ has parameter $\kappa_2 > 4$; or 
		\item \label{item-ig-gff} $\eta_1$ and $\eta_2$ are the space-filling counterflow lines of the same whole-plane GFF with angles 0 and $\theta-\pi/2$ for some $\theta\in(-\pi/2,\pi/2)$ and both have parameter $\kappa >4$ (when $\gamma^2  = 16/\kappa$, the permuton $\perm$ is a skew Brownian permuton). 
	\end{enumerate}
	Almost surely, the closed support of $\perm$ determines the curve-decorated quantum sphere $(\BB C , h ,  \infty , \eta_1 , \eta_2 )$ up to complex conjugation.  
\end{thm}

An equivalent rephrasing of the above conclusion is that there exists a measurable function $F$ such that 
\begin{equation}\label{eq:function-det}
	(\BB C , h ,  \infty , \eta_1 , \eta_2 )_{\mathrm{cc}}= F(\op{supp} \perm),
\end{equation}
where $(\BB C , h ,  \infty , \eta_1 , \eta_2 )_{\mathrm{cc}}$ denotes the equivalence class of 5-tuples $(\BB C , h ,  \infty , \eta_1 , \eta_2 )$ under the equivalence relation whereby\footnote{The specific form of this equivalence relation is standard in LQG theory and is justified by e.g.\ \cite[Theorem 2.8]{bp-lqg-notes}, where is shown that if $f$ is a conformal map, then $\mu_h \circ f^{-1}=\mu_{h\circ f^{-1}+Q\log((f^{-1})')}$.}
\eqb\label{eq:cc}
(\BB C ,h , \infty, \eta_1 , \eta_2 ) \sim_{\mathrm{cc}} \left(\BB C ,h \circ f + Q\log|f'| , \infty, f^{-1} \circ \eta_1 , f^{-1} \circ \eta_2 \right), \quad\text{with}\quad Q = \frac{2}{\gamma} + \frac{\gamma}{2},
\eqe 
whenever\footnote{The notation $\sim_{\mathrm{cc}}$ has been chosen to highlight that we are also considering complex conjugations in the definition of our equivalence classes $(\BB C , h ,  \infty , \eta_1 , \eta_2 )_{\mathrm{cc}}$. As explained in \cref{def-curve-decorated} below, the standard definition of the equivalence classes $(\BB C , h ,  \infty , \eta_1 , \eta_2 )$ only involves complex affine transformation (without complex conjugations).} $f : \BB C\rta \BB C$ is a complex affine transformation \emph{or} a complex affine transformation composed with complex conjugation; equivalently, whenever $f : \BB C\rta \BB C$ is a composition of rotations, scalings, translations, or reflections in $\BB C$.

The permuton $\perm$ cannot determine the curve-decorated quantum surface $(\BB C , h ,  \infty , \eta_1 , \eta_2 )$ -- see \cref{def-curve-decorated} below -- because replacing $(\BB C , h ,  \infty , \eta_1, \eta_2 )$ by $(\BB C , h(\ol{\cdot}) , \infty,  \ol \eta_1, \ol \eta_2)$ does not change $\perm$.

The function $F$ in \cref{eq:function-det} is completely explicit and will be detailed at the end of the paper in \cref{sect:explicit-function}.

\begin{remark}
	As a consequence of \cref{thm-permuton-determ-lqg-sles}, the support of the permuton $\perm$ in \cref{thm-permuton-determ-lqg-sles} determines the permuton $\perm$ itself. Indeed, according to \cref{thm-permuton-determ-lqg-sles}, the support of $\perm$ determines the curve-decorated quantum sphere $(\mathbb{C}, h, \infty, \eta_1, \eta_2)$ up to complex conjugation. This, of course, suffices to determine the permuton $\perm$ via its construction.
\end{remark}

We now discuss three consequences of our result. The first is the following result, whose simple proof can be found at the end of \cref{sect:explicit-function}.
	
\begin{cor}\label{cor:mutual-sing1}
	Let $\perm$ be as in the statement of  \cref{thm-permuton-determ-lqg-sles}.
	Then the laws of the permutons $\perm$ obtained for different values of $(\gamma,\kappa_1,\kappa_2)$ in Assumption~\ref{item-indep}, or for different values of $(\gamma,\kappa, \theta)$ in Assumption~\ref{item-ig-gff}, are mutually singular.
	
	In particular, the laws of the skew Brownian permutons  $\mu_{\rho,q}$  are mutually singular for all $(\rho,q) \in (-1,1) \times (0,1)$.
\end{cor}

%

\begin{remark}
	In the setting of the skew Brownian permuton $\mu_{\rho,q}$ with $(\rho,q) \in (-1,1) \times (0,1)$, the support of the permuton $\mu_{\rho,q}$ determines the two-dimensional Brownian excursion $(X_t,Y_t)$ in the first quadrant and of correlation $\rho$ used in its original construction. 
	Indeed, according to \cref{thm-permuton-determ-lqg-sles},  the support of $\mu_{\rho,q}$ determines the curve-decorated quantum sphere $(\mathbb{C}, h, \infty, \eta_1, \eta_2)$ up to complex conjugation, and this is enough to (explicitly) determine $(X_t,Y_t)$ thanks to \cite[Theorem 1.1]{sphere-constructions} and \cite[Theorem 1.17]{borga-skew-permuton}.
\end{remark}

We now explain another important consequence of our \cref{thm-permuton-determ-lqg-sles}. The mating-of-trees theorem~\cite{wedges}, due to Duplantier, Miller, and Sheffield, establishes that when $\gamma^2=16/\kappa$, a $\gamma$-Liouville quantum gravity sphere $(\mathbb{C}, h, \infty)$ decorated by an independent space-filling SLE curve $\eta$ of parameter $\kappa$ are both determined by a natural pair of correlated Brownian excursions, constructed from $(\mathbb{C}, h, \infty, \eta)$. This remarkable result has been extensively used to establish connections between discrete models of curve-decorated random planar maps and SLE-decorated Liouville quantum gravity surfaces in the so-called \emph{peanosphere sense}. Proving convergence in the peanosphere sense for these models involves verifying that a certain 2D random walk, which encodes the random planar map under study, converges to the 2D Brownian motion with the correct correlation between its two coordinates, see for instance~\cite{shef-burger, lsw-schnyder-wood, gkmw-burger, kmsw-bipolar, bhs-site-perc}.

This notion of convergence has two important features. First, it allows one to extract properties of the discrete models from the continuum and vice versa, as demonstrated, for instance, in \cite{ghs-map-dist}. Second, it constitutes the first step towards establishing stronger types of convergence \cite{gms-tutte, hs-cardy-embedding}. Notably, the mating-of-trees/peanosphere theory applies exclusively in the regime where $\gamma^2 = 16/\kappa$.

Our \cref{thm-permuton-determ-lqg-sles} admits a similar interpretation. Several models of curve-decorated random planar maps are bijectively encoded by corresponding models of random constrained permutations (e.g., bipolar orientations and Baxter permutations \cite{fps-counting-bipolar}, and meanders and meandric permutations \cite{lacroix-meander-survey}). Thus, permuton convergence for such permutation models can be reinterpreted as a new notion of convergence for the corresponding models of curve-decorated random planar maps towards SLE-decorated Liouville quantum gravity spheres. Crucially, this new notion of convergence is not restricted to the case $\gamma^2 = 16/\kappa$. For instance, it includes the case of random meanders, where, conjecturally, $\gamma^2 \neq 16/\kappa$.

In a similar vein to the mating-of-trees/peanosphere theory, we anticipate that this new notion of convergence will facilitate the transfer of results between the continuum and the discrete, and will serve as a foundational step towards establishing stronger results in the future.

\medskip

We present the third consequence of \cref{thm-permuton-determ-lqg-sles}. We start by recalling the following re-rooting invariance property of the meandric permuton.

\begin{thm}[{\cite[Theorem 1.22]{bgs-meander}}]\label{thm-permuton-re-root}
	Suppose that we are in the setting of  \cref{thm-permuton-determ-lqg-sles} (Assumption~\ref{item-indep}) with $\gamma\in (0,2)$ arbitrary, $\kappa_1 = \kappa_2 =8$. 
	For $t\in [0,1]$, let $\psi(t)$ be the a.s.\ unique
	$s\in [0,1]$ for which $(t,s)$ is in the closed support of $\perm$. Let $\perm_t$ be the pushforward of the measure $\perm$ under the mapping $[0,1]^2 \rta [0,1]^2$ defined by
	\eqb \label{eq:rerooting}
	(u, v) \mapsto \left( u-t   - \lfloor u - t   \rfloor , v - \psi(t)   -  \lfloor v - \psi(t) \rfloor \right) . 
	\eqe
	For each fixed $t\in [0,1]$, we have $\perm_t \eqD \perm$. 
\end{thm} 

As a consequence of \cref{thm-permuton-determ-lqg-sles} and \cite[Proposition 5.2]{bgs-meander}, the statement of Theorem~\ref{thm-permuton-re-root} is not true when either $\kappa_1$ or $\kappa_2$ is not equal to 8.

\begin{cor}\label{cor:meandric-perm}
	Suppose that we are in one of the two settings in \cref{thm-permuton-determ-lqg-sles}. 
	Let $T\in [0,1]$ be a uniform random variable independent from $\perm$.
	Assume that, in the notation of Theorem~\ref{thm-permuton-re-root}, we have $\perm_T \eqD \perm$.
	Then $\kappa_1=\kappa_2=8$.  
\end{cor}

The proof of \cref{cor:meandric-perm} can be found at the end of \cref{sect:explicit-function}. We recall from \cite[Section 1.5.3]{bgs-meander} that it is very natural to expect that the permuton limit of uniform meandric permutations is re-rooting invariant, since this property is satisfied at the discrete level. 
Therefore, if the permuton limit of uniform meandric permutations is one of the permutons in \cref{thm-permuton-determ-lqg-sles} (Assumption~\ref{item-indep}) -- as conjectured in \cite[Conjecture 1.8]{bgs-meander} -- then we should have that $\kappa_1=\kappa_2=8$. We do not yet have a special property which singles out the LQG parameter $\gamma =  \sqrt{\frac13 \left( 17 - \sqrt{145} \right)}$.

\medskip

The proof of \cref{thm-permuton-determ-lqg-sles} will follow from a fine analysis of the properties of the support of the permuton $\perm$. In particular, we will need to carefully describe the interplay between the support of the permuton $\perm$ and the multiple points of the two SLEs $\eta_1$ and $\eta_2$, as explained in the next section.

\subsection{Proof strategy and intermediate results: interplay between the support of the permuton and  multiple points of SLEs}

For a parametrized curve $\eta: [0,1] \rta \BB C$, we define the \textbf{intersection set} of $\eta$ by 
\eqb \label{eqn-intersect-set}
\mcl T \mcl M = \mcl T \mcl M(\eta) := \left\{ (s,t) \in [0,1]^2 : \eta(s) =\eta(t) \right\} .
\eqe
We use the letter $\mcl T \mcl M$, because $\mcl T \mcl M(\eta)$ describes the set of pairs of hitting-times ($\mcl T$) of multiple points ($\mcl M$) of $\eta$.
Since $\eta$ is continuous, the set $\mcl T \mcl M$ is closed.

One fundamental step in the proof of \cref{thm-permuton-determ-lqg-sles} -- see below for further explanations -- is the following statement for a single SLE. This result is of independent interest and proved in \cref{sec-multi-det-lqg-andsles}.

\begin{thm} \label{thm-sle-msrble}
	Let $\gamma\in (0,2)$ and $\kappa >4$. Let $(\BB C ,h, \infty)$ be a singly marked unit area $\gamma$-Liouville quantum sphere and let $\eta$ be a whole-plane space-filling SLE$_\kappa$ from $\infty$ to $\infty$ sampled independently from $h$ and then parametrized by $\gamma$-LQG mass with respect to $h$. 
	
	Almost surely, the set $\mcl T \mcl M (\eta)\subset [0,1]^2$ defined in~\eqref{eqn-intersect-set} determines the curve-decorated quantum surface $(\BB C , h ,  \infty , \eta )$ up to complex conjugation.
\end{thm}

The set $\mcl T \mcl M(\eta)$ does not determine the curve-decorated quantum surface $(\BB C , h ,  \infty , \eta )$ since replacing $(\BB C , h ,  \infty , \eta )$ by $(\BB C , h(\ol{\cdot}) , \infty,  \ol \eta)$ does not change $\mcl T \mcl M(\eta)$. 

\medskip

To prove Theorem~\ref{thm-sle-msrble}, we will first consider the graph  
\[\mcl G^n = (\mcl V\mcl G^n, \mcl E\mcl G^n),\]
where each vertex $x \in \mcl V\mcl G^n = (0,1] \cap \frac{1}{n} \BB Z$ corresponds to the SLE cell $\eta ([x-1/n,x])$. Two vertices $x, y \in \mcl V\mcl G^n$ are joined by an edge if and only if $\eta ([x-1/n,x])$ and $\eta ([y-1/n,y])$ intersect along a non-trivial boundary arc. We point out that when $\gamma^2 = 16/\kappa$ and $\mcl G^n$ is viewed as a planar map rather than just a graph, $\mcl G^n$ is the so-called \emph{mated-CRT map}~\cite{ghs-dist-exponent,gms-tutte}.

Next, we will show that $\mcl T \mcl M$ determines\footnote{We are only able to show that $\mcl T \mcl M$ determines $\mcl G^n$ as a \emph{graph}, not as a planar map (this is related to the fact that the curve-decorated LQG surface is determined by $\mcl T \mcl M$ only up to complex conjugation).} $\mcl G^n$. We will then define an embedding function $\Phi^n : \mcl V\mcl G^n \to \BB C$ in terms of observables related to the random walk on $\mcl G^n$ (in particular, we will consider a version of the Tutte embedding of $\mcl G^n$ as in~\cite{gms-tutte}).

Since the random walk on $\mcl G^n$, under the a priori embedding $x \mapsto \eta (x)$, converges to Brownian motion modulo time parametrization, as established\footnote{The result in \cite{gms-tutte} is only stated for the mated-CRT map case, but the proof does not require that $\gamma^2 = 16/\kappa$.} in~\cite{gms-tutte} (see also Proposition~\ref{prop-rw-conv} below), we will deduce that $\Phi^n$ is in some sense close to the a priori embedding $x \mapsto \eta (x)$ when $n$ is large. Since $\mcl G^n$ is determined by $\mcl T \mcl M$, this will lead to a proof of Theorem~\ref{thm-sle-msrble}.

\medskip

In \cref{sect:deduction}, we will easily deduce \cref{thm-permuton-determ-lqg-sles}  combining the next result, proved in \cref{sect:supp_perm}, with \cref{thm-sle-msrble}.


\begin{thm} \label{thm-permuton-multi-points}
	Let $\gamma \in (0,2)$. 
	Let $(\BB C , h , \infty)$ be a singly marked unit area $\gamma$-Liouville quantum sphere.
	Let $(\eta_1,\eta_2)$ be a pair of whole-plane space-filling SLEs from $\infty$ to $\infty$, sampled independently from $h$ and then parametrized by $\gamma$-LQG mass with respect to $h$.  
	Let $\perm$ be the permuton associated with $(\eta_1,\eta_2)$ as in~\eqref{eqn-permuton-def}.  
	Assume that either
	\begin{enumerate}
		\item \label{item-indep2} $\eta_1$ and $\eta_2$ are independent (viewed modulo time parametrization) and $\eta_1$ has parameter $\kappa_1 > 4$ and $\eta_2$ has parameter $\kappa_2 > 4$; or 
		\item \label{item-ig-gff2} $\eta_1$ and $\eta_2$ are the space-filling counterflow lines of the same whole-plane GFF with angles 0 and $\theta-\pi/2$ for some $\theta\in(-\pi/2,\pi/2)$ and both have parameter $\kappa >4$ (when $\gamma^2  = 16/\kappa$, the permuton $\perm$ is a skew Brownian permuton).   
	\end{enumerate}
	Almost surely, the closed support of $\perm$ determines the sets $\mcl T \mcl M_1 = \mcl T \mcl M(\eta_1)$ and $\mcl T \mcl M_2 = \mcl T \mcl M(\eta_2)$ defined in~\eqref{eqn-intersect-set}.
\end{thm}

We now explain the strategy outlined in \cref{sect:supp_perm} for proving the above theorem and highlight several partial results that we believe are of independent interest. We start with three important, and simple to prove, facts which will determine what we need to show to establish \cref{thm-permuton-multi-points}.

\medskip

\noindent\underline{\emph{Fact 1}:} As shown in \cref{lem-closure-to-intersect}, 
the intersection set $\mcl T \mcl M_1 = \mcl T \mcl M(\eta_1)$ from~\eqref{eqn-intersect-set} is a.s.\ determined by the set 
\[
\ol{\{ (t,\psi_-(t)) : t\in [0,1]\}}, \quad\text{ with }\quad \psi_-(t) := \inf\left\{ s \in  [0,1]: \eta_2(s) = \eta_1(t) \right\},
\]
that is, $\psi_-(t)$ is the first time that the curve $\eta_2$ hits the point $\eta_1(t)$.

\medskip

\noindent\underline{\emph{Fact 2}:} As shown in the second assertion of \cref{lem-permuton-defined}, for any choice of the function $\psi$ from~\eqref{eqn-psi-property},  
\begin{equation}\label{eq:wewfihbw9euobfwu0}
	\op{supp} \perm \subset \ol{\{(t,\psi(t)) : t\in [0,1]\} }\subset \{(t,s) \in [0,1] : \eta_1(t) = \eta_2(s) \}  .
\end{equation}
But, in general, both inclusions in~\eqref{eq:wewfihbw9euobfwu0} can potentially be strict. For example, if $\eta_1 = \eta_2$ then $\op{supp} \perm$ is the diagonal in $[0,1]^2$. The third set in~\eqref{eq:wewfihbw9euobfwu0} includes off-diagonal points, which arise from pairs of distinct times $(t,s)$ such that $\eta_1(t) = \eta_1(s)$ or $\eta_2(t) = \eta_2(s)$. The middle set in~\eqref{eq:wewfihbw9euobfwu0} can include off-diagonal points or not, depending on the choice of $\psi$.

\medskip

\noindent\underline{\emph{Fact 3}:} As shown in the third assertion in  \cref{lem-permuton-defined}, if $(t,s) \in [0,1]^2$  are such that $\eta_1(t) = \eta_2(s)$ and $\eta_2(s)$ is hit only once by $\eta_2$, then $(t,s) $ belongs to the support of $\perm$; and the same is true with the roles of $\eta_1$ and $\eta_2$ interchanged.

This tells us that multiple points of $\eta_1$ and $\eta_2$ are the \emph{only} potential reason why the inclusions in~\eqref{eq:wewfihbw9euobfwu0} can fail to be equalities.

\medskip

From here, the proof of \cref{thm-permuton-multi-points} diverge according to the two cases in Assumptions~\ref{item-indep2}~and~\ref{item-ig-gff2}.

\medskip

\noindent\underline{\emph{Case in Assumption~\ref{item-indep2}, that is, when $\eta_1$ and $\eta_2$ are independent}:} Combining Facts 1 and 2, one can see that  \cref{thm-permuton-multi-points} will immediately follow if we can show that
\eqb \label{eq:weydfvyuiwe}
\op{supp} \perm \supset \ol{\{(t,\psi_-(t)) : t\in [0,1]\} }. 
\eqe
Morover, thanks to Fact 3, to prove the latter equality,  one needs a fine analysis of the behavior of the two SLEs $(\eta_1,\eta_2)$ around their multiple points. 

Since from Assertion 1 in Lemma~\ref{lem-permuton-defined}, we know that for each rectangle $[a,b]\times[c,d] \subset [0,1]^2$, 
\eqbn 
\perm\left([a,b]\times[c,d]\right) = \mu_h\left( \eta_1([a,b]) \cap \eta_2([c,d]) \right)  ,
\eqen
to show that $(t,\psi_-(t))\in \op{supp} \perm$, it is enough to show that a.s.\ whenever $(a,b)\times(c,d)\subset [0,1]^2$ is a rectangle which contains a point of the form $(t,\psi_-(t))$, then $\mu_h(\eta_1([a,b]) \cap \eta_2([c,d])) > 0$. To obtain this, roughly speaking, it is enough to show that when $\eta_2$ hits $\eta_1(t)$, then $\eta_2$ must immediately fill an open subset contained in $\eta_1([a,b])$. This result will be achieved by proving several lemmas in \cref{sec-permuton-support} regarding how space-filling SLEs hit certain points, thereby  obtaining the desired equality in \eqref{eq:weydfvyuiwe} (\cref{prop-permuton-support}).

An interesting consequence (\cref{prop-permuton-full}) is that if at least one among $\kappa_1$ and $\kappa_2$ is greater or equal than $8$ then, almost surely,
\eqbn
	\op{supp} \perm = \ol{\{(t,\psi_-(t)) : t\in [0,1]\} } = \{(t,s) \in [0,1] : \eta_1(t) = \eta_2(s) \}  .
\eqen

\medskip

\noindent\underline{\emph{Case in Assumption~\ref{item-ig-gff2}, that is, when $\eta_1$ and $\eta_2$ are counterflow lines of the same whole-plane GFF}:} The proof in this case is much more involved, the reason being that the equality in \eqref{eq:weydfvyuiwe} is true if and only if  
\begin{equation}\label{eq:ewivbdfb}
	\kappa \geq 12
	\quad\text{and}\quad 
	\theta \in \left[ - \frac{(\kappa-12)\pi}{2(\kappa-4)} ,  \frac{(\kappa-12)\pi}{2(\kappa-4)}  \right],
\end{equation}
as shown in Assertion 1 of \cref{prop-counterflow-dichotomy}.

When the pair $(\kappa,\theta)$ does not satisfy \eqref{eq:ewivbdfb}, we adopt an \emph{ad hoc} approach for different types of multiple points of $\eta_1$ and $\eta_2$.  When $\kappa\geq 8$ the situation is a bit simpler: Indeed, in this regime, the space-filling SLEs $\eta_1$ and $\eta_2$ only have either double points or merge points (some specific triple points described in \cref{sect:SLES}). 
We will show in \cref{lem:triple_are_simple} that a merge point for $\eta_1$ cannot be hit more than once by $\eta_2$. This property of merge points, which we believe to be of independent interest, will help us when $\kappa\geq 8$, but in the case when $\kappa\in(4,8)$ we will need to perform a very fine analysis of multiple points of SLEs, as done in Sections~\ref{sect:m-tuples}~and~\ref{sect:eight-flow}.

We conclude by highlighting one fact which illustrates what are some of the difficulties in establishing \cref{thm-permuton-multi-points} when the pair $(\kappa,\theta)$ does not satisfy \eqref{eq:ewivbdfb}. Moreover, we found the next property quite interesting (and rather surprising at the same time).

\begin{prop}
	In the setting of \cref{thm-permuton-multi-points}, assume that  $\eta_1$ and $\eta_2$ are the space-filling counterflow lines of the same whole-plane GFF with angles 0 and $\theta-\pi/2$ for some $\theta \in (-\pi/2,\pi/2)$.
	Let $z$ be a multiple point both for $\eta_1$ and $\eta_2$; this is possible when 
	\[\kappa\in(4,12)\qquad\text{or}\qquad\kappa\geq 12 \text{ and }\theta \notin \left[ - \frac{(\kappa-12)\pi}{2(\kappa-4)} ,  \frac{(\kappa-12)\pi}{2(\kappa-4)} \right].\]
	Let  $t^1_1,t^1_2,\dots,t^1_{m_1}$ be the $m_1\geq 2$ times when $\eta_1$ hits $z$ and $t^2_1,t^2_2,\dots,t^2_{m_2}$ be the $m_2\geq 2$ times when $\eta_2$ hits $z$. Then
	\begin{itemize}
		\item almost surely, there exists a pair $(i,j)$ such that the point $(t^1_i,t^1_j)$ of the unit square $[0,1]^2$ is \emph{not} included in the support of the permuton $\perm$;
		\item almost surely, the support of the permuton $\perm$   contains enough points among $\{(t^1_i,t^1_j)\}_{i,j}$ so that it is always possible to establish that 
		\begin{equation*}
			\eta_1(t^1_i)=\eta_2(t^2_j),\quad\text{for all}\quad i\leq m_1 \text{ and } j\leq m_2.
		\end{equation*}
	\end{itemize}  
\end{prop}

We refer the curious reader to Propositions~ \ref{prop-counterflow-dichotomy}~and~\ref{prop-counterflow-dichotomy2} and \cref{fig-spiral-perm} for more precise (but technical) statements of the result above. See also \cref{sect:explicit-function}.

\begin{ack}
	We thank Scott Sheffield for his help in understanding the boundary values of imaginary geometry flow lines, and Morris Ang, Jason Miller and Xin Sun for helpful discussions. J.B.\ was partially supported by NSF grant DMS-2441646. E.G. was partially supported by a Clay research fellowship and by NSF grant DMS-2245832.
\end{ack}

\section{Space-filling whole-plane SLEs and LQG} 
\label{sec-sle}

This section is devoted to review several definitions and preliminary results about space-filling whole-plane SLEs and LQG.
	
\subsection{Basic definitions and facts about space-filling whole-plane SLEs}\label{sect:SLES}

Let $\kappa > 4$. The \textbf{whole-plane space-filling SLE$_\kappa$} from $\infty$ to $\infty$ is a random non-self-crossing space-filling curve $\eta$ in the Riemann sphere $\BB C$ which starts and ends at $\infty$. We view $\eta$ as a curve defined modulo monotone increasing time re-parametrization. We  record some of its qualitative properties which follow from the construction in~\cite{ig4} and will be often used in this paper. 
\begin{enumerate}[$(i)$]
	\item When $\kappa\geq 8$, the curve $\eta$ is a two-sided version of chordal SLE$_\kappa$. It can be obtained from chordal SLE$_\kappa$ by stopping the chordal SLE curve at the first time when it hits some fixed interior point $z$, then ``zooming in'' near $z$. In particular, for any two times $s < t$ the region $\eta([s,t])$ is simply connected.
	\item When $\kappa\in (4,8)$, the curve $\eta$ can be obtained from a two-sided variant of chordal SLE$_\kappa$ by iteratively ``filling in'' the bubbles which it disconnects from its target point. In this case, for typical times $s < t$, the region $\eta([s,t])$ is not simply connected. 
	
	\item For each fixed $z\in\BB C$, the left and right outer boundaries of $\eta$ stopped when it hits $z$ are a pair of coupled whole-plane SLE$_{16/\kappa}(2-16/\kappa)$ curves from $z$ to $\infty$ (whole-plane SLE$_{16/\kappa}(\rho)$ for $\rho > -2$ is a variant of whole-plane SLE introduced in~\cite[Section 2.1]{ig4}). In particular, they are two interior flow lines of a whole-plane GFF whose angles differ by $\pi$. These flow lines are disjoint for $\kappa \geq 8$, a.s.\ intersect each other for $\kappa \in (4, 8)$, and are a.s.\ self-intersecting for $ \kappa \in (4, 6)$; see \cref{fig-flow-line-def}. 
		
	In \cref{sec:SLE_as_flow}, we will more precisely recall the construction of whole-plane space-filling SLE$_\kappa$ from the flow lines of a whole-plane GFF, i.e.\ as a space-filling SLE$_\kappa$ counterflow line of a whole-plane GFF. (Readers who are less familiar with the topic might find it helpful to jump ahead to that section now and return to this part later for a more comprehensive understanding.)
	
	\item For any times $a < b$, the set $\eta([a,b])$ has non-empty interior. Moreover, if $\mu$ is a locally finite, non-atomic measure on $\BB C$ which assigns positive mass to every open set, then $\eta$ is continuous if we parametrize $\eta$ so that 
	\begin{equation}\label{eq:time-parm-SLE}
		\mu(\eta([a,b])) = b-a \quad\text{for each}\quad a  <b.
	\end{equation}
	
	\item The law of $\eta$ is invariant under scaling, translation, rotation, and time reversal, i.e., $a\eta + b$ and $\eta( - \cdot)$ have the same law as $\eta$ modulo time parametrization for each $a \in\BB C \setminus \{0\}$ and each $b\in\BB C$. 
	\item For each fixed $z\in\BB C$, a.s.\ $\eta$ hits $z$ exactly once. The maximum number of times that $\eta$ hits any $z\in\BB C$ is 3 if $\kappa\geq 8$ and is a finite, $\kappa$-dependent number if $\kappa \in (4,8)$~\cite[Theorem 6.3]{ghm-kpz}. Furthermore, it is a.s.\ the case that any point $z$ that $\eta$ hits more than once lies on a flow line or dual flow line started at a different rational point~\cite[Lemma 8.11]{wedges}.
	\item  A \textbf{merge point} of $\eta$, seen as a space-filling SLE$_\kappa$ counterflow line of a whole-plane GFF of angle zero, is a point $z\in \BB C$ such that, for some two distinct points $x, y \in \BB C$ with rational
	coordinates, the $\pi/2$-angle flow lines $\beta^L_x$ and $\beta^L_y$ started at $x$ and $y$ merge at $z$.
	Almost surely, all merge points are triple points.
	When $\kappa\geq 8$, a.s., every triple point is a merge point and the set of merge points of $\eta$ is countable. When $\kappa\in (4,8)$, a.s., the set of merge points of $\eta$ is still countable, but a.s., there are (uncountably many) triple points which are not merge points ~\cite[Lemma 8.12/Eq.\ (8.9)]{wedges}. These latter points are 3-tuple points, introduced in the next item.
	\item  Fix $m\geq 2$. A \textbf{$m$-tuple point} of $\eta$, seen as a space-filling SLE$_\kappa$ counterflow line of a whole-plane GFF of angle zero, is a point $z \in \BB C$ that is not a merge
	point, such that for some rational points $x_1, x_2, \dots , x_{m-1}\in\BB C$ (and a choice of flow line/dual flow line for
	each $i \in \{1, 2, \dots, {m-1}\}$) the ${m-1}$ flow/dual flow lines started from the $x_i$ all reach the point $z$ without
	merging with each other first.   
	Moreover, a.s., every multiple point of $\eta$ is either a merge point or a $m$-tuple point of some $m \geq 2$ in which case it is hit by $\eta$ exactly $m$ times~\cite[Lemma 8.13]{wedges}.
	Almost surely, if $\kappa \in (4,8)$ and $M\in\BB N_{\geq3}$ is such that
	\[ 
	\kappa \in \left[\frac{4M}{M-1} , \frac{4(M-1)}{M-2}\right),
	\] 
	then $\eta$ has, for all $2\leq m\leq M$, uncountably many $m$-tuple points, and no $m$-tuple point of all $m>M$~\cite[Eq.\ (8.9)]{wedges}. 
	For instance, this implies that a.s.\ $\eta$ has only $2$-tuple and $3$-tuple points when $\kappa\in [6,8)$, or $\eta$ has only $2$-tuple, $3$-tuple and $4$-tuple points when $\kappa\in [16/3,6)$. A more detailed description of $m$-tuple points will be given in \cref{lem:m-tuple-points}.
\end{enumerate}

\subsection{Space-filling SLEs from the flow lines of a whole-plane Gaussian free field}\label{sec:SLE_as_flow}

We give here a brief description of the construction of whole-plane space-filling SLE$_\kappa$ curve $\eta$ from the flow lines of a whole-plane Gaussian free field (GFF), following~\cite[Section 1.2.3]{ig4} (see also~\cite[Section 1.4.1]{wedges}). The fundamental idea of the definition is based on the so-called SLE \emph{duality}, which states that when $\kappa > 4$ the outer boundary of an SLE$_\kappa$ curve stopped at a given time should consist of one or more SLE$_{16/\kappa}$-type curves~\cite{zhan-duality1,zhan-duality2,dubedat-duality,ig1,ig4}. Let
\begin{equation*}
	\chi:=\frac{\sqrt\kappa}{2}-\frac{2}{\sqrt\kappa}.
\end{equation*}
Let $\wh h$ be a whole-plane GFF viewed modulo additive multiples of $2\pi\chi$, as in~\cite{ig4}. 
Also, fix an angle $\theta\in(-\pi/2,\pi/2)$. We invite the reader to compare the next explanations with \cref{fig-flow-line-def}.  For $z\in\BB Q^2$, let $\beta_z^L$ and $\beta_z^R$ be the flow lines of $e^{i\wh h/\chi}$ started from $z$, in the sense of~\cite[Theorem 1.1]{ig4}, with angles\footnote{We measure angles in counter-clockwise order, where zero angle corresponds to the north direction.} $\theta$ and $\theta-\pi$, respectively.   Each of these flow lines is a whole-plane SLE$_{16/\kappa}(2-16/\kappa )$ curve from $z$ to $\infty$. Importantly, the GFF $\wh h$ establishes a coupling between these curves. Moreover, these flow lines will be the left and right outer boundaries of the space-filling SLE$_\kappa$ $\eta$ stopped when it hits the point $z$.

\begin{figure}[ht!]
	\begin{center}
		\includegraphics[width=.99\textwidth]{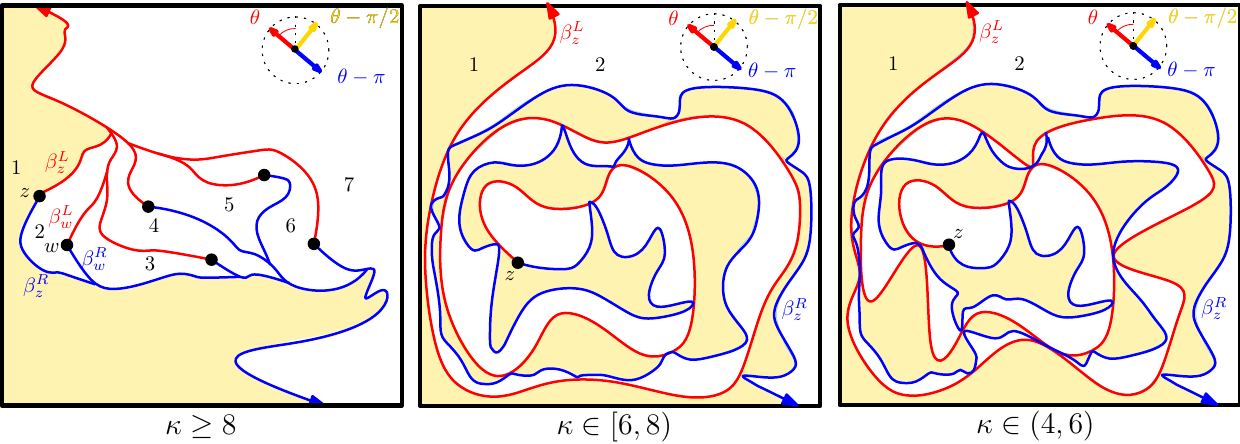}  
		\caption{\label{fig-flow-line-def} 
			\textbf{Left:} The squared box is a portion of the complex plane $\mathbb C$. We fix $\theta$ as shown in the picture. We plot in red the flow lines $\beta_{x}^{L}$  and in blue the flow lines $\beta_{x}^{R}$ for six points $x\in\mathbb C$. For every $x\in\mathbb C$, the flow lines $\beta_{x}^{L}$ and $\beta_{x}^{R}$ are the left and right outer boundaries of the space-filling SLE$_{\kappa}$ counterflow line $\eta$ of angle $\theta-\pi/2$ stopped when it hits $x$. The curve $\eta$ fills in the regions 1 (in yellow), 2, 3, 4, 5, 6 and 7 in this order.  The left figure illustrates the case when $\kappa\geq 8$. \textbf{Middle:} The same illustration as in the left-hand side but when $\kappa\in [6,8)$. In this case we just considered a single point $z\in\mathbb C$ to simplify the illustration. The flow lines $\beta_{z}^{L}$ (in red) and $\beta_{z}^{R}$ (in blue) started from the same point $z$ can hit each other and bounce off (but a.s.\ are non-self-intersecting). The the space-filling SLE$_{\kappa}$ counterflow line $\eta$ of angle $\theta-\pi/2$  first fills in the bubbles of the region 1 (in yellow) and then the bubbles of the region 2 (in white).
				\textbf{Right:} The same illustration as in the middle when $\kappa\in (4,6)$. The flow lines $\beta_{z}^{L}$ (in red) and $\beta_{z}^{R}$ (in blue) started from the same point $z$ can hit each other and bounce off and are also a.s.\ self-intersecting.}
	\end{center}
	\vspace{-3ex}
\end{figure} 

By~\cite[Theorem 1.9]{ig4} (see also \cref{lem:flow_lines_crossing} below), for distinct $z,w\in\BB Q^2$, the flow lines $\beta_z^L$ and $\beta_z^R$ cannot cross $\beta_w^L \cup \beta_w^R$, but they may hit and bounce off if $\kappa \in (4,8)$. Furthermore, a.s.\ the flow lines $\beta_z^L$ and $\beta_w^L$ (resp.\ $\beta_z^R$ and $\beta_w^R$) eventually merge into each other.

We define a total ordering on $\BB Q^2$ by declaring that $z$ comes before $w$ if and only if $w$ lies in a connected component of $\BB C\setminus (\beta_z^L\cup \beta_z^R)$ which lies to the right of $\beta_z^L$ and to the left of $\beta_z^R$. It follows from~\cite[Theorem 1.16]{ig4} (see also~\cite[Footnote 4]{wedges}) that there is a unique space-filling path $\eta$ from $\infty$ to $\infty$ which hits the points of $\BB Q^2$ in the prescribed order and is continuous when it is time-parametrized,e.g., as in \eqref{eq:time-parm-SLE}.  This curve $\eta$ is defined to be the \textbf{space-filling SLE$_\kappa$ counterflow line of $\wh h$ from $\infty$ to $\infty$ of angle\footnote{In this paper we are using the convention that counterflow line of angle $\theta$ have the \emph{same} direction has flow line of angle $\theta$. We point out that in other papers (for instance, \cite{ig1}) the authors might use the opposite convention, that is, counterflow line of angle $\theta$ have the \emph{opposite} direction as flow line of angle $\theta$.} $\theta-\pi/2$}.

The law of the curve $\eta$ does not depend on $\theta$, but the coupling of $\eta$ with $\wh h$ depends on $\theta$. In particular, if $\theta \in [-\pi/2,\pi/2]$ and $\eta_1$ and $\eta_2$ are the space-filling SLE counterflow lines of $\wh h$ with angles 0 and $\theta-\pi/2$, respectively, then $\eta_1$ and $\eta_2$ are coupled together in a non-trivial way. By~\cite[Theorem 1.16]{ig4}, each of $\eta_1$ and $\eta_2$ is a.s.\ given by a measurable function of the other.

The latter is exactly the coupling used in the SLEs and LQG description of the skew Brownian permutons.

\subsection{Liouville quantum gravity} \label{sec-intro-lqg}

In this subsection we recall some key definitions related to the theory of Liouville quantum gravity (LQG). We refer to~\cite{gwynne-ams-survey,sheffield-icm} for brief introduction to LQG and to~\cite{bp-lqg-notes} for a detailed exposition. 
Let $\gamma \in (0,2)$ and let
\begin{equation*}\label{eqn-Q}
	Q := \frac{2}{\gamma} + \frac{\gamma}{2} .
\end{equation*}

\begin{defn} \label{def-lqg-surface}
	A \textbf{$\gamma$-Liouville quantum gravity (LQG) surface with $k\in\BB N_0$ marked points} is an equivalence class of $k+2$-tuples $(U,h,z_1,\dots,z_k)$, where 
	\begin{itemize}
		\item $U\subset\BB C$ is open and $z_1,\dots,z_k \in U\cup \bdy U$, with $\bdy U$ viewed as a collection of prime ends;
		\item $h$ is a generalized function on $U$ (which we will always take to be some variant of the GFF);
		\item $(U,h,z_1,\dots,z_k)$ and $(\wt U, \wt h,\wt z_1,\dots,\wt z_k)$ are declared to be equivalent if there is a conformal map $\phi : \wt U \rta U$ such that
		\eqb \label{eqn-lqg-coord}
		\wt h = h\circ \phi + Q\log|\phi'| \quad \text{and} \quad \phi(\wt z_j)  = z_j ,\quad \forall j =1,\dots k ,
		\eqe
		where $Q$ is as in~\eqref{eqn-Q}.
	\end{itemize}
	If $(U,h,z_1,\dots,z_k)$ is an equivalence class representative, we refer to $h$ as an \textbf{embedding} of the quantum surface into $(U,z_1,\dots,z_k)$. 
\end{defn}

All of the LQG surfaces we consider will be random, and the generalized function $h$ will be a \textbf{GFF plus a continuous function}, meaning that there is a coupling $(h , h^U)$ of $h$ with the whole-plane or zero-boundary GFF $h^U$ on $U$ (as appropriate) such that a.s.\ $h-h^U$ is a function which is continuous on $U$ except possibly at finitely many points.

If $h$ is a whole-plane GFF plus a continuous function, one can define the \textbf{LQG area measure} $\mu_h$ on $U$, which is a limit of regularized versions of $e^{\gamma h} \,d^2 z$, where $d^2z$ denotes Lebesgue measure on $U$~\cite{kahane,shef-kpz,rhodes-vargas-log-kpz}. Almost surely, the measure $\mu_h$ assigns positive mass to every open set and zero mass to every point, but is mutually singular with respect to Lebesgue measure.

In this paper we will be interested in the canonical LQG surface with the topology of the sphere, i.e., the unit area quantum sphere. (There are multiple equivalent ways of defining the unit area quantum sphere; the definition we give here is~\cite[Definition 2.2]{ahs-sphere} with $k=3$ and $\alpha_1=\alpha_2=\alpha_3=\gamma$.)

\begin{defn}\label{def-sphere} 
	We define the \textbf{Liouville field}
	\begin{equation}\label{eqn-sphere-gaussian}
		h_{\op{L}} := h^{\BB C} +  \gamma G_{\BB C}(1 ,\cdot) + \gamma G_{\BB C}(0,\cdot)  - (2Q-\gamma) \log|\cdot|_+,
	\end{equation}
	where
	\begin{equation*}
		  \quad G_{\BB C}(z,w) := \log \frac{|z|_+  |w|_+ }{|z-w| } \quad \text{and} \quad |z|_+ := \max\{|z| , 1\}.
	\end{equation*}
	We let $\mu_{h_{\op{L}}}$ be its associated $\gamma$-LQG measure. 
	
	\begin{itemize} 
		\item The \textbf{triply marked unit area quantum sphere} is the LQG surface $(\BB C , h , \infty, 0,1 )$, where the law $\BB P_h$ of $h$ is obtained from the law  $\BB P_{\tilde h_{\op{L}}}$ of
		\begin{equation*}\label{eqn-sphere-subtract}
			\tilde h_{\op{L}}:=h_{\op{L}} - \frac{1}{\gamma} \log   \mu_{h_{\op{L}}}(\BB C)
		\end{equation*}
		weighted by a $\gamma$-dependent constant times $\left[\mu_{h_{\op{L}}}(\BB C)\right]^{4/\gamma^2 - 2}$, i.e.\ $\frac{d\BB P_h}{d \BB P_{\tilde h_{\op{L}}}}=C(\gamma)\cdot\left[\mu_{h_{\op{L}}}(\BB C)\right]^{4/\gamma^2 - 2}$. 
		\item The \textbf{doubly (resp.\ singly) marked unit area quantum sphere} is the doubly (resp.\ singly) marked quantum surface obtained from the triply marked unit area quantum sphere by forgetting the marked point at 1 (resp.\ the marked points at 1 and 0). 
	\end{itemize}
\end{defn} 

We note that subtracting $\frac{1}{\gamma} \log \mu_{h_{\op{L}}}(\BB C)$ in~\eqref{eqn-sphere-subtract} makes it so that the unit area quantum sphere a.s.\ satisfies $\mu_h(\BB C) = 1$. The triply marked quantum sphere has only one possible embedding in $(\BB C , \infty,0,1)$ (Definition~\ref{def-lqg-surface}). This is because the only conformal automorphism of the Riemann sphere which fixes $\infty,0,1$ is the identity. By contrast, the singly (resp.\ doubly) marked quantum sphere has multiple possible embeddings, obtained by applying~\eqref{eqn-lqg-coord} when $\phi$ is a complex affine transformation (resp.\ a rotation and scaling). Any statements for a singly or doubly marked quantum sphere are assumed to be independent of the choice of embedding, unless otherwise specified. Many of the results in this paper are stated for a singly marked quantum sphere since we only need one marked point (i.e., the starting point of the SLE curves). However,considering the triply marked quantum sphere makes possible to have an exact description of the Liouville field~\ref{def-sphere}.

We conclude this section with an additional definition. In this paper, we will often consider quantum surfaces that are decorated by two space-filling curves:

\begin{defn} \label{def-curve-decorated}
	Let $\gamma\in (0,2)$. A \textbf{$\gamma$-LQG surface with a single marked point, decorated by two curves} is an equivalence class of 5-tuples $(U,h,z, \eta_1,\eta_2)$, where 
	\begin{itemize}
		\item $U\subset\BB C$ is open and $z \in U\cup \bdy U$, with $\bdy U$ viewed as a collection of prime ends;
		\item $h$ is a generalized function on $U$ (which we will always take to be some variant of the GFF);
		\item $\eta_1 : [a_1,b_1] \rta U \cup \bdy U$ and $\eta_2 : [a_2,b_2] \rta U\cup \bdy U$ are curves in $U \cup \bdy U$;
		\item $(U,h,z,\eta_1,\eta_2)$ and $(\wt U, \wt h,\wt z, \wt\eta_1 ,\wt\eta_2)$ are declared to be equivalent if there is a conformal map $\phi : \wt U \rta U$ such that
		\eqbn
		\wt h = h\circ \phi + Q\log|\phi'| ,\quad \phi(\wt z )  = z  ,\quad \phi\circ\wt\eta_1=\eta_1 \quad \text{and} \quad \phi\circ\wt\eta_2=\eta_2
		\eqen
		where $Q$ is as in~\eqref{eqn-Q}.
	\end{itemize} 
\end{defn}

Quite often we will refer to a $\gamma$-LQG surface with a single marked point, decorated by two curves, simply as curve-decorated quantum surface.

Furthermore, a curve-decorated quantum surface $(U , h ,  \infty , \eta_1, \eta_2)_{\mathrm{cc}}$ \textbf{up to complex conjugation} is the same equivalence class above, where the conformal map $\phi : \wt U \rta U$ is replaced by a conformal map composed with a complex conjugation.

\section{Preliminary lemmas about space-filling whole-plane SLEs and flow lines}\label{sect:sle-flow}

We gather here a series of technical lemmas that will be used later in the paper. We invite the reader to skip at a first read the technical proof of \cref{lem:triple_are_simple}, since it is quite complicated and does not use any technique that will be used later in the paper.

\medskip

The first result we need is the following simple lemma explaining how the double points of a whole-plane space-filling SLE$_\kappa$ are hit. See the left-hand side of \cref{fig-hitting-double-points} for an illustration of the next lemma.

\begin{lem}\label{lem:SLE_hitting double points}
	Let $\kappa\geq 8$ and $\eta$ be a whole-plane space-filling SLE$_\kappa$ from $\infty$ to $\infty$. Almost surely, the following assertion is true. Fix a double point $z$ of $\eta$ (i.e., a point which is hit exactly twice) and let $w\in\BB C$ such that $w\neq z$ and $z\in\beta_{w}^L$. Let $\tau_z<\tau_w$ be the first hitting times of $z$ and $w$, respectively. Fix also $\ep>0$ so that $\tau_z+\ep<\tau_w$. Then there exists $\delta>0$ such that
	\[
	B_\delta(z) \cap \eta([0,\tau_w])\subset \eta([\tau_z-\ep,\tau_z+\ep]).
	\]
\end{lem}

\begin{proof}
	Note that since $z$ is a double point of $\eta$, $z$ can be hit by $\eta$ only once before time $\tau_w$ (then $\eta$ will hit $z$ a second time after time $\tau_w$).
	Suppose by way of contradiction that for every $\delta > 0$, $\eta[\tau_z-\ep , \tau_z+\ep]$ does not contain $B_\delta(z) \cap \eta([0,\tau_w])$. See the right-hand side of \cref{fig-hitting-double-points} for an illustration.  Then $z$ must be on the boundary of $\eta[0,\tau_z+\ep]$, so $\eta$ (which is space-filling) will have to hit $z$ again before time $\tau_w$, since we assume that $\tau_z+\ep<\tau_w$. This contradicts our assumption that $z$ is a double point of $\eta$.
\end{proof}

\begin{figure}[ht!]
	\begin{center}
		\includegraphics[width=0.95\textwidth]{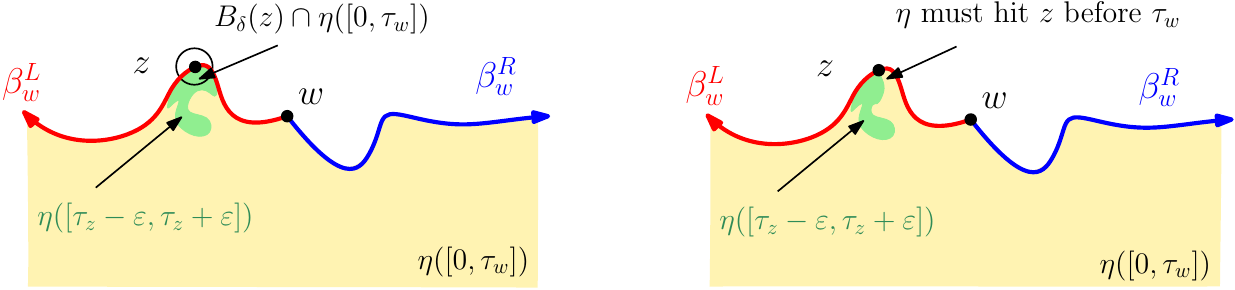}  
		\caption{\label{fig-hitting-double-points} \textbf{Left:} An illustration of the statement of \cref{lem:SLE_hitting double points} which shows that $B_\delta(z) \cap \eta([0,\tau_w])\subset \eta([\tau_z-\ep,\tau_z+\ep])$. \textbf{Right:} If $B_\delta(z) \cap \eta([0,\tau_w])\nsubseteq \eta([\tau_z-\ep,\tau_z+\ep])$ for all $\delta>0$, then $\eta$ must hit $z$ twice before time $\tau_w$, but this is not possible since $z$ is a double point and so it can only be hit once before time $\tau_w$. 
		}
	\end{center}
	\vspace{-3ex}
\end{figure}

We also need to recall how the flow lines of the same whole-plane GFF interact.
Recall from \cref{sec:SLE_as_flow} that $\wh h$ is the whole-plane GFF viewed modulo a global additive multiple of $2\pi\chi$ which is used to construct two whole-plane space-filling SLE$_\kappa$ of angle $0$ and $\theta$, denoted by $\eta_1$ and $\eta_2$, respectively (recall Assumption~\ref{item-ig-gff} in \cref{thm-permuton-determ-lqg-sles}).  
For $i \in \{1,2\}$ and $z\in\BB C$, let $\beta_{z,i}^L$ and $\beta_{z,i}^R$ be the left and right outer boundaries of $\eta_i$ stopped when it hits $z$. 
Equivalently, $\beta_{z,1}^L$ and $\beta_{z,1}^R$ are the flow lines of $\wh h$ of angles $\pi/2$ and $-\pi/2$ starting from $z$; and $\beta_{z,1}^L$ and $\beta_{z,2}^R$ are the flow lines of $\wh h$ of angles $\theta$ and $\theta-\pi$ started from $z$. 

\begin{lem}[{\cite[Theorem 1.7]{ig4}}]\label{lem:flow_lines_crossing}
	Let $u,v\in\BB C$ and $\theta\in [-\pi/2,\pi/2)$. Consider the flow lines $\beta_{u,1}^L$ and $\beta_{v,2}^L$ (of angle $\pi/2$ and $\theta$, respectively). We characterize when these two flow lines cross each other, bounces off each other, or do not intersect each other (\emph{c.f.}, \cref{fig-interaction-flows}):
	\begin{itemize}
		\item If $u=v$ then $\beta_{v,2}^L$ never crosses $\beta_{u,1}^L$, but $\beta_{v,2}^L$ bounces off $\beta_{u,1}^L$ if and only if
		\[\theta\in\left(\frac{(\kappa-12)\pi}{2(\kappa-4)},\frac{\pi}{2}\right).\]
		In particular, if $\theta\in\left[-\frac{\pi}{2},\frac{(\kappa-12)\pi}{2(\kappa-4)}\right]$ then $\beta_{v,2}^L\cap\beta_{u,1}^L=\{u\}$. 
		
		\item If $u \neq v$ and $\theta\in\left(\frac{(\kappa-12)\pi}{2(\kappa-4)},\frac{\pi}{2}\right),$ then it is possible that $\beta_{v,2}^L$ hits  $\beta_{u,1}^L$  on the right-hand side of $\beta_{u,1}^L$. If this is the case, the curve $\beta_{v,2}^L$ bounces off $\beta_{u,1}^L$.
  
		\item If $u \neq v$ and $\theta\in\left[-\frac{\pi}{2},\frac{(\kappa-12)\pi}{2(\kappa-4)}\right]$, then it is not possible  that $\beta_{v,2}^L$ hits  $\beta_{u,1}^L$  on the right-hand side of $\beta_{u,1}^L$.
		
		\item If $u \neq v$ and $\beta_{v,2}^L$ hits  $\beta_{u,1}^L$  on the left-hand side of $\beta_{u,1}^L$, then the curve $\beta_{v,2}^L$ crosses  $\beta_{u,1}^L$ upon intersecting but does not subsequently
		cross back. After crossing, if $\theta\in\left(\frac{(\kappa-12)\pi}{2(\kappa-4)},\frac{\pi}{2}\right)$, it is possible that $\beta_{v,2}^L$ hits $\beta_{u,1}^L$ again on the right-hand side of $\beta_{u,1}^L$. If this is the case, the curve $\beta_{v,2}^L$ bounces off $\beta_{u,1}^L$.
	\end{itemize}
\end{lem}

\begin{figure}[ht!]
	\begin{center}
		\includegraphics[width=0.95\textwidth]{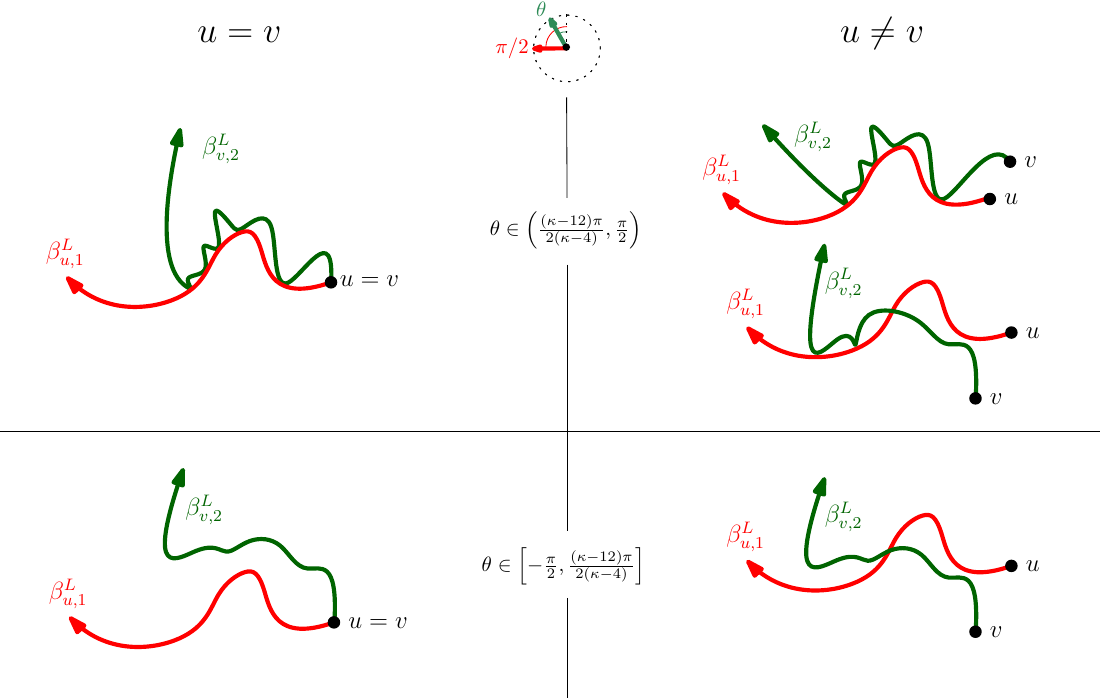}  
		\caption{\label{fig-interaction-flows} A schema for \cref{lem:flow_lines_crossing} to explain when the two flow lines $\beta_{u,1}^L$ and $\beta_{v,2}^L$ cross each other, bounces off each other, or do not intersect each other.
		}
	\end{center}
	\vspace{-3ex}
\end{figure}

The next result describes a useful property of the merge points. 

\begin{lem}\label{lem:triple_are_simple}
	Fix $\kappa>4$ and $\theta\in[-\pi/2,\pi/2)$.
	Let $\wh h$ be a whole-plane GFF, viewed modulo additive multiples of $2\pi\chi$. 
	Let $\eta_1$ and $\eta_2$ be its space-filling SLE$_\kappa$ counterflow lines from $\infty$ to $\infty$ with angles 0 and $\theta-\pi/2$, respectively.
	Then, almost surely, a merge point for $\eta_1$ cannot be hit more than once by $\eta_2$. 
\end{lem}

The rest of the section is devoted to the proof of \cref{lem:triple_are_simple} -- which should be skipped at a first read. The proof builds on the following result, whose proof is postponed to the end of the section.

\begin{lem}\label{lem:not-hitting}
	Fix $\kappa\in(0,4)$ and $({\rho}^L,{\rho}^R)\in\BB{R}^n\times \BB{R}^m$. Let $\eta$ be an SLE$_{\kappa}({\rho}^L,{\rho}^R)$ curve in the upper half-plane $\BB H$ from $0$ to $\infty$ and with force points at $p_n^L<\dots<p_1^L$ on $\BB{R}_{< 0}$ and $p_1^R<\dots<p_m^R$ on $\BB{R}_{> 0}$ such that $p_i^*$ has weight $\rho^*_i$ with $*\in\{L,R\}$. Assume that for some $k\in[m]$, at least one of the following two assumptions hold:
		\begin{enumerate}
			\item  $\sum_{j=1}^{k-1}\rho^R_j>-2$ and $\,\sum_{j=1}^{k}\rho^R_j>\frac{\kappa}{2}-4$; or,
			\item  $\rho^R_k=0$.
	\end{enumerate}
	Then almost surely $\eta$ does not hit $p^R_k$.
\end{lem}

\begin{proof}[Proof of \cref{lem:triple_are_simple}]
	Fix three distinct points $z,w,u\in\BB C$. Consider the flow lines $\beta_{z,1}^L , \beta_{w,1}^L , \beta_{u,2}^L$. Let $x$ be the point where $\beta_{z,1}^L$ and $\beta_{w,1}^L$ merge. Every merge point of $\eta_1$ is a point $x$ of this type for some rational $z,w\in\BB C$ (or a point of this type with $R$ instead of $L$). Every multiple point of $\eta_2$ lies on the flow line $\beta_{u,2}^L$ or $\beta_{u,2}^R$ for some rational $u\in\BB C$. Therefore, recalling that there are countably many merge points of $\eta_1$, it suffices to show the following claim.
	
	\medskip
	
	\noindent\textbf{Claim 1.} Fix three distinct points $z,w,u\in\BB C$. Let $x$ be the point where $\beta_{z,1}^L$ and $\beta_{w,1}^L$ merge. Then a.s.\ $\beta_{u,2}^L$ does not hit $x$.

	\medskip
	
	To prove the claim, we consider the three flow lines $\beta_{u,1}^L$, $\beta_{z,1}^L$ and $\beta_{w,1}^L$. The boundary of $\beta_{z,1}^L\cup\beta_{w,1}^L$ is naturally divided into six intervals by the points $x$, $z$ and $w$. Let $y$ be the point where $\beta_{u,1}^L$ merges into $\beta_{z,1}^L\cup\beta_{w,1}^L$. We assume that $y$ is on the interval shown on the left-hand side of Figure~\ref{fig-boundary-conditions}. The other five cases can be treated similarly to this one, so we will skip them for the sake of conciseness. 
	
	We introduce the parameters, 
	\begin{equation}\label{eq:parameters}
		\lambda=\frac{\pi\sqrt{\kappa}}{4},\quad \lambda'=\frac{\pi}{\sqrt{\kappa}}\quad\text{and}\quad \chi=\frac{\sqrt{\kappa}}{2}-\frac{2}{\sqrt{\kappa}}.
	\end{equation}
	In Figure~\ref{fig-boundary-conditions} we show in red (resp.\ green) the boundary values of the field $\hat{h}$ along the boundaries of $\beta_{u,1}^L$, $\beta_{z,1}^L$ and $\beta_{w,1}^L$ (resp. $\beta_{u,2}^L$), using the standard convention introduced in \cite[Figures 1.9-10-11]{ig1}. See \cite[Section 3.2]{ig4} and \cite[Lemma 3.11]{ss-dgff} for further explanations on why these are the correct values; we simply mention that the general principle is that conditioning on additional flow lines does not change the boundary data along the previous flow lines we have already conditioned on.
	
	\begin{figure}[ht!]
		\begin{center}
			\includegraphics[width=\textwidth]{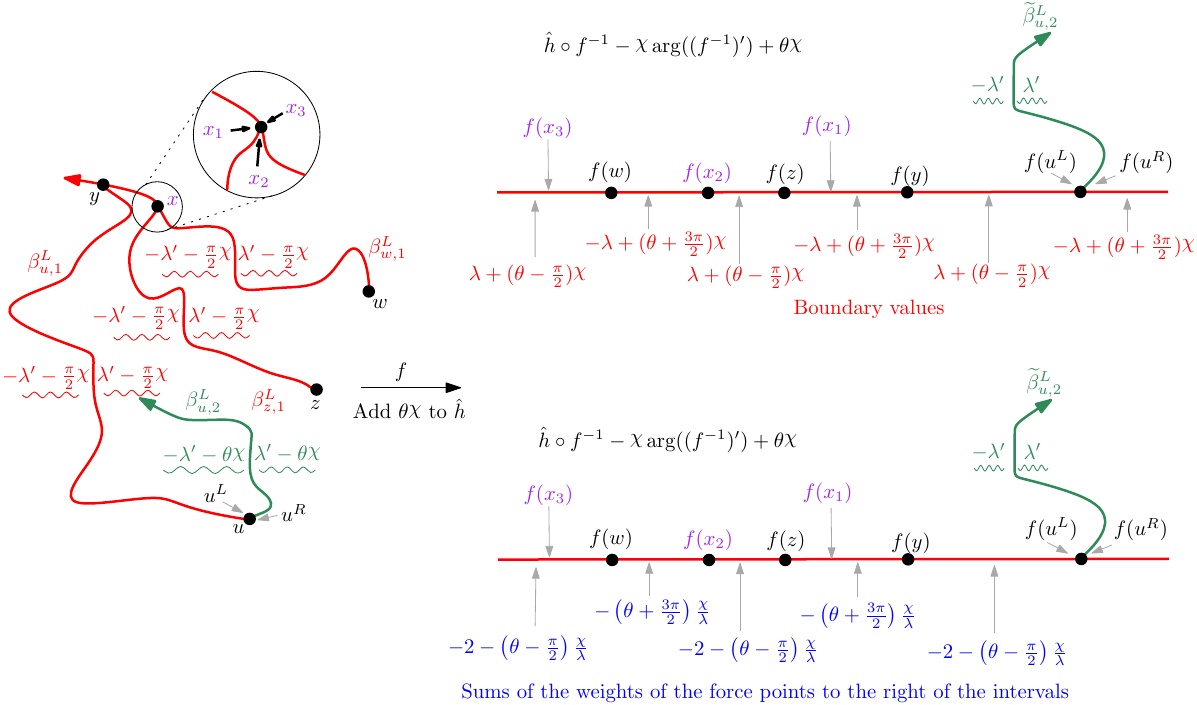}  
			\caption{\label{fig-boundary-conditions} A diagram for the proof of \cref{lem:triple_are_simple} when $\kappa\geq 6$. The goal is to show that the green flow line $\beta^L_{u,2}$ does not hit $x$. Note that $x$ has three corners, that we labeled $x_1$, $x_2$, $x_3$. In the proof we will show that none of these three corners can be hit by $\beta^L_{u,2}$. From left to right, we are considering a conformal map $f$ that sends $\mathbb{C}\setminus\{\beta_{u,1}^L\cup\beta_{z,1}^L\cup\beta_{w,1}^L\}$ to the upper half-plane $\mathbb{H}$, mapping $u$ to $0$ and $\infty$ to $\infty$. On the top-right illustration, we are showing the new boundary values of the imaginary geometry GFF $\hat{h}\circ f^{-1}-\chi\arg((f^{-1})')+\theta\chi$, while on the bottom-right illustration, we are showing the sums of the weights of the force points that determine the boundary values in each interval, as explained in \cite[Theorem 1.1]{ig4}.
			}
		\end{center}
		\vspace{-3ex}
	\end{figure}
	
	We note that $\mathbb{C}\setminus\{\beta_{u,1}^L\cup\beta_{z,1}^L\cup\beta_{w,1}^L\}$ is connected only if $\kappa\geq 6$. Indeed, when $\kappa < 6$, then the flow lines can have self-intersections. In the rest of the proof we assume that $\kappa\geq 6$ so that $\mathbb{C}\setminus\{\beta_{u,1}^L\cup\beta_{z,1}^L\cup\beta_{w,1}^L\}$ is connected, and we will point out the necessary modifications for the case when $\kappa< 6$ at the end of the proof.

	Now let $f$ be a conformal map that sends $\mathbb{C}\setminus\{\beta_{u,1}^L\cup\beta_{z,1}^L\cup\beta_{w,1}^L\}$ to the upper half-plane $\mathbb{H}$, mapping $u$ to $0$ and $\infty$ to $\infty$, as shown in the right-hand side of Figure~\ref{fig-boundary-conditions}. By adding the constant $\theta\chi$ to the field $\hat{h}$ and then applying the function $f$, we get that the boundary values of $\hat{h}\circ f^{-1}-\chi\arg((f^{-1})')+\theta\chi$ are as shown on the right-hand side of Figure~\ref{fig-boundary-conditions}. Let $\widetilde{\beta}_{u,2}^L$ be the image of the flow line $\beta_{u,2}^L$ under the above transformation.
	
	We performed this transformation because in this way we are now in the setting of \cite[Theorem 1.1]{ig4} and so, we can deduce that $\widetilde{\beta}_{u,2}^L$ is an SLE$_{16/\kappa}({\rho}^L,{\rho}^R)$ with 
		\begin{align*}
			{\rho}^L&=\bigg(-2-\left(\theta-\frac{\pi}{2}\right)\frac{\chi}{\lambda}\,\,,\,\,2-2\pi\frac{\chi}{\lambda}\,\,,\,\,-2+2\pi\frac{\chi}{\lambda}\,\,,\,\,2-2\pi\frac{\chi}{\lambda}\,\,,\,\,-2+2\pi\frac{\chi}{\lambda}\bigg),\\
			{\rho}^R&=\bigg(-2+\left(\theta+\frac{3\pi}{2}\right)\frac{\chi}{\lambda}\bigg),
		\end{align*}
		and the left force points at $p^L_1=f(u^L)$, $p^L_2=f(y)$, $p^L_3=f(z)$, $p^L_4=f(x_2)$, $p^L_5=f(w)$ and the right force point at $p^L_1=f(u^R)$.
		
		With this new piece of information, we now want to conclude that a.s.\ $\widetilde{\beta}_{u,2}^L$ does not hit any of the points $f(x_1)$, $f(x_2)$, $f(x_3)$ in Figure~\ref{fig-boundary-conditions}. 
		
		Note that when $\theta\in [-\pi/2,\pi/2)$, then the sum of the weights of the force points at $f(u^L)$, $f(y)$ and $f(z)$ is
		\begin{equation*}
			-2-\left(\theta-\frac{\pi}{2}\right)\frac{\chi}{\lambda}\in\left(-2,\frac{16/\kappa}{2}\right].
		\end{equation*}
		while the sum of the weights of the force points at $f(u^L)$, $f(y)$, $f(z)$ and $f(x_2)$ is
		\begin{equation*}
			-\left(\theta+\frac{3\pi}{2}\right)\frac{\chi}{\lambda}\in\left(16/\kappa-4,\frac{16/\kappa}{2}-2\right].
		\end{equation*}
		Therefore, using \cref{lem:not-hitting}, we can conclude that a.s.\ $\widetilde{\beta}_{u,2}^L$ does not hit $f(x_1)$ and $f(x_3)$ thanks to the second assumption (since there are no force points at $f(x_1)$ and $f(x_3)$) and also that $\widetilde{\beta}_{u,2}^L$ does not hit $f(x_2)$ thanks to the first assumption. This completes the proof of the claim above when $\kappa\geq 6$.
		
		\medskip
		
		When $\kappa < 6$ then the flow lines can have self-intersections, as shown in \cref{fig-boundary-conditions2}.
		In this case, we let $U$ be a connected component of $\BB C \setminus \{\beta_{u,1}^L \cup \beta_{z,1}^L \cup \beta_{w,1}^L\}$ with $x$ on its boundary and let $\tau$ be the first time $\beta_{u,2}^L$ enters $U$; see \cref{fig-boundary-conditions2}. Since $\beta_{u,2}^L$ can cross each of $\beta_{z,1}^L$ and $\beta_{w,1}^L$ at most once and can hit the same point at most finitely many times~\cite[Proposition 3.31]{ig1}, $\beta_{u,2}^L$ can re-enter $U$ at most a finite number of times. Conditional on $\beta_{u,1}^L \cup \beta_{z,1}^L \cup \beta_{w,1}^L$ and $\beta_{u,2}^L|_{[0,\tau]}$, the conditional law of $\beta_{u,2}^L$ between time $\tau$ and its exit time from $U$ is that of a flow line of the field $h|_U$. So, we can read off the boundary data of $U$ and conclude as before that $\beta_{u,2}^L$ does not hit $x$ between time $\tau$ and its exit time from $U$. If $\beta_{u,2}^L$ re-enters $U$ multiple times, one can use the same argument to show that $\beta_{u,2}^L$ does not hit $x$ each of the times it enters $U$.
\end{proof}

\begin{figure}[ht!]
	\begin{center}
		\includegraphics[width=0.9\textwidth]{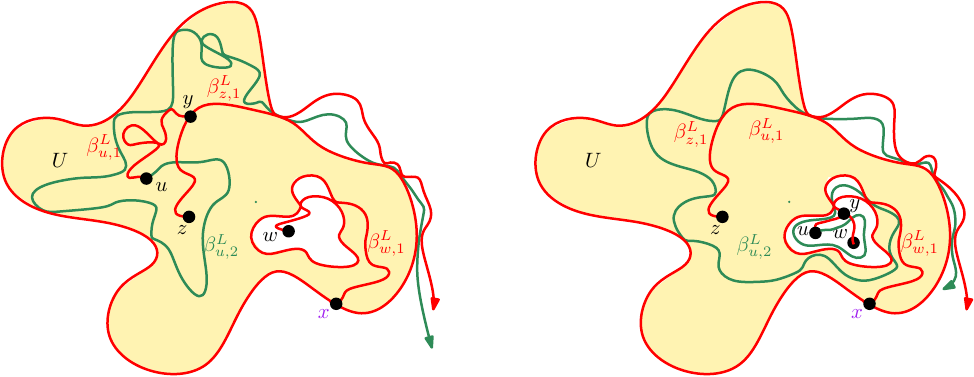}  
		\caption{\label{fig-boundary-conditions2} A diagram for the proof of \cref{lem:triple_are_simple} when $\kappa< 6$. On the left, the case when $u$ is inside $U$. On the right, the case when $u$ is outside $U$.
		}
	\end{center}
	\vspace{-3ex}
\end{figure}

We now turn to the proof of \cref{lem:not-hitting}. We first need two more preliminary lemmas.

\begin{lem}\label{lem:two-one-the-left}
	Fix $\kappa\in(0,4)$, $z\in(0,1)$ and ${\rho}_z,{\rho}_1\in \BB{R}$. Let $\eta$ be an SLE$_{\kappa}({\rho}_z,{\rho}_1)$ curve in the upper half-plane $\BB H$ from $0$ to $\infty$ and with force points at $z$ and $1$ such that $z$ has weight ${\rho}_z$ and $1$ has weight ${\rho}_1$. Assume that
 \begin{equation*}
			\rho_z>-2 \quad \text{and} \quad
			\rho_z+\rho_1>\frac{\kappa}{2}-4.
		\end{equation*}
		Then almost surely $\eta$ does not hit $1$.
\end{lem}

\begin{proof}
	The lemma statement follows immediately from \cite[Theorem 1.8]{miller-wu-dim}.
\end{proof}

\begin{lem}\label{lem:two-one-the-left2}
	Fix $\kappa\in(0,4)$ and ${\rho}_z,{\rho}_1\in \BB{R}$. Let $\eta$ be an SLE$_{\kappa}({\rho}_z,{\rho}_1)$ curve in the upper half-plane $\BB H$ from $0$ to $\infty$ and with force points at $z\in(0,1)$ and $1$ such that $z$ has weight ${\rho}_z$ and $1$ has weight ${\rho}_1$. Assume that 
		\begin{equation*}
	 			\rho_1=0.
		\end{equation*}
	Then for almost every $z\in(0,1)$,  $\eta$ almost surely does not hit $1$.
\end{lem}

\begin{proof}
		Note that, thanks to \cite[Lemma 15]{dubedat-duality}, if $\rho_{z}\geq\frac{\kappa}{2}-2$, then $\eta$ does not hit the right boundary of $\BB H$ at all, and if $\rho_{z}\leq\frac{\kappa}{2}-4$ then $\eta$ hits $z$ (without hitting $1$) and dies immediately after. 
		
		So we are left with the case $\frac{\kappa}{2}-4<\rho_{z}<\frac{\kappa}{2}-2$, i.e.\ the regime where  $\eta$ hits $(z,\infty)$. 
		If $-2<\rho_{z}<\frac{\kappa}{2}-2$, then the result follows immediately from \cref{lem:two-one-the-left}.
		
		We finally address the case when $\frac{\kappa}{2}-4<\rho_{z}\leq -2$. 
		We consider the conformal map $f(w)=w/z$. By scale invariance, we see that the probability that $\eta$ hits $1$ is equal to the probability that $f(\eta)$ hits $1/z$. Now, since $\rho_{z}\leq -2$ then $\eta$ dies immediately after hitting $(z,\infty)$. Therefore, $f(\eta)\cap (1,\infty)$ contains a single point. Hence, we see that the probability that $f(\eta)$ hits $1/z$ is zero for almost every $z\in(0,1)$. As a consequence, for almost all $z\in(0,1)$,  $\eta$ almost surely does not hit $1$. 
\end{proof}

We are now ready to complete the proof of \cref{lem:not-hitting}.

\begin{proof}[Proof of \cref{lem:not-hitting}]
	By \cite[Theorem 1]{ig1}, we can assume that $\eta$ is a flow line of a GFF $\hat h$ on $\BB H$, viewed modulo additive multiples of $2\pi\chi$, with boundary values equal to $\lambda(1+\overline{\rho}^R_i)$ on $(p^R_i,p^R_{i+1})$ and $-\lambda(1+\overline{\rho}^L_i)$ on $(p^L_{i+1},p^L_i)$, where $\overline{\rho}^*_i=\sum_{j=1}^{i}{\rho}^*_j$ for $*\in\{L,R\}$ (with the conventions that $p^*_0=0$, $p^R_{m+1}=\infty$ and $p^L_{n+1}=-\infty$).
	
	Let $A(r,R)$ be the closed annulus centered at $p_k^R$ and of radii $0<r<R$ such that $p_k^R+R<p_{k+1}^R$ and $p_k^R-R>p_{k-1}^R$. See the upper-left side of \cref{fig-does-not-hit}. Let $\tau$ be a stopping time such that $\eta(\tau)$ is on the internal boundary of $A(r,R)$, i.e.\ of the boundary of the disk centered at $p_k^R$ and of radius $r$. Let $f$ be the conformal mapping-out function for  $\eta$ that sends $\eta(\tau)$ to $0$ and such that $f(p_k^R)=1$. Using again \cite[Theorem 1]{ig1}, we find that the boundary values of $\hat{h}\circ f^{-1} - \chi \arg((f^{-1})')$ are equal to $\lambda(1+\overline{\rho}^R_i)$ on $(f(p^R_i),f(p^R_{i+1}))$ and $-\lambda(1+\overline{\rho}^L_i)$ on $(f(p^L_{i+1}),f(p^L_i))$; as shown on the top-right side of \cref{fig-does-not-hit}. We also set $U$ to be the image through $f$ of the closed disk centered at $p_k^R$ and of radius $R$.
	
	Let now $z\in(0,1)$ -- to be fixed later -- and consider another GFF $\widetilde h$, viewed modulo additive multiples of  $2\pi\chi$, with boundary values equal to $-\lambda$ on $(-\infty,0)$, $\lambda$ on $(0,z)$, $\lambda (1+\overline{\rho}_{k-1}^R)$ on $(z,1)$, and $\lambda (1+\overline{\rho}_{k}^R)$ on $(1,\infty)$. See the bottom part of \cref{fig-does-not-hit}. Let $U'$ be an open neighborhood of $U$ small enough so that the boundary values of $\hat{h}\circ f^{-1} - \chi \arg((f^{-1})')$ and $\widetilde h$ agree on $U'\cap\partial \BB H$ (note that this is possible because the boundary values of $\hat{h}\circ f^{-1} - \chi \arg((f^{-1})')$ and $\widetilde h$ are the same in a neighborhood of $0$ and $1$). Then, thanks to \cite[Proposition 3.4]{ig1}, $(\hat{h}\circ f^{-1} - \chi \arg((f^{-1})'))|_U$ and ${\widetilde h}|_U$ are mutually absolutely continuous (note that $U$ is independent of $\hat{h}\circ f^{-1} - \chi \arg((f^{-1})')$ and $\widetilde h$). In particular, if $\widetilde \eta^0$ is a flow line of $\widetilde h$ of angle 0 started from $0$, then the curves $f(\eta|_{[\tau, \infty)})$ and ${\widetilde \eta^0}$ stopped when they exist $U$ are also mutually absolutely continuous. Note that the law of $\widetilde \eta^0$ is that of an SLE$_\kappa(\overline{\rho}_{k-1}^R,\rho_{k}^R)$ with force points at $z$ and $1$.
	
	Let now $\tau_1$ be the first hitting time for $\eta$ of the internal boundary of $A(r,R)$, i.e.\ of the boundary of the disk centered at $p_k^R$ and of radius $r$.
	By \cref{lem:two-one-the-left} (for the Assumption 1) and \cref{lem:two-one-the-left2} (for the Assumption 2) -- applied to $\widetilde \eta^0$ with $z$ as in the statement of \cref{lem:two-one-the-left} -- and by the mutual absolute continuity discussed in the previous paragraph, we get that the probability $\eta$ hits $p^R_k$ between time $\tau$ and the first time after $\tau$ when it leaves the ball of radius $R$ centered at $p_k^R$ is zero.

	Finally, let $\sigma_1$ be the first time after $\tau_1$ at which $\eta$ exists the open disk of radius $R$ centered at $p_k^R$. Similarly, for all $\ell\geq 2$, let $\tau_\ell$ be the first hitting time after $\sigma_{\ell-1}$ when $\eta$ hits the closed disk of radius $r$ centered at $p_k^R$, and $\sigma_\ell$ be the first exit time after $\tau_{\ell}$ when $\eta$ exits the open disk of radius $R$ centered at $p_k^R$.
	Note that if $\eta$ hits $p^R_k$, then there must exist $\ell\geq 1$ such that $\eta$ hits $p^R_k$ between time $\tau_\ell$ and $\sigma_\ell$. By the conclusion of the previous paragraph, we see that the probability $\eta$ hits $p^R_k$ between time $\tau_\ell$ and $\sigma_\ell$ is also zero for all $\ell\geq 1$. This concludes the proof of the lemma.
\end{proof}

\begin{figure}[ht!]
	\begin{center}
		\includegraphics[width=0.9\textwidth]{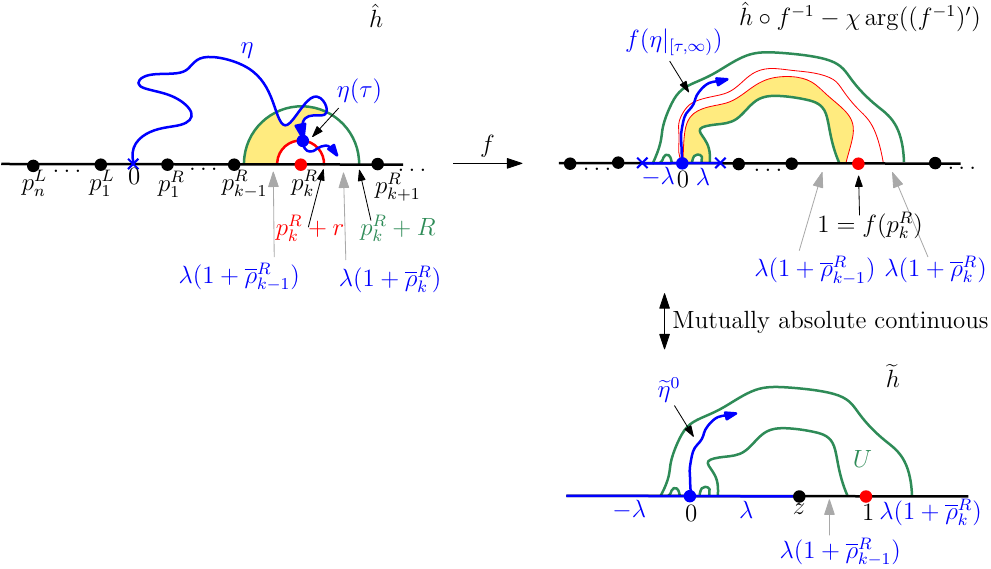}  
		\caption{\label{fig-does-not-hit} A schema for the proof of \cref{lem:not-hitting}.
		}
	\end{center}
	\vspace{-3ex}
\end{figure}

\section{Interplay between the support of the permuton and the multiple points of the SLEs}\label{sect:supp_perm}

The primary goal of this Section is to prove \cref{thm-permuton-multi-points}. We will start in \cref{sec-permuton-supp} by recalling and proving some basic results about the support of the permuton $\perm$ introduced in \eqref{eqn-permuton-def}. Then, in \cref{sec-permuton-support} we will prove \cref{thm-permuton-multi-points} in the case when  the two SLEs are independent, while in Sections~\ref{sec-skew-brownian-supp}~and~\ref{sect:proof--props}, we will prove \cref{thm-permuton-multi-points} in the case when the two SLEs are coupled as in Assumption~\ref{item-ig-gff2}.

\subsection{Basic properties of the support of the permuton}
\label{sec-permuton-supp}

We first recall the following basic results from \cite{bgs-meander}.

\begin{lem}{{\cite[Lemmas 2.8, 2.9, 2.10]{bgs-meander}}}\label{lem-permuton-defined} 
	Let $\perm$ be the permuton introduced in \eqref{eqn-permuton-def}.
	Almost surely, the following assertions are true. 
	\begin{enumerate}
		\item The permuton $\perm$ is well-defined, i.e., the definition does not depend on the choice of $\psi$. Moreover, a.s.\ for each rectangle $[a,b]\times[c,d] \subset [0,1]^2$, 
		\eqb \label{eqn-permuton-formula} 
		\perm\left([a,b]\times[c,d]\right) = \mu_h\left( \eta_1([a,b]) \cap \eta_2([c,d]) \right)  ,
		\eqe
		where we recall that $\mu_h$ is the $\gamma$-LQG area measure associated with a singly marked unit area $\gamma$-Liouville quantum sphere $(\BB C, h, \infty)$.
		\item For any choice of the function $\psi$ from~\eqref{eqn-psi-property},  
		\eqb \label{eqn-permuton-inclusion}
		\op{supp} \perm \subset \ol{\{(t,\psi(t)) : t\in [0,1]\} }\subset \{(t,s) \in [0,1] : \eta_1(t) = \eta_2(s) \}  .
		\eqe 
		
		\item  Let $(t,s) \in [0,1]^2$ such that $\eta_1(t) = \eta_2(s)$ and $\eta_2(s)$ is hit only once by $\eta_2$. Then $(t,s) $ belongs to the support of $\perm$. The same is true with the roles of $\eta_1$ and $\eta_2$ interchanged.
	\end{enumerate}
\end{lem}

\begin{remark}\label{rem-mult-points}
	Recall from the discussion below \eqref{eq:wewfihbw9euobfwu0} that both inclusions in~\eqref{eqn-permuton-inclusion} can potentially be strict. 
	Assertion 3 tells us that multiple points of $\eta_1$ and $\eta_2$ are the \emph{only} potential reason why the inclusions in~\eqref{eqn-permuton-inclusion} can fail to be equalities.
\end{remark}

We will prove (here and in the next sections) several lemmas concerning the set $\op{supp} \perm$.
We consider a particular choice of the function $\psi$ from~\eqref{eqn-psi-property}. For each $t\in [0,1]$, let 
\eqb \label{eqn-psi-def}
\psi_-(t) := \inf\left\{ s \in  [0,1]: \eta_2(s) = \eta_1(t) \right\} ,
\eqe
that is, the first time that the curve $\eta_2$ hits the point $\eta_1(t)$.

\begin{lem} \label{lem-closure-to-intersect}
Let $\psi_-$ be as in~\eqref{eqn-psi-def}. Almost surely, for each $s,t\in [0,1]$ we have $\eta_1(s) = \eta_1(t)$ if and only if there exists $q\in [0,1]$ such that 
\eqbn
(s,q) , (t,q) \in \ol{\{ (t,\psi_-(t)) : t\in [0,1]\}}.
\eqen
In particular, the intersection set $\mcl T \mcl M_1 = \mcl T \mcl M(\eta_1)$ from~\eqref{eqn-intersect-set} is a.s.\ determined by the set 
\[
\ol{\{ (t,\psi_-(t)) : t\in [0,1]\}}.
\]
\end{lem}
\begin{proof}
If $\eta_1(s) = \eta_1(t)$ then by~\eqref{eqn-psi-def} we have $\psi_-(s) = \psi_-(t)$, and so the condition in the lemma statement holds with $q = \psi_-(s)$. 
Conversely, assume that there exists $q\in [0,1]$ such that $(s,q), (t,q) \in \ol{\{ (t,\psi_-(t)) : t\in [0,1]\}}$.
Then there are sequences of times $s_j \rta s$ and $t_j \rta t$ such that $ \psi_-(s_j) \rta q$ and $\psi_-(t_j) \rta q$. By the definition of $\psi_-$ and the continuity of $\eta_1$ and $\eta_2$, 
\eqbn
|\eta_1(s) - \eta_1(t)| = \lim_{j\rta\infty} |\eta_1(s_j) - \eta_1(t_j)| = \lim_{j\rta\infty} |\eta_2(\psi_-(s_j)) - \eta_2(\psi_-(t_j))|  = 0 .
\eqen
This concludes the proof.
\end{proof}

\subsection{Characterization of the support when the two SLEs are independent}
\label{sec-permuton-support}

In this section, we fix $\gamma \in (0,2)$ and $\kappa_1,\kappa_2 > 4$. 
Let $(\BB C , h , \infty)$ be a singly marked unit area $\gamma$-Liouville quantum sphere.
Let $(\eta_1,\eta_2)$ be a pair of independent whole-plane space-filling SLE processes from $\infty$ to $\infty$, with parameters $\kappa_1$ and $\kappa_2$, sampled independently from $h$ and then parametrized by $\gamma$-LQG mass with respect to $h$.  
Let $\perm$ be the permuton associated with $(\eta_1,\eta_2)$ as in~\eqref{eqn-permuton-def}.  

The following proposition is our first non-trivial result about the support of $\perm$. 

\begin{prop} \label{prop-permuton-support}
Let $\psi_-$ be as in~\eqref{eqn-psi-def}. 
If $\eta_1$ and $\eta_2$, viewed modulo time parametrization, are independent then, almost surely,
\eqbn
\op{supp} \perm = \ol{\{(t,\psi_-(t)) : t\in [0,1]\} }.
\eqen
\end{prop} 

The statement above allows us to easily complete the proof of \cref{thm-permuton-multi-points} in the case where the two SLEs are independent.

\begin{proof}[Proof of  \cref{thm-permuton-multi-points} (Assumption~\ref{item-indep2})]
		Note that Proposition~\ref{prop-permuton-support} combined with Lemma~\ref{lem-closure-to-intersect} immediately implies that $\mcl T \mcl M_1$ is determined by $\op{supp}\perm$ when $\eta_1$ and $\eta_2$ are independent viewed modulo time parametrization. This proves \cref{thm-permuton-multi-points} when the two SLEs are independent.
\end{proof}

We note that, as we discussed in Remark~\ref{rem-mult-points}, Proposition~\ref{prop-permuton-support} is not true for a general choice of coupling between $\eta_1$ and $\eta_2$. Before we discuss the proof of Proposition~\ref{prop-permuton-support}, we record the following corollary when $\kappa_1 \geq 8$ and $\kappa_2 \geq 8$.

\begin{prop} \label{prop-permuton-full}
If $\eta_1$ and $\eta_2$, viewed modulo time parametrization, are independent and at least one among $\kappa_1$ and $\kappa_2$ is greater or equal than $8$ then, almost surely,
\eqb\label{eq:equalities-supp}
\op{supp} \perm = \ol{\{(t,\psi_-(t)) : t\in [0,1]\} } = \{(t,s) \in [0,1] : \eta_1(t) = \eta_2(s) \}  .
\eqe
\end{prop}

\begin{proof}[Proof of Proposition~\ref{prop-permuton-full} assuming Proposition~\ref{prop-permuton-support}]
Without loss of generality, we can assume that $\kappa_2\geq 8$.
The first equality is Proposition~\ref{prop-permuton-support}.
By Assertion 2 in Lemma~\ref{lem-permuton-defined}, 
\[\op{supp} \perm \subset \{(t,s) \in [0,1] : \eta_1(t) = \eta_2(s) \}.\] 
We need to prove the reverse inclusion. 
Extending~\eqref{eqn-psi-def}, for $t\in [0,1]$, we let $\psi_-(t)$ be the first time $\eta_2$ hits $\eta_1(t)$ and let $\psi_+(t)$ be the last time that $\eta_2$ hits $\eta_1(t)$. 
By the first equality in \eqref{eq:equalities-supp}, a.s.\ the support of $\perm$ contains $(t,\psi_-(t))$ for each $t\in [0,1]$. 
By the time reversibility of $\eta_2$ (and the fact that the definition of $\perm$ does not depend on the particular choice of $\psi$, see Assertion 1 in Lemma~\ref{lem-permuton-defined}), a.s.\ the support of $\perm$ contains $(t,\psi_+(t))$ for each $t\in [0,1]$. 

Therefore, a.s.\ the support of $\perm$ contains $(t,s)$ whenever $\eta_1(t) = \eta_2(s)$ and $\eta_1(t)$ is hit at most twice by $\eta_2$. 
By the symmetry between $\eta_1$ and $\eta_2$, a.s.\ the support of $\perm$ also contains $(t,s)$ whenever $\eta_1(t) = \eta_2(s)$ and $\eta_2(s)$ is hit at most twice by $\eta_2$. 

Since when $\kappa_2\geq 8$ the maximum number of times that $\eta_2$ hits any $z\in\BB C$ is 3, it remains to look at the triple points of $\eta_2$. Almost surely, the set of triple points of $\eta_2$ is countable. 
Since $\eta_1$ and $\eta_2$ are independent modulo time parametrization, and every fixed point $x \in\BB C$ is a.s.\ hit once by $\eta_1$, we conclude that a.s.\ each triple point of $\eta_2$ is hit only once by $\eta_1$.
In particular, a.s. each point of $\BB C$ is hit at most twice by $\eta_1$ or $\eta_2$ (or both). Combining this with the conclusion of the preceding paragraph, we can conclude the proof.
\end{proof}

We now turn our attention to the proof of Proposition~\ref{prop-permuton-support}. 
In light of Assertion 1 in Lemma~\ref{lem-permuton-defined}, we need to show that a.s.\ whenever $(a,b)\times(c,d)\subset [0,1]^2$ is a rectangle which contains a point of the form $(t,\psi_-(t))$, then $\mu_h(\eta_1([a,b]) \cap \eta_2([c,d])) > 0$. 
To establish this, we want to, roughly speaking, show that when $\eta_2$ hits $\eta_1(t)$, then $\eta_2$ must immediately fill an open subset contained in $\eta_1([a,b])$.

\medskip

This will require several lemmas about space-filling SLE. 
The next three lemmas do not involve any LQG, so for these lemmas we will parametrize our space-filling SLE curves by Lebesgue measure instead of $\mu_h$-mass. We note that the time change for space-filling SLE to get from the Lebesgue measure parametrization to the $\mu_h$-mass parametrization is a homeomorphism from $\BB{R}$ to $(0,1)$.  
Our first lemma gives a strong Markov property for whole-plane space-filling SLE.

\begin{lem} \label{lem-space-filling-law}
Let $\kappa>4$ and let $\eta$ be a whole-plane space-filling SLE$_\kappa$ from $\infty$ to $\infty$, parametrized by Lebesgue measure in such a way that $\eta(0) = 0$. Let $\tau$ be a stopping time for $\eta$. 
\begin{itemize}
\item If $\kappa \geq 8$, the conditional law of $\eta|_{[\tau,\infty)}$ given $\eta|_{(-\infty,\tau]}$ is that of a chordal SLE$_\kappa$ curve from $\eta(\tau)$ to $\infty$ in $\BB C\setminus \eta((-\infty,\tau])$. 
\item If $\kappa \in (4,8)$, the conditional law of $\eta|_{[\tau,\infty)}$ given $\eta|_{(-\infty,\tau]}$ is that of a concatenation of conditionally independent chordal space-filling SLE$_\kappa$ curves in the connected components of $\BB C\setminus \eta((-\infty,\tau])$, each of which goes between the two points on the boundary of the component where the left and right boundaries of $\eta((-\infty,\tau])$ meet.\footnote{
It is possible that the boundary of one of the connected components of $\BB C\setminus \eta((-\infty,\tau])$ (necessarily the one with $\eta(\tau)$ on its boundary) is entirely part of the left boundary of $\eta((-\infty,\tau])$ or entirely part of the right boundary of $\eta((-\infty,\tau])$. In this case, the conditional law of the segment of $\eta$ in this component is instead that of a chordal space-filling SLE$_\kappa$ loop based at $\eta(\tau)$. See~\cite[Proposition A.3]{bg-lbm} for more on chordal space-filling SLE loops.
}
\end{itemize}
\end{lem}
\begin{proof}
For $z\in\BB C$, let $\tau_z$ be the first time that $\eta$ hits $z$. By the description of the law of whole-plane space-filling SLE$_\kappa$ in Section~\ref{sec-sle} (see also \cite[Footnote 4]{wedges}), together with the translation invariance of the law of $\eta$, the description of the conditional law of $\eta|_{[\tau,\infty)}$ in the statement of the lemma is true whenever $\tau = \tau_z$ for a fixed $z\in\BB C$. From this, we immediately deduce that this statement is also true if there is a deterministic countable set $Z\subset \BB C$ such that a.s.\ $\eta(\tau) \in Z$. To get the lemma statement for a general stopping time $\tau$, let $\tau_k$, for $k\in \BB N$,  be the smallest time $t\geq \tau$ for which $\eta(t) \in 2^{-k} \BB Z^2$. For every $\ep > 0$, the set $\eta([\tau , \tau+\ep])$ has non-empty interior, so this set contains a point of $2^{-k} \BB Z^2$ for each large enough $k \in\BB N$. Consequently, $\tau_k$ a.s.\ decreases to $\tau$.
Now, let $U_1,\dots,U_N$ be connected components of $\BB C\setminus \eta((-\infty,\tau])$, chosen in a manner which depends only on $\eta|_{(-\infty,\tau]}$. For each large enough $k$, we have that $U_1,\dots,U_N$ are also connected components of $\BB C\setminus \eta((-\infty,\tau_k])$. We know that for each $k$, the conditional law given $\eta|_{(-\infty,\tau_k]}$ of the segments of $\eta$ in $\ol U_1,\dots,\ol U_N$ is that of a collection of independent space-filling SLE$_\kappa$ curves between the marked points of the components. By the backward martingale convergence theorem (applied to an arbitrary bounded measurable function of these $N$ curve segments), the same holds if we instead condition on $\eta|_{(-\infty,\tau]}$. Therefore, we can now deduce the description of the conditional law of $\eta|_{[\tau,\infty)}$ in the general case. 
\end{proof}

Lemma~\ref{lem-space-filling-law} allows us to reduce certain problems about whole-plane space-filling SLE$_\kappa$ to problems about chordal SLE$_\kappa$. We will need the following lemma in the chordal case.

\begin{lem} \label{lem-chordal-fill}
Let $\eta$ be a chordal space-filling SLE$_\kappa$ from 0 to $\infty$ in $\BB H$ parametrized by Lebesgue measure. Almost surely, for each $\ep > 0$ the set $\eta([0,\ep])$ contains an open neighborhood of 0 in $\BB H$.
\end{lem}
\begin{proof}
Standard results on SLE (see, e.g.,~\cite[Section 6.2]{lawler-book}) show that if $\wt\eta$ is an ordinary chordal SLE$_\kappa$ curve from 0 to $\infty$ in $\BB H$, parametrized by half-plane capacity, then a.s.\ for each $\delta > 0$ the hull generated by $\wt\eta([0,\delta])$ contains an open neighborhood of 0 in $\BB H$. By construction of space-filling SLE$_\kappa$ in~\cite{ig4}, the space-filling SLE$_\kappa$ curve $\eta$ visits the points of $\wt\eta$ in chronological order and fills in each region which is disconnected from $\infty$ by $\wt\eta$ immediately after it is disconnected. Since $\eta$ is parametrized by Lebesgue measure, if $\delta > 0$ is chosen to be small enough so that the hull generated by $\wt\eta([0,\delta])$ has Lebesgue measure at most $\ep$, then this hull is contained in $\eta([0,\ep])$. 
Since this hull contains an open neighborhood of the origin in $\BB H$, so does $\eta([0,\ep])$. 
\end{proof}

The following lemma will be used in the proof of Proposition~\ref{prop-permuton-support} to show that the intersection of two segments of $\eta_1$ and $\eta_2$ contains an open set (and hence has positive $\mu_h$-mass).

\begin{figure}[ht!]
\begin{center}
\includegraphics[scale=.78]{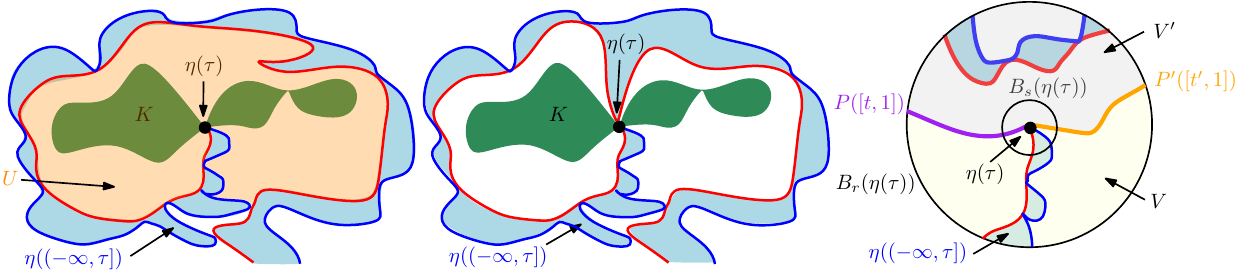}  
\caption{\label{fig-space-filling-fill} \textbf{Left:} Illustration of the statement of Lemma~\ref{lem-space-filling-fill} in the case when $\kappa \in (4,8)$. The domain $U$ is the connected component of $\BB C\setminus \eta([0,\tau])$ which contains $K\setminus \{\eta(\tau)\}$. The lemma asserts that for every $\ep > 0$, $\eta([\tau,\tau+\ep])$ contains the intersection with $U$ of an open disk centered at $\eta(\tau)$. 
\textbf{Middle:} One ``bad configuration'' that is forbidden by the statement of Lemma~\ref{lem-space-filling-fill} in the case when $\kappa \in (4,8)$.
\textbf{Right:} Illustration for the proof of the claim that there is at most one connected component of $\BB C\setminus \eta((-\infty,\tau])$ with $\eta(\tau)$ on its boundary. The intersection of $\eta((-\infty,\tau])$ with $B_r(\eta(\tau))$ is shown in light blue. The continuity of $\eta$ and the fact that $\eta$ does not hit $P$ or $P'$ before time $\tau$ allows us to find an $s>0$ such that $\eta((-\infty,\tau])$ is disjoint from the intersection of $B_s(\eta(\tau))$ with one of the two connected components of $B_r(\eta(\tau)) \setminus (P([t,1]) \cup P'([t',1]))$. This intersection contains a path from a point of $P([t,1))$ to a point of $P'([t,1))$, which contradicts the assumption that $P$ and $P'$ lie in different connected components of $\BB C\setminus \eta((-\infty,\tau])$. 
}
\end{center}
\vspace{-3ex}
\end{figure}

\begin{lem} \label{lem-space-filling-fill}
Let $\kappa>4$ and let $\eta$ be a whole-plane space-filling SLE$_\kappa$ from $\infty$ to $\infty$, parametrized by Lebesgue measure in such a way that $\eta(0) = 0$. Let $\tau$ be a stopping time for $\eta$.
\begin{itemize}
\item If $\kappa\geq 8$, then a.s.\ for each $\ep > 0$ the set $\eta([\tau,\tau+\ep])$ contains the intersection of $\BB C\setminus \eta((-\infty,\tau])$ with an open disk centered at $\eta(\tau)$. 
\item Suppose that $\kappa \in (4,8)$ and the following is true (see the left-hand side of Figure~\ref{fig-space-filling-fill} for an illustration). The stopping time $\tau$ is the first time at which $\eta$ hits $K$, where $K$ is a random closed path-connected non-singleton subset of $\BB C$ which is measurable with respect to $\sigma\left(\eta|_{(-\infty,\tau]} \right)$. Almost surely, there is a unique connected component $U$ of $\BB C\setminus \eta([0,\tau])$ which contains $K\setminus \{\eta(\tau)\}$ and for each $\ep > 0$ the set $\eta([\tau,\tau+\ep])$ contains the intersection of $U$ with an open disk centered at $\eta(\tau)$. 
\end{itemize} 
\end{lem}
\begin{proof}
If $\kappa \geq 8$, then by Lemma~\ref{lem-space-filling-law} the conditional law of $\eta|_{[\tau,\infty)}$ given $\eta|_{(-\infty,\tau]}$ is that of a chordal space-filling SLE$_\kappa$ from $\eta(\tau)$ to $\infty$ in $\BB C\setminus \eta((-\infty,\tau])$. Therefore, the result in this case is immediate from Lemma~\ref{lem-chordal-fill}.

\medskip

Now suppose that $\kappa \in (4,8)$ and let $K$ be as in the lemma statement. 
We first prove the following result.

\medskip

\noindent\textbf{Claim.}\ There is at most one connected component of $\BB C\setminus \eta((-\infty,\tau])$ which has $\eta(\tau)$ on its boundary.

\medskip

Suppose by way of contradiction that there are two distinct connected components $U$ and $U'$ with $\eta(\tau)$ on their boundaries. 
Since $\eta$ is a curve, each of $U$ and $U'$ is simply connected and is bounded by a curve. 
Choose points $z\in U$ and $z'\in U'$.  
By, e.g.,~\cite[Theorem 2.1]{pom-book} we can find a simple path $P : [0,1] \rta \ol U$ from $z$ to $\eta(\tau)$ with $P([0,1)) \subset U$. 
We can similarly find a path $P'$ satisfying the same properties with $(U',z')$ instead of $(U,z)$. 

Let $r > 0$ be small enough so that $z,z' \notin B_r(\eta(\tau))$ and let $t$ and $t'$ be the last times that $P$ and $P'$, respectively, enter $B_r(\eta(\tau))$. See the right-hand side of Figure~\ref{fig-space-filling-fill} for an illustration. Then $P([t,1])$ and $P'([t',1])$ are simple paths in $\ol{B_r(\eta(\tau))}$ from $\bdy B_r(\eta(\tau))$ to $\eta(\tau)$. Moreover, these paths do not intersect each other except at $\eta(\tau)$ since $U\cap U' = \emptyset$. 
Therefore, $B_r(\eta(\tau)) \setminus (P([t,1]) \cup P'([t',1])$ has exactly two connected components, both homeomorphic to the disk. Call these components $V$ and $V'$. Since $\eta$ is a continuous curve and $\eta$ does not hit $P([t,1]) \cup P'([t',1])$ before time $\tau$, there is a $\delta > 0$ such that $\eta([\tau-\delta,\tau))$ is entirely contained in one of these two connected components, say $V$. 

Hence there exists $s  \in (0,r)$ such that $\eta((-\infty,\tau])$ does not hit $B_s(\eta(\tau)) \cap V'$. However, there is a path in $B_s(\eta(\tau)) \cap V'$ from a point of $P([t,1))$ to a point of $P'([t',1))$. This path does not intersect $\eta((-\infty,\tau])$, so $P$ and $P'$ must be contained in the same connected component of $\BB C\setminus \eta((-\infty,\tau])$. This contradicts our initial choice of $U$ and $U'$, and we conclude the proof of the claim. 

\medskip

By assumption, $\eta(\tau)$ is the only point of $\eta((-\infty,\tau])$ that belongs to $K$. Since $K$ is path-connected and not a singleton, every connected component of $K\setminus \{\eta (\tau)\}$ is a connected non-singleton set with $\eta (\tau)$ on its boundary. In particular, each connected component of $K\setminus \{\eta (\tau)\}$ must be contained in a connected component of $\BB C\setminus \eta((-\infty,\tau])$ with $\eta (\tau)$ on its boundary. From the Claim above, there is at most one such connected component, so in fact there is exactly one connected component of $\BB C\setminus \eta((-\infty,\tau])$ with $\eta(\tau)$ on its boundary, and this connected component contains $K\setminus \{\eta(\tau)\}$. We call this connected component $U$. 

Since $\tau$ is the first time $\eta$ hits $K$, $\tau$ is also the first time that $\eta$ hits $\eta(\tau)$. Therefore, the first point of $\bdy U$ hit by $\eta$ is not equal to $\eta(\tau)$ and so $\bdy U$ has non-trivial arcs which are part of each of the left and right boundaries of $\eta((-\infty,\tau])$. 

Consequently, Lemma~\ref{lem-space-filling-law} implies that the conditional law given $\eta|_{(-\infty,\tau]}$ of $\eta|_{[\tau,\infty)}$ stopped at its first exit time from $\ol U$ is that of a chordal space-filling SLE$_\kappa$ in $U$ from $\eta(\tau)$ to some point in $\bdy U\setminus \{\eta(\tau)\}$. Therefore, Lemma~\ref{lem-chordal-fill} implies that a.s.\ for each $\ep > 0$, the set $\eta([\tau,\tau+\ep])$ contains the intersection of $U$ with an open disk centered at $\eta(\tau)$. 
\end{proof}

We can now prove Proposition~\ref{prop-permuton-support}.

\begin{proof}[Proof of Proposition~\ref{prop-permuton-support}]
From Assertion 2 in Lemma~\ref{lem-permuton-defined} we have that the support of $\perm$ is contained in $\ol{\{(t,\psi_-(t)) : t\in [0,1]\}}$. 
We now prove the reverse inclusion under the assumption that $\eta_1$ and $\eta_2$, viewed modulo time parametrization, are independent.

Fix an open rectangle $(a,b) \times (c,d) \subset [0,1]^2$. We have the following result.

\medskip

\noindent\textbf{Claim.}\ Almost surely, if there exists $t\in (a,b)$ such that $\psi_-(t) \in (c,d)$, then $\eta_1([a,b]) \cap \eta_2([c,d])$ contains an open set (and hence has positive $\mu_h$-mass).

\medskip

Before proving the claim, we explain why it implies the lemma statement. By Lemma~\ref{lem-permuton-defined}, Assertion 1, and a union bound over rational choices of $a,b,c,d$, we see that the claim implies that the following is true. Almost surely, for every $a,b,c,d \in [0,1] \cap \BB Q$ with $a<b$ and $c < d$ such that $ ((a,b) \times (c,d)) \cap   \ol{\{(t,\psi_-(t)) : t\in [0,1]\}} \not=\emptyset$, we have $\perm([a,b] \times [c,d])  >0$. 
Hence, almost surely, $\perm$ assigns a positive mass to arbitrarily small neighborhoods of $(t,\psi_-(t))$ for each $t\in [0,1]$.
This implies that almost surely the support of $\perm$ contains $\{(t,\psi_-(t)) : t\in [0,1]\}$ and hence its closure as well.

\medskip

It remains to prove our previous claim. Since $\psi_-(t)$ is defined to be the \emph{first} time that $\eta_2$ hits $t$, a.s.\ if there exists $t\in (a,b)$ such that $\psi_-(t) \in (c,d)$, then $\eta_1(t) \notin \eta_2([0,c])$. Since $\eta_2([0,c])$ is closed, if such a $t$ exists, then there are rational times $p,q \in (a,b)$ with $p < q$ such that $\eta_1([p,q])$ is disjoint from $\eta_2([0,c])$ and $\eta_2$ hits $\eta_1([p,q])$ during the time interval $(c,d)$. Therefore, it suffices to fix $p,q \in (a,b)$ with $p<q$ and show that a.s.\ if $\eta_2$ hits $\eta_1([p,q])$ for the first time during the time interval $(c,d)$, then $\eta_1([p,q]) \cap \eta_2([c,d])$ contains an open set.

To prove this last assertion, let $\tau$ be the first time that $\eta_2$ hits $\eta_1([p,q])$. We need to show that a.s.\ if $\tau \in (c,d)$, then $\eta_1([p,q]) \cap \eta_2([c,d])$ contains an open set.
Since $\eta_2$, viewed modulo time parameterization, is independent of $(h,\eta_1)$, the time $\tau$ is a stopping time for $\eta_2$ under the conditional law given $(h,\eta_1)$. By Lemma~\ref{lem-space-filling-fill}, applied under the conditional law given $(h,\eta_1)$ and with $K = \eta_1([p,q])$, a.s. for each $\ep>0$ the set $\eta_2([\tau,\tau+\ep])$ contains the intersection of some open disk centered at $\eta_2(\tau)$ with the connected component of $\BB C\setminus \eta_2([0,\tau])$ that contains $\eta_1([p,q])\setminus \{\eta_2(\tau)\}$. Hence, a.s.\ $\eta_2([\tau,\tau+\ep])$ contains the intersection of some open disk centered at $\eta_2(\tau)$ with the set $\eta_1([p,q])\setminus \{\eta_2(\tau)\}$. 

The set $\eta_1([p,q])$ is equal to the closure of its interior, so the last sentence of the preceding paragraph implies that $\eta_1([p,q])\cap \eta_2([\tau,\tau+\ep])$ contains an open set for each $\ep > 0$. In particular, if $\tau \in (c,d)$ then $\eta_1([p,q])\cap \eta_2([c,d])$ contains an open set, as required.
\end{proof}

\subsection{Characterization of the support when the two SLEs are coupled}
\label{sec-skew-brownian-supp}

Throughout this subsection, we fix $\kappa > 4$, $\theta\in (-\pi/2,\pi/2)$, and $\gamma\in (0,2)$. Let $(\BB C ,h , \infty)$ be a singly marked unit area $\gamma$-Liouville quantum sphere as introduced in \cref{def-sphere}. Let $\wh h$ be a whole-plane GFF, viewed modulo additive multiples of $2\pi\chi$, sampled independently of $h$. 
Let $\eta_1$ and $\eta_2$ be its space-filling counterflow lines from $\infty$ to $\infty$ with angles 0 and $\theta-\pi/2$, respectively, and both with parameter $\kappa$. We parametrize $\eta_1$ and $\eta_2$ by $\gamma$-LQG mass with respect to $h$ and let $\perm$ be the associated permuton as in~\eqref{eqn-permuton-def}. When $\gamma^2  = 16/\kappa$, the permuton $\perm$ is the skew Brownian permuton with parameters $\rho = -\cos(\pi\gamma^2/4)$ and $q = q_\gamma(\theta) \in (0,1)$.

\begin{prop} \label{prop-counterflow-dichotomy}
	Let $\kappa \geq 8$ and $\theta\in (-\pi/2,\pi/2)$. 
	\begin{enumerate}[$(i)$]
		\item \label{item-counterflow-good} If 
		\eqb \label{eqn-angle-condition}
		\theta \in \left[ - \frac{(\kappa-12)\pi}{2(\kappa-4)} ,  \frac{(\kappa-12)\pi}{2(\kappa-4)}  \right] 
		\eqe
		then almost surely
		\eqbn
		\op{supp} \perm = \ol{\{(t,\psi_{-}(t)) : t\in [0,1]\} } = \{(t,s) \in [0,1] : \eta_1(t) = \eta_2(s) \}  ,
		\eqen 
		where we recall that $\psi_-  :[0,1] \rta [0,1]$, introduced in~\eqref{eqn-psi-def}, is defined so that $\psi_-(t)$ is the first time that the curve $\eta_2$ hits the point $\eta_1(t)$.
		
		\item \label{item-counterflow-good2} If~\eqref{eqn-angle-condition} does not hold, then almost surely there are uncountably many points of 
		\eqb\label{eq:points_not_in}
		\ol{\{(t,\psi_{-}(t)) : t\in [0,1]\} }
		\eqe
		which do not belong to $\op{supp} \perm$. 
		
		But, almost surely, if $z\in \BB C$ is not simultaneously a double point of both $\eta_1$ and $\eta_2$ and $t\in[0,1]$ is a time when $\eta_1$ hits $z$, then $(t,\psi_{-}(t))\in \op{supp} \perm$. While, if $z$ is simultaneously a double point of both $\eta_1$ and $\eta_2$, then $\op{supp} \perm$ contains at least three of the four pairs $(t,s)\in[0,1]^2$ for which $\eta_1(t) = \eta_2(s) = z$ (c.f., left-hand side of \cref{fig-spiral-perm}).
\end{enumerate}
\end{prop}

We note that there exists $\theta \in (-\pi/2,\pi/2)$ for which~\eqref{eqn-angle-condition} holds if any only if $\kappa \geq 12$. If $\kappa =12$, the only allowable angle is $\theta=0$, which corresponds to the case of the Baxter permuton \cite{bm-baxter-permutation,bhsy-baxter-permuton}. 

\cref{prop-counterflow-dichotomy} will be proved in \cref{sect:kappa-geq-8}. Some of the points of the set \eqref{eq:points_not_in} not contained in $\op{supp} \perm$ are explicitly described in the proof of Lemma~\ref{lem-counterflow-bad}.

\begin{prop} \label{prop-counterflow-dichotomy2}
	Let $\kappa \in (4,8)$ and $\theta\in (-\pi/2,\pi/2)$, then almost surely there are uncountably many points of 
		\eqb\label{eq:points_not_in2}
		\ol{\{(t,\psi_{-}(t)) : t\in [0,1]\} }
		\eqe
	which do not belong to $\op{supp} \perm$. But, almost surely,
	\begin{itemize}
		\item if for all $m_1\geq 2$ and $m_2\geq 2$, 
		\[\text{$z\in \BB C$ is not simultaneously  a $m_1$-tuple point of $\eta_1$ and a $m_2$-tuple point of $\eta_2$,}\] 
		and $t$ is a time when $\eta_1$ hits $z$, then $(t,\psi_{-}(t))\in \op{supp} \perm$.
		\item while, if for some $m_1\geq 2$ and $m_2\geq 2$, 
		\[\text{$z\in \BB C$ is simultaneously an $m_1$-tuple point of $\eta_1$ and an $m_2$-tuple point of $\eta_2$,}\]
		then the following fact is true. Let $t^1_1,t^1_2,\dots,t^1_{m_1}$ be the $m_1$ times when $\eta_1$ hits $z$ and $t^2_1,t^2_2,\dots,t^2_{m_2}$ be the $m_2$ times when $\eta_2$ hits $z$. Then, for all $(s,t) \in [0,1]^2$ such that $\eta_1(s) =\eta_1(t)=z$, there exists a sequence of times $(t^1_{i_\ell}, t^2_{j_\ell})_{\ell=0}^M$ such that (c.f., right-hand side of \cref{fig-spiral-perm})
		\begin{itemize}
			\item $t^1_{i_0}=s$ and $t^1_{i_M}=t$;
			\item for every $\ell\in\{1,2,\dots,M\}$, $(t^1_{i_\ell}, t^2_{j_\ell}) \in \op{supp} \perm $;
			\item for every $\ell\in\{2,\dots,M\}$, either $t^1_{i_\ell}=t^1_{i_{\ell-1}}$ or $t^2_{j_\ell}=t^2_{j_{\ell-1}}$.
		\end{itemize} 
	\end{itemize}	
\end{prop}

\begin{figure}[ht!]
	\begin{center}
		\includegraphics[scale=.6]{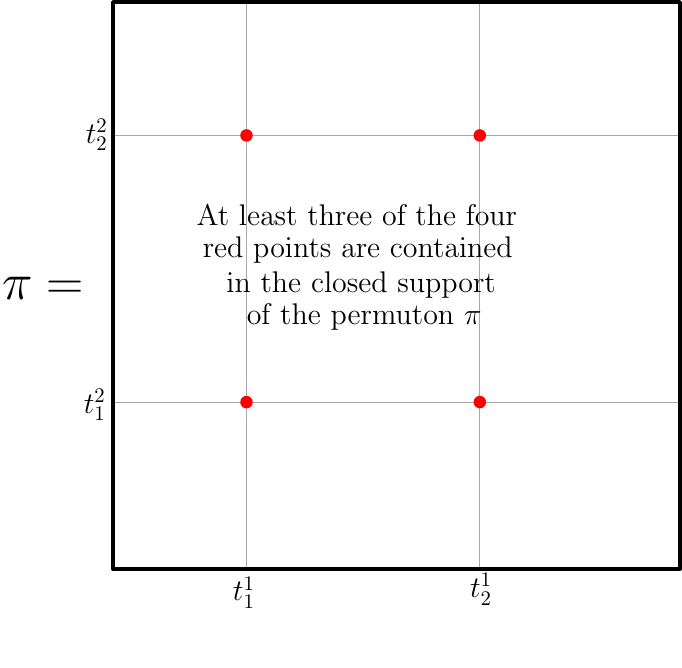}  
		\hspace{.5cm}
		\includegraphics[scale=.6]{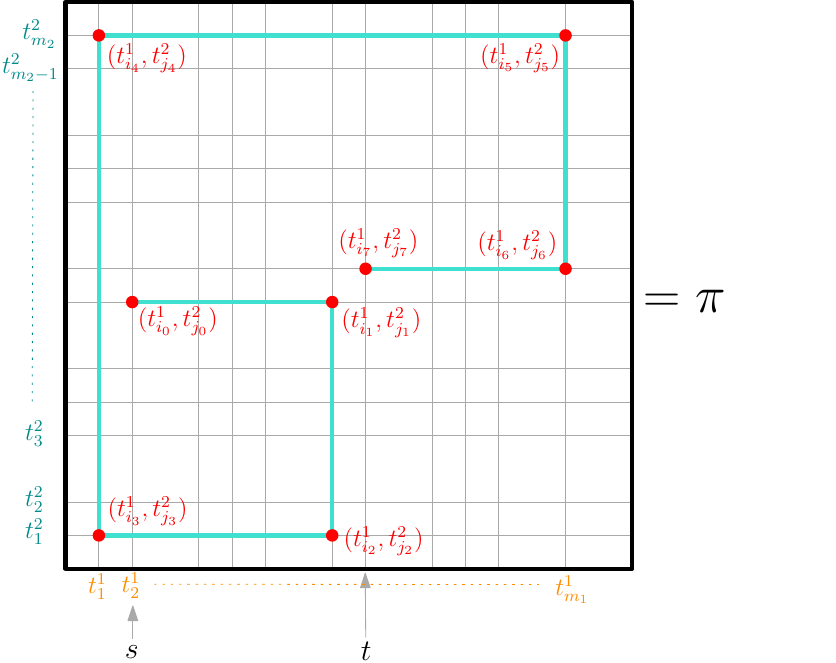}  
		\caption{\label{fig-spiral-perm} 
		\textbf{Left:} Here $z$ is simultaneously a double point of $\eta_1$ and $\eta_2$, and $t^1_1$ and $t^1_2$ (resp. $t^2_1$ and $t^2_2$) are the first and second time when $\eta_1$ (resp.\ $\eta_2$) hits $z$.  At least three of the four points $(t^1_1,t^2_1)$,  $(t^1_1,t^2_2)$,  $(t^1_2,t^2_1)$ and  $(t^1_2,t^2_2)$ are contained in $\op{supp} \perm$.
		\textbf{Right:} Here $z$ is simultaneously an $m_1$-tuple point of $\eta_1$ and an $m_2$-tuple point of $\eta_2$ for some $m_1\geq 2$ and $m_2\geq 2$ and $t^1_1,t^1_2,\dots,t^1_{m_1}$ (resp.\ $t^2_1,t^2_2,\dots,t^2_{m_2}$) are the $m_1$ (resp.\ $m_2$) times when $\eta_1$ (resp.\ $\eta_2$) hits $z$.  For all $(s,t) \in [0,1]^2$ such that $\eta_1(s) =\eta_1(t)=z$, i.e., for all $(s,t) \in [0,1]^2$ such that $s=t^1_i$ and $t=t^1_j$ for some $i,j\in[m_1]$, \cref{prop-counterflow-dichotomy2} ensures the existence of a sequence of red points $(t^1_{i_\ell}, t^2_{j_\ell})_{\ell=0}^M$ in the support of the permuton with the properties highlighted in the picture and precisely described in the proposition statement.
		}
	\end{center}
	\vspace{-3ex}
\end{figure}

\cref{prop-counterflow-dichotomy2} will be proved in \cref{sect:proof-remaining}.
 Some of the points in the set \eqref{eq:points_not_in2} not contained in $\op{supp} \perm$ are explicitly described in \cref{lem:points_in_support}.

\begin{remark}
    Note that since $\ol{\{(t,\psi_-(t)) : t\in [0,1]\} }\subset \{(t,s) \in [0,1] : \eta_1(t) = \eta_2(s) \}$ by \cref{lem-permuton-defined}, then we can immediately deduce from Propositions~\ref{prop-counterflow-dichotomy}~and~\ref{prop-counterflow-dichotomy2} that, when $\kappa\in (4,8)$ or $\kappa \geq 8$ and~\eqref{eqn-angle-condition} does not hold, then there are uncountably many points of $\{(t,s) \in [0,1] : \eta_1(t) = \eta_2(s) \}$ that do not belong to $\op{supp} \perm$.
\end{remark}

We now explain how Theorem~\ref{thm-permuton-multi-points} in the case in Assumption~\ref{item-ig-gff2} follows from Propositions~\ref{prop-counterflow-dichotomy}~and~\ref{prop-counterflow-dichotomy2}.

\begin{proof}[Proof of \cref{thm-permuton-multi-points} (Assumption~\ref{item-ig-gff2})]
    We prove the theorem for $\mcl T \mcl M_1$. The proof for $\mcl T \mcl M_2$ is identical.
	We assume that $\eta_1$ and $\eta_2$ are the space-filling counterflow lines of the same whole-plane GFF with angles 0 and $\theta-\pi/2$ for some $\theta\in(-\pi/2,\pi/2)$ and both have parameter $\kappa >4$. We distinguish three cases:
	
	\medskip
	
	\noindent\underline{Case 1:} Assume that $\kappa \geq 8$ and $\theta \in (-\pi/2,\pi/2)$ satisfies~\eqref{eqn-angle-condition}. Then, the theorem statement follows combining Assertion~\eqref{item-counterflow-good} in  \cref{prop-counterflow-dichotomy} and Lemma~\ref{lem-closure-to-intersect}.
	
	\medskip
	
	\noindent\underline{Case 2:} Assume that $\kappa \geq 8$ and $\theta \in (-\pi/2,\pi/2)$ does not satisfy~\eqref{eqn-angle-condition}.
	We prove that the set $\mcl{T}\mcl{M}_1=\left\{ (s,t) \in [0,1]^2 : \eta_1(s) =\eta_1(t) \right\}$  is a.s.\ determined by $\op{supp} \perm$, by proving that, almost surely, for all $(s,t) \in [0,1]^2$,
	\begin{equation}\label{eq:iif-cond}
		\eta_1(s) =\eta_1(t)\iff \exists\, q\in [0,1] \text{ such that } (s,q) , (t,q) \in \op{supp} \perm.
	\end{equation}
	Note that the $\Longleftarrow$ implication immediately follows from 
	\cref{lem-closure-to-intersect} and the fact that 
	$$\op{supp} \perm \subset \ol{\{ (t,\psi_-(t)) : t\in [0,1]\}}.$$ 
	It remains to prove the $\Longrightarrow$ implication.
	Fix $(s,t) \in [0,1]^2$ such that $\eta_1(s) =\eta_1(t)$. 
	
	We first assume that $\eta_1(t)$ is not simultaneously a double point of both $\eta_1$ and $\eta_2$. Then, thanks to Assertion~\eqref{item-counterflow-good2} in \cref{prop-counterflow-dichotomy}, we have  $(t,\psi_-(t))\in \op{supp} \perm$ and $(s,\psi_-(s))\in \op{supp} \perm$. But since $\eta_1(s) =\eta_1(t)$ by assumption, then $\psi_-(t)=\psi_-(s)$, as we wanted.
	
	We now assume that $\eta_1(t)$ is simultaneously a double point of both $\eta_1$ and $\eta_2$. Then from Assertion~\eqref{item-counterflow-good2} in \cref{prop-counterflow-dichotomy}, we have that at least three of the four points $(t,\psi_-(t))$, $(t,\psi_+(t))$, $(s,\psi_-(t))$, $(s,\psi_+(t))$ are contained in $\op{supp} \perm$.
	
	Hence, in both cases, there exists $q\in [0,1]$ such that $(s,q) , (t,q) \in \op{supp} \perm$. This completes the proof of Case 2 because, thanks to \cref{lem:triple_are_simple}, there are no merge points of $\eta_1$ that are hit more than once by $\eta_2$, or vice versa; and, moreover, both $\eta_1$ and $\eta_2$ have no $m$-tuple points when $\kappa\geq 8$.
	
	\medskip
	
	\noindent\underline{Case 3:} Assume that $\kappa \in (4,8)$ and $\theta \in (-\pi/2,\pi/2)$. Also in this case, we prove that the set $\mcl{T}\mcl{M}_1=\left\{ (s,t) \in [0,1]^2 : \eta_1(s) =\eta_1(t) \right\}$  is a.s.\ determined by $\op{supp} \perm$, by proving an \emph{if and only if} condition similar to the one in \eqref{eq:iif-cond}; but with a modified (and more sophisticated) right-hand side.

	Fix $(s,t) \in [0,1]^2$ such that $\eta_1(s) =\eta_1(t)=z$. Then, from \cref{prop-counterflow-dichotomy2}, we have that almost surely,
		\begin{itemize}
			\item if for all $m_1\geq 2$ and $m_2\geq 2$, $z$ is not simultaneously a $m_1$-tuple point of $\eta_1$ and a $m_2$-tuple point of $\eta_2$ , then $(t,\psi_{-}(t))\in \op{supp} \perm$ and $(s,\psi_{-}(t))\in \op{supp} \perm$.
			\item while, if for some $m_1\geq 2$ and $m_2\geq 2$, $z$ is simultaneously an $m_1$-tuple point of $\eta_1$ and an $m_2$-tuple point of $\eta_2$, then the following fact is true. Let $t^1_1,t^1_2,\dots,t^1_{m_1}$ be the $m_1$ times when $\eta_1$ hits $z$ and $t^2_1,t^2_2,\dots,t^2_{m_2}$ be the $m_2$ times when $\eta_2$ hits $z$. Then, since  $\eta_1(s) =\eta_1(t)=z$, there exists a sequence of times $(t^1_{i_\ell}, t^2_{j_\ell})_{\ell=0}^M$ such that 
			\begin{itemize}
				\item $t^1_{i_0}=s$ and $t^1_{i_M}=t$;
				\item for every $\ell\in\{1,2,\dots,M\}$, $(t^1_{i_\ell}, t^2_{j_\ell}) \in \op{supp} \perm $;
				\item for every $\ell\in\{2,\dots,M\}$, either $t^1_{i_\ell}=t^1_{i_{\ell-1}}$ or $t^2_{j_\ell}=t^2_{j_{\ell-1}}$.
			\end{itemize} 
		\end{itemize}	
	Hence, in order to prove that  $\mcl{T}\mcl{M}_1=\left\{ (s,t) \in [0,1]^2 : \eta_1(s) =\eta_1(t) \right\}$  is a.s.\ determined by $\op{supp} \perm$, it is enough to prove that if either there exists $q\in [0,1]$  such that  $(s,q) , (t,q) \in \op{supp} \perm$, or if there exists a sequence $(t^1_{i_\ell}, t^2_{j_\ell})_{\ell=0}^M$ with the properties listed above, then $\eta_1(s) =\eta_1(t)$. The former case  follows again from \cref{lem-closure-to-intersect} and the fact that $\op{supp} \perm \subset \ol{\{ (t,\psi_-(t)) : t\in [0,1]\}}$.
	
	For the latter case, it suffices to note that if $(t^1_{i_\ell}, t^2_{j_\ell}), (t^1_{i_{\ell-1}}, t^2_{j_{\ell-1}}) \in \op{supp} \perm $ and $t^2_{j_\ell}=t^2_{j_{\ell-1}}$, then $\eta_1(t^1_{i_\ell}) =\eta_1(t^1_{i_{\ell-1}})$, and similarly if $(t^1_{i_\ell}, t^2_{j_\ell}), (t^1_{i_{\ell-1}}, t^2_{j_{\ell-1}}) \in \op{supp} \perm $ and $t^1_{i_\ell}=t^1_{i_{\ell-1}}$ then $\eta_2(t^2_{j_\ell}) =\eta_2(t^2_{j_{\ell-1}})$. (This claim follows once again from 
	\cref{lem-closure-to-intersect} and the fact that 
	$\op{supp} \perm \subset \ol{\{ (t,\psi_-(t)) : t\in [0,1]\}}.$) 
	This fact, combined with the simple observation that if $(t^1_{i_\ell}, t^2_{j_\ell}) \in \op{supp} \perm $ then $\eta_2(t^2_{j_\ell})=\eta_1(t^1_{i_\ell})$, allows us to conclude that $\eta_1(s) =\eta_1(t)$. This ends the proof of Case 3 and of the entire theorem.
\end{proof}

\subsection{Proofs of Propositions~\ref{prop-counterflow-dichotomy}~and~\ref{prop-counterflow-dichotomy2}}\label{sect:proof--props}

We now turn to the proof of Propositions~\ref{prop-counterflow-dichotomy}~and~\ref{prop-counterflow-dichotomy2}. The rest of this section is organized as follow. In \cref{Sect:prem-1}, we develop some preliminary results in the case when $\kappa\geq 8$, and then we prove  Proposition~\ref{prop-counterflow-dichotomy} in \cref{sect:kappa-geq-8}. Next, in \cref{sect:m-tuples,sect:eight-flow}  we establish preliminary results for the case when $\kappa\in (4,8)$, and then we prove Proposition~\ref{prop-counterflow-dichotomy2} in \cref{sect:proof-remaining}. In particular, \cref{sect:m-tuples,sect:eight-flow} will investigate the structure of $m$-tuple points of space-filling whole-plane SLEs when $\kappa\in (4,8)$; this is the most technical part of the paper. 


We point out that the main ideas of why the support of the permuton permuton $\perm$ determines the intersection set $\mcl{T}\mcl{M}(\eta_1)$, even though the support does not contain all of the pairs $(s,t)$ for which $\eta_1(s) = \eta_1(t)$, are already contained in \cref{sect:kappa-geq-8}. The arguments in \cref{sect:m-tuples,sect:eight-flow} are a (rather complicated) generalization of the arguments in \cref{sect:kappa-geq-8}. Therefore, we invite the reader to skip \cref{sect:m-tuples,sect:eight-flow} on a first read.

\subsubsection{Preliminary results when $\kappa\geq 8$}\label{Sect:prem-1}

Recall that $\wh h$ is the whole-plane GFF viewed modulo a global additive multiple of $2\pi\chi$ which is used to construct $\eta_1$ and $\eta_2$.  
For $i \in \{1,2\}$ and $z\in\BB C$,  recall that  $\beta_{z,i}^L$ and $\beta_{z,i}^R$ are the left and right outer boundaries of $\eta_i$ stopped when it hits $z$. 
Equivalently, $\beta_{z,1}^L$ and $\beta_{z,1}^R$ are the flow lines of $\wh h$ of angles $\pi/2$ and $-\pi/2$ started from $z$; and $\beta_{z,1}^L$ and $\beta_{z,2}^R$ are the flow lines of $\wh h$ of angles $\theta$ and $\theta-\pi$ started from $z$. 
The source of the condition~\eqref{eqn-angle-condition} is the following lemma, which comes from known results on SLEs. See Figure~\ref{fig-angle-condition} for an illustration.

\begin{figure}[ht!]
\begin{center}
\includegraphics[width=\textwidth]{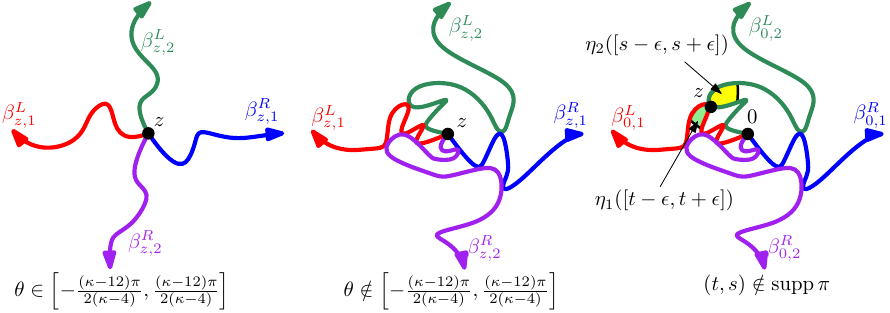}  
\caption{\label{fig-angle-condition} \textbf{Left and middle:} Illustration of the flow lines which form the left and right outer boundaries of $\eta_1$ and $\eta_2$ stopped upon hitting $z$ when~\eqref{eqn-angle-condition} holds (left) and when~\eqref{eqn-angle-condition} does not hold (middle). By Lemma~\ref{lem-angle-condition}, the four flow lines intersect only at $z$ if and only if~\eqref{eqn-angle-condition} holds. 
Note that since $\kappa \geq 8$, the flow lines $\beta^L_{z,i}$ and $\beta^R_{z,i}$ do not intersect each other.
\textbf{Right:} Illustration of the proof of Lemma~\ref{lem-counterflow-bad}. If $z \in \beta_{0,1}^L \cap \beta_{0,2}^L \setminus \{0\}$, then there are times $t$ and $s$ such that $\eta_1(t) = \eta_2(s)$ but the set $\eta_1([t-\ep,t+\ep]) \cap \eta_2([s-\ep,s+\ep])$ has zero $\mu_h$-mass for each small enough $\ep > 0$. This shows that  $(t,s) \notin \op{supp} \perm$. 
}
\end{center}
\vspace{-3ex}
\end{figure}

\begin{lem}\label{lem-angle-condition}
If~\eqref{eqn-angle-condition} holds, then for each fixed $z\in\BB C$ a.s.\ the four flow lines $\beta_{z,1}^L,\beta_{z,1}^R,\beta_{z,2}^L$, and $\beta_{z,2}^R$ intersect only at the point $z$. Conversely, if~\eqref{eqn-angle-condition} does not hold, then a.s.\ the intersection of $\beta_{z,2}^L$ with at least one of $\beta_{z,1}^L$ or $\beta_{z,1}^R$ is uncountable; and the same is true for $\beta_{z,2}^R$ in place of $\beta_{z,2}^L$.
\end{lem}
\begin{proof}
By~\cite[Proposition 3.28]{ig4}, the conditional law of $\beta_{z,2}^L$ given $\beta_{z,1}^L$ and $\beta_{z,2}^R$ is that of a SLE$_{16/\kappa}(\rho^L;\rho^R)$ curve from 0 to $\infty$ in the appropriate connected component of $\BB C\setminus (\beta_{z,1}^L \cup \beta_{z,1}^R)$, with force points immediately to the left and right of the origin, where
\eqb
\rho^L = \frac{(\pi/2-\theta) \chi}{\lambda} - 2 \quad\text{and} \quad \rho^R = \frac{(\theta + \pi/2) \chi}{\lambda} - 2 ,
\eqe
where we recall from \eqref{eq:parameters} that $\lambda=\frac{\pi\sqrt{\kappa}}{4}$ and $ \chi=\frac{\sqrt{\kappa}}{2}-\frac{2}{\sqrt{\kappa}}$.
On the other hand, by~\cite[Theorem 1.6]{miller-wu-dim}, such an SLE$_{16/\kappa}(\rho^L;\rho^R)$ curve hits the boundary of its domain if and only if at least one of $\rho^L$ or $\rho^R$ is less than $8/\kappa-2$, in which case the intersection of the curve with the boundary is uncountable. We obtain the statement of the lemma by solving for the values of $\theta$ for which this is the case.
\end{proof}

We can now prove a result useful for the first part of Assertion~\eqref{item-counterflow-good2} of Proposition~\ref{prop-counterflow-dichotomy}. 

\begin{lem} \label{lem-counterflow-bad}
Assume that $\kappa>4$ and $\theta\in(-\pi/2,\pi/2)$ are such that~\eqref{eqn-angle-condition} does not hold. Almost surely, there are uncountably many points of 
\eqbn
		\ol{\{(t,\psi_{-}(t)) : t\in [0,1]\} } \quad\text{ and }\quad \ol{\{(t,\psi_{+}(t)) : t\in [0,1]\} }
\eqen
which do not belong to $\op{supp}\perm$.
\end{lem}
\begin{proof}
See the right-hand side of Figure~\ref{fig-angle-condition} for an illustration. 
By Lemma~\ref{lem-angle-condition}, if~\eqref{eqn-angle-condition} does not hold, then a.s.\ the flow line $\beta_{0,2}^L$ intersects at least one of $\beta_{0,1}^L$ or $\beta_{0,1}^R$ in an uncountable set. For simplicity, we work on the event that $\beta_{0,1}^L \cap \beta_{0,2}^L$ is uncountable (the other case is treated similarly). 

Since $\beta_{0,2}^L$ and $\beta_{0,2}^R$ are conditionally independent given $\beta_{0,1}^L$ and $\beta_{0,1}^R$, on the event that $\beta_{0,1}^L \cap \beta_{0,2}^L$ is uncountable a.s.\ also $(\beta_{0,1}^L \cap \beta_{0,2}^L) \setminus \beta_{0,2}^R$ is uncountable. 
Let $z\in ( \beta_{0,1}^L \cap \beta_{0,2}^L )\setminus   \beta_{0,2}^R$. Let $t$ be the first time at which $\eta_1$ hits $z$ and let $s=\psi_+(t)$ be the last time at which $\eta_2$ hits $z$. 
To prove the lemma statement it suffices to show that $(t,s) \notin \op{supp}\perm$. (Note that the claim that there are also uncountably many points of $\ol{\{(t,\psi_{-}(t)) : t\in [0,1]\} }$ which do not belong to $\op{supp}\perm$ then follows by time reversal symmetry.)

Since $z\not=0$ and $\eta_1$ is continuous, it holds for small enough $\ep  > 0$ that $\eta_1([t-\ep , t+\ep])$ is contained in the region which is hit by $\eta_1$ before it hits 0, i.e., the region lying to the left of $\beta_{0,1}^L$ and to the right of $\beta_{0,2}^L$ (recall the direction of $\beta_{0,1}^L$ and $\beta_{0,2}^L$ from the arrows in Figure~\ref{fig-angle-condition}). 
Furthermore, since $z\notin \beta_{0,2}^R$, it holds for each small enough $\ep > 0$ that $\eta_1([t-\ep,t+\ep])$ is disjoint from $\beta_{0,2}^R$. 
Hence, for each small enough $\ep > 0$, $\eta_1([t-\ep,t+\ep])$ is contained in the closure of the region bounded by the left side of $\beta_{0,1}^L$ and the right side of $\beta_{0,2}^R$. 

Similarly, for each small enough $\ep > 0$, $\eta_2([s-\ep,s+\ep])$ is contained in the closure of the region bounded by the right side of $\beta_{0,2}^L$ and the left side of $\beta_{0,1}^R$. This region and the region in the preceding paragraph intersect only along their boundaries, which have zero $\mu_h$-mass. Hence, for each small enough $\ep > 0$, $\eta_1([t-\ep,t+\ep]) \cap \eta_2([s-\ep,s+\ep])$ has zero $\mu_h$-mass. By Lemma~\ref{lem-permuton-defined}, Assertion 1, this implies that $(t,s)$ is not in the support of $\perm$.
\end{proof}

We now turn our attention to the proof of Assertion~\eqref{item-counterflow-good} of Proposition~\ref{prop-counterflow-dichotomy}. For $i\in \{1,2\}$ and $z\in\BB C$, let $\tau_z^i$ be the first time at which $\eta_i$ hits $z$.

\begin{figure}[ht!]
	\begin{center}
		\includegraphics[width=0.6\textwidth]{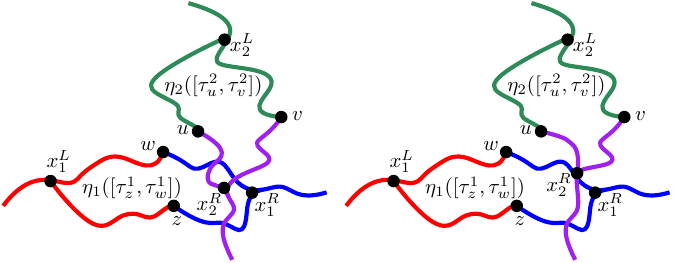}  
		\caption{\label{fig-counterflow-intersect} Illustration of two possible scenarios where $\eta_1([\tau_z^1,\tau_w^1]) \cap \eta_2([\tau_u^2,\tau_v^2]) \not=\emptyset$. It follows from~\cref{lem:flow_lines_crossing} that whenever one of the four flow lines which form the boundary of $\eta_1([\tau_z^1,\tau_w^1])$ intersects one of the four flow lines which form the boundary of $\eta_2([\tau_u^2,\tau_v^2])$, these two flow lines must cross. Hence, either $\eta_1([\tau_z^1,\tau_w^1]) \cap \eta_2([\tau_u^2,\tau_v^2])$ contains a non-empty open set (left panel); or the only intersection points of $\eta_1([\tau_z^1,\tau_w^1]) $ and $ \eta_2([\tau_u^2,\tau_v^2])$ are among the points $x_1^L,x_1^R,x_2^L,x_2^R$ where the red, blue, green, and purple (respectively) flow lines merge (right panel). These four points are merge points for $\eta_1$ or $\eta_2$, respectively, and so \cref{lem:triple_are_simple} shows that this second scenario has probability zero.
		}
	\end{center}
	\vspace{-3ex}
\end{figure}

\begin{lem} \label{lem-counterflow-intersect}
	Assume that~\eqref{eqn-angle-condition} holds. Fix distinct points $z,w,u,v\in\BB C$. Almost surely, if $\eta_1([\tau_z^1,\tau_w^1]) \cap \eta_2([\tau_u^2,\tau_v^2]) \not=\emptyset$, then $\eta_1([\tau_z^1,\tau_w^1]) \cap \eta_2([\tau_u^2,\tau_v^2])$ contains a non-empty open set.
\end{lem}

\begin{proof}
	See Figure~\ref{fig-counterflow-intersect} for an illustration of the statement and proof. 
	Throughout the proof, we work on the event $\eta_1([\tau_z^1,\tau_w^1]) \cap \eta_1([\tau_u^2,\tau_v^2]) \not=\emptyset$. Note that this event implies, in particular, that $\tau_z^1 < \tau_w^1$ and $\tau_u^2 < \tau_v^2$. 
	We need to show that a.s.\ under these assumptions,
	\eqb \label{eqn-counterflow-intersect-show}
	\text{$\eta_1([\tau_z^1,\tau_w^1]) \cap \eta_2([\tau_u^2,\tau_v^2])$ contains a non-empty open set.}
	\eqe
	
	From the definition of space-filling SLE, we see that a.s.\ $\eta_1([\tau_z^1-\ep,\tau_z^1+\ep])$ contains a neighborhood of $z$ for each $\ep > 0$. Similar statements hold with $w,u,v$ in place of $z$ and/or with $\eta_2$ in place of $\eta_1$. Therefore, the lemma statement is true automatically if either of $\tau_u^1$ or $\tau_v^1$ belongs to $[\tau_z^1,\tau_w^1]$; or if either of $\tau_z^2$ or $\tau_w^2$ belongs to $[\tau_u^2,\tau_v^2]$. Hence, we can assume without loss of generality that $u,v \notin \eta_1([\tau_z^1,\tau_w^1])$ and $z,w \notin \eta_2([\tau_u^2,\tau_v^2])$. This implies in particular that neither $\eta_1([\tau_z^1,\tau_w^1]) $ nor $ \eta_2([\tau_u^2,\tau_v^2])$ is a subset of the other. Hence, our assumption that  $\eta_1([\tau_z^1,\tau_w^1]) \cap \eta_2([\tau_u^2,\tau_v^2]) \not=\emptyset$ implies that in fact
	\eqb \label{eqn-counterflow-intersect-bdy}
	\bdy \eta_1([\tau_z^1,\tau_w^1]) \cap \bdy \eta_2([\tau_u^2,\tau_v^2]) \setminus \{z,w,u,v\} \not=\emptyset .
	\eqe
	
	Let $x_1^L$ be the point where the flow lines $\beta_{z,1}^L$ and $\beta_{w,1}^L$ merge. 
	Define the point $x_1^R \in \bdy \eta_1([\tau_z^1,\tau_w^1])$ and the points $x_2^L,x_2^R \in \bdy \eta_2([\tau_u^2,\tau_v^2])$ analogously to the point $x_1^L$ above. Then $\bdy\eta_1([\tau_z^1,\tau_w^1])$ is the union of the segment of $\beta_{z,1}^L$ from $z$ to $x_1^L$, the segment of $\beta_{w,1}^L$ from $w$ to $x_1^L$, the segment of $\beta_{z,1}^R$ from $z$ to $x_1^R$, the segment of $\beta_{w,1}^R$ from $w$ to $x_1^R$. Moreover, an analogous statement holds for $ \bdy \eta_2([\tau_u^2,\tau_v^2])$. 
	
	Since by \cref{lem:triple_are_simple}, the points $x_1^L$ and $x_1^R$ cannot be on the boundary $\bdy \eta_2([\tau_u^2,\tau_v^2])$, and the points $x_2^L$ and $x_2^R$ cannot be on the boundary $\bdy \eta_1([\tau_z^1,\tau_w^1])$, it follows from~\eqref{eqn-counterflow-intersect-bdy} that there is a point 
	\eqb
	y \in  \bdy \eta_1([\tau_z^1,\tau_w^1]) \cap \bdy \eta_2([\tau_u^2,\tau_v^2])  \setminus \{ z,w,u,v,x_1^L,x_1^R,x_2^L,x_2^R\}  .
	\eqe
	Assume without loss of generality that $y \in \beta_{w,1}^R \cap \beta_{u,2}^R$ (the other possibilities are treated similarly). 
	By the preceding paragraph, $y$ must be in the segment of $\beta_{w,1}^R$ from $w$ to $x_1^R$ and in the segment of $\beta_{u,2}^R$ from $u$ to $x_2^R$. 
	By~\cref{lem:flow_lines_crossing}, a.s. the flow lines $\beta_{w,1}^R$ and $\beta_{u,2}^R$ cross when they meet, so, in particular, they cross at the point $y$. Since $y\notin \{ z,w,u,v, x_1^L,x_1^R,x_2^L,x_2^R\}$, this implies that the segments of $\beta_{w,1}^R$ and $\beta_{u,2}^R$ which are part of the boundaries of $\eta_1([\tau_z^1,\tau_w^1])$ and $ \eta_2([\tau_u^2,\tau_v^2])$ cross each other. 
	
	By~\eqref{eqn-angle-condition}, we have $\kappa\geq 12$, so a.s.\ $\eta_1([\tau_z^1,\tau_w^1])$ and $ \eta_2([\tau_u^2,\tau_v^2])$ are the closures of Jordan domains. By the previous paragraph, under our assumptions a.s.\ the boundary curves of these Jordan domains cross each other. Hence, topological considerations show that a.s.~\eqref{eqn-counterflow-intersect-bdy} holds, as required.
\end{proof}

\subsubsection{Proof of Proposition~\ref{prop-counterflow-dichotomy}}\label{sect:kappa-geq-8}

\begin{proof}[Proof of Proposition~\ref{prop-counterflow-dichotomy}]
		We first prove Assertion~\eqref{item-counterflow-good}. Recall that from Assertion 2 in Lemma~\ref{lem-permuton-defined} we have that  
		\eqbn
		\op{supp} \perm \subset \ol{\{(t,\psi_{-}(t)) : t\in [0,1]\} }\subset \{(t,s) \in [0,1] : \eta_1(t) = \eta_2(s) \}  ,
		\eqen 
		and so, it is enough to prove that  
		\eqbn
		\{(t,s) \in [0,1] : \eta_1(t) = \eta_2(s) \}
		\subset
		\op{supp} \perm .
		\eqen
		Fix $(t,s) \in [0,1]$ such that $\eta_1(t) = \eta_2(s)$.	A.s.\ the preimage of $\BB Q^2$ under $\eta_1$ is dense in $[0,1]^2$, and the same is true for $\eta_2$ (see, e.g.,~\cite[Caption to Figure 4.10]{ig4}).
		Hence, a.s.\ there exists distinct points $z,w,u,v\in\BB Q^2$ such that $t \in [\tau_z^1,\tau_w^1]$, $s\in [\tau_u^2,\tau_v^2]$, and the rectangle $[\tau_z^1,\tau_w^1] \times [\tau_u^2,\tau_v^2]$ is contained in an arbitrarily small neighborhood of $(t,s)$. Since $\eta_1(t) =\eta_2(s)$, we have $\eta_1([\tau_z^1,\tau_w^1]) \cap \eta_2([\tau_u^2,\tau_v^2]) \not=\emptyset$.
		
		
		By Lemma~\ref{lem-counterflow-intersect}, 
		a.s.\ the region $\eta_1([\tau_z^1,\tau_w^1]) \cap \eta_2([\tau_u^2,\tau_v^2])$ contains a non-empty open set,  and hence has positive $\mu_h$-mass.
		By Lemma~\ref{lem-permuton-defined}, Assertion 1, this implies that $\perm\left( [\tau_z^1,\tau_w^1] \times [\tau_u^2,\tau_v^2]  \right) > 0$.  Hence $(t,s) \in \op{supp} \perm$.
		
%
		
		\medskip
		
		We now prove Assertion~\eqref{item-counterflow-good2}. The first part of the statement follows from Lemma~\ref{lem-counterflow-bad}.
		We now prove that if $z\in \BB C$ is not simultaneously a double point of both $\eta_1$ and $\eta_2$ and $t\in[0,1]$ is a time when $\eta_1$ hits $z$, then $(t,\psi_{-}(t))\in \op{supp} \perm$.
		Note that if $\eta_1(t)$ is a simple point of $\eta_1$ or a simple point of $\eta_2$, then $(t,\psi_{-}(t))\in \op{supp} \perm$ by Assertion 3 in Lemma~\ref{lem-permuton-defined}. Similarly, if $\eta_1(t)$ is a triple point of $\eta_1$ (and so it must be a merge point because $\kappa\geq 8$), then $\eta_1(t)$ is a simple point of $\eta_2$ thanks to Lemma~\ref{lem:triple_are_simple}, and so we conclude as before. The same argument works if $\eta_1(t)$ is a triple point of $\eta_2$. 
		
		It remains to prove that if $z$ is simultaneously a double point of both $\eta_1$ and $\eta_2$, then $\op{supp} \perm$ contains at least three of the four pairs $(t,s)\in[0,1]^2$ for which $\eta_1(t) = \eta_2(s) = z$.
		Every double point of $\eta_1$ (resp.\ $\eta_2$) lies on the flow line  $\beta_{u,1}^L$ or $\beta_{u,1}^R$ (resp.\ $\beta_{v,2}^L$ or $\beta_{v,2}^R$) for some $u\in\BB Q$ (resp.\ $v\in\BB Q$). Let $\eta_1(t)=z$ be a double point of both $\eta_1$ and $\eta_2$. Let $u\in\BB C$ and $v\in\BB C$ be such that $z\neq u$, $z\neq v$, and assume that
		\begin{itemize}
			\item (Case 1) $z\in\beta_{u,1}^L$ and $z\in\beta_{v,2}^L$;
			\item (Case 2) $z\in\beta_{u,1}^L$ and $z\in\beta_{v,2}^R$;
			\item (Case 3) $z\in\beta_{u,1}^R$ and $z\in\beta_{v,2}^L$;
			\item (Case 4) $z\in\beta_{u,1}^R$ and $z\in\beta_{v,2}^R$.
		\end{itemize}  
		Let $t_1$ (resp.\ $t_2$) be the first (resp.\ second) time when $\eta_1$ hits $z$. Similarly, let $s_1$ (resp.\ $s_2$) be the first (resp.\ second) time when $\eta_2$ hits $z$.

		\begin{figure}[ht!]
			\begin{center}
				\includegraphics[width=1\textwidth]{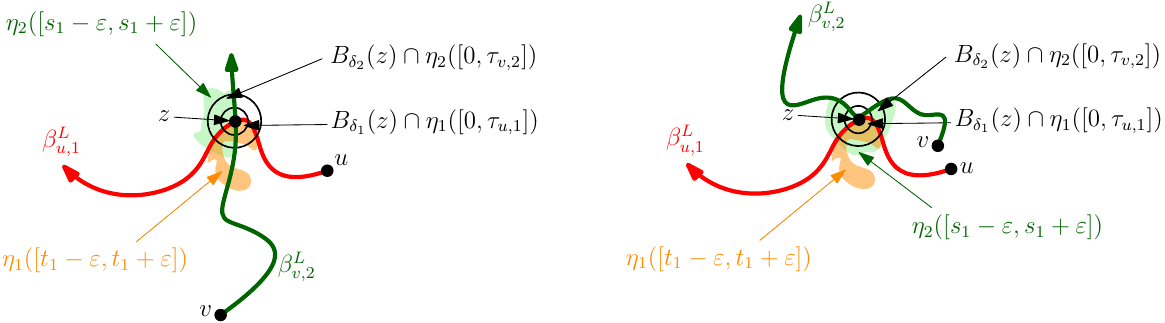}  
				\caption{\label{fig-int-double-points} The two subcases in the proof of Assertion~\eqref{item-counterflow-good2} of Proposition~\ref{prop-counterflow-dichotomy} when $z$ is a double point for both $\eta_1$ and $\eta_2$. \textbf{Left:} Here $u\neq v$ and the curve $\beta_{v,2}^L$ hits $\beta_{u,1}^L$  on its left-hand side and so the two flow lines $\beta_{u,1}^L$ and $\beta_{v,2}^L$ cross each other at $z$. \textbf{Right:} Here $u\neq v$ and the curve $\beta_{v,2}^L$ hits $\beta_{u,1}^L$  on its right-hand side and so the two flow lines $\beta_{u,1}^L$ and $\beta_{v,2}^L$ do not cross each other at $z$.
				}
			\end{center}
			\vspace{-3ex}
		\end{figure} 
		
		\medskip
	
		\underline{Case 1.} We  begin with the analysis of Case 1. There are two subcases (recall \cref{lem:flow_lines_crossing}):
		\begin{itemize}
			\item If $u\neq v$ and the curve $\beta_{v,2}^L$ hits $\beta_{u,1}^L$  on its left-hand side, then the curve $\beta_{v,2}^L$ hits $\beta_{u,1}^L$ exactly once at $z$, at which point it crosses $\beta_{u,1}^L$ (\cref{lem:flow_lines_crossing}). See the left-hand side of \cref{fig-int-double-points}. Fix $\ep>0$ such that $t_1+\ep<\tau_{u,1}$ and $s_1+\ep<\tau_{v,2}$. By \cref{lem:SLE_hitting double points} there exists $\delta_1>0$ and $\delta_2>0$ such that 
			\begin{equation*}
				B_{\delta_1}(z) \cap \eta_1([0,\tau_{u,1}])\subset \eta_1([t_1-\ep,t_1+\ep]) \quad\text{and}\quad B_{\delta_2}(z) \cap \eta_2([0,\tau_{v,2}])\subset \eta_2([s_1-\ep,s_1+\ep]).
			\end{equation*}
			Since the curve $\beta_{u,1}^L$ crosses $\beta_{v,2}^L$ when it hits $z$, we can conclude that 
			$$B_{\delta_1}(z) \cap \eta_1([0,\tau_{u,1}]) \cap B_{\delta_2}(z) \cap \eta_2([0,\tau_{v,2}])$$ 
			contains an open non-empty set, and so also $\eta_1([t_1-\ep,t_1+\ep])\cap \eta_2([s_1-\ep,s_1+\ep])$ contains an open non-empty set. Hence $\perm([t_1-\ep,t_1+\ep]\times[s_1-\ep,s_1])>0$. Since $\ep>0$ is arbitrary, we conclude that $(t_1,s_1)\in\op{supp} \perm$. Similar arguments also show  that $(t_1,s_2)\in\op{supp} \perm$, $(t_2,s_1)\in\op{supp} \perm$, and $(t_2,s_2)\in\op{supp} \perm$.
			
			\item If $u\neq v$ and the curve $\beta_{v,2}^L$ hits $\beta_{u,1}^L$  on its right-hand side, or if $u=v$, then  $\beta_{u,1}^L$ and $\beta_{v,2}^L$ do not cross each other (\cref{lem:flow_lines_crossing}). We focus on the case when $u\neq v$, the other case is identical. See the right-hand side of \cref{fig-int-double-points}. Fix $\ep>0$ such that $t_1+\ep<\tau_{u,1}$ and $s_1+\ep<\tau_{v,2}$. As before, by \cref{lem:SLE_hitting double points} there exists $\delta_1>0$ and $\delta_2>0$ such that 
			\begin{equation}\label{eq:first_incl}
				B_{\delta_1}(z) \cap \eta_1([0,\tau_{u,1}])\subset \eta_1([t_1-\ep,t_1+\ep])
			\end{equation} 
			and 
			\begin{equation}\label{eq:second_incl}
				B_{\delta_2}(z) \cap \eta_2([0,\tau_{v,2}])\subset \eta_2([s_1-\ep,s_1+\ep]).
			\end{equation}
			Note that if $\delta_1\leq\delta_2$ then 
			\begin{multline}\label{eq:third_incl}
				B_{\delta_1}(z) \cap \eta_1([0,\tau_{u,1}])
				\stackrel{\eqref{eq:first_incl}}{\subset}
				B_{\delta_1}(z) \cap \eta_1([t_1-\ep,t_1+\ep])\\
				\stackrel{\delta_1\leq\delta_2}{\subset}
				B_{\delta_2}(z) \cap \eta_1([t_1-\ep,t_1+\ep])
				\subset
				B_{\delta_2}(z) \cap \eta_2([0,\tau_{v,2}])
				\stackrel{\eqref{eq:second_incl}}{\subset}
				\eta_2([s_1-\ep,s_1+\ep])
			\end{multline}
			where for the second-to-last inclusion we used the fact that $	B_{\delta_2}(z)\cap\eta_1([t_1-\ep,t_1+\ep])\subset \eta_2([0,\tau_{v,2}])$ since $\beta_{v,2}^L$ hits $\beta_{u,1}^L$ on its right-hand side.
			Since $B_{\delta_1}(z) \cap \eta_1([0,\tau_{u,1}])$ contains an open non-empty set, also $\eta_1([t_1-\ep,t_1+\ep])\cap \eta_2([s_1-\ep,s_1+\ep])$ contains an open non-empty set thanks to \eqref{eq:first_incl} and \eqref{eq:third_incl}. Hence $\perm([t-\ep,t+\ep]\times[s_1-\ep,s_1])>0$. Since $\ep>0$ is arbitrary, we conclude that $(t_1,s_1)\in \op{supp} \perm$.
			
			Otherwise, if $\delta_2<\delta_1$, with similar arguments we get that 
			\begin{align*}
				B_{\delta_2}(z) \cap \eta_1([0,\tau_{u,1}])&\subset B_{\delta_1}(z) \cap \eta_1([0,\tau_{u,1}])\stackrel{\eqref{eq:first_incl}}{\subset} \eta_1([t-\ep,t+\ep])\\
				B_{\delta_2}(z) \cap \eta_1([0,\tau_{u,1}])&\subset B_{\delta_2}(z) \cap \eta_2([0,\tau_{v,2}])\stackrel{\eqref{eq:second_incl}}{\subset}
				\eta_2([\psi_-(t)-\ep,\psi_-(t)+\ep]).
			\end{align*}
			Since $B_{\delta_2}(z) \cap \eta_1([0,\tau_{u,1}])$ contains an open non-empty set, we can deduce from the last displayed equation that also $\eta_1([t-\ep,t+\ep])\cap \eta_2([\psi_-(t)-\ep,\psi_-(t)+\ep])$ contains an open non-empty set. As before, we conclude that $(t,\psi_-(t))\in \op{supp} \perm$. Similar arguments show also that  $(t_2,s_1)\in\op{supp} \perm$, and $(t_2,s_2)\in\op{supp} \perm$ (but note that $(t_1,s_2)\notin\op{supp} \perm$ as proved in \cref{lem-counterflow-bad}).
		\end{itemize}
	
		The proof of the remaining three cases is quite similar (and skipped for brevity), the only difference being in which point among the four points $(t_1,s_1)$, $(t_1,s_2)$, $(t_2,s_1)$, $(t_2,s_2)$ is excluded from the support in the second subcase above.
		This completes the proof of Assertion~\eqref{item-counterflow-good2}.
\end{proof}


\subsubsection{Preliminary results when $\kappa\in (4,8)$: the structure of $m$-tuple points for one space-filling SLE}\label{sect:m-tuples}

In this section and the next, we will develop a series of preliminary results that are essential for the proof of \cref{prop-counterflow-dichotomy2}. As a reminder, \cref{prop-counterflow-dichotomy2} addresses the case when $\kappa\in (4,8)$, which is the parameter range where space-filling SLEs can exhibit $m$-tuple points for $m\geq 3$.\footnote{For the reader's convenience, we recall that by definition, we do not consider a merge point to be a 3-tuple point, and vice versa.}

The proof of \cref{prop-counterflow-dichotomy2} is similar in spirit to the one of \cref{prop-counterflow-dichotomy}, but there is one additional complication: we need to understand how $m$-tuple points for one space-filling SLE $\eta_1$ are visited by the second space-filling SLE $\eta_2$ of a different angle. Therefore, in this section, we will start by describing the behavior of $\eta_1$ around one of its $m$-tuple points. Then, in  \cref{sect:eight-flow}, we will describe the behavior of $\eta_2$ around the same $m$-tuple point, and finally, in \cref{sect:proof-remaining}, we will complete the proof of \cref{prop-counterflow-dichotomy2}.

\medskip

The next lemma gives a precise description of the structure of $m$-tuple points of a whole-plane space-filling SLE$_\kappa$ when $\kappa\in (4,8)$. It is a rephrasing of the results stated and proved in~\cite[Lemma 8.13]{wedges} (see also ~\cite[Figure 8.1]{wedges}). We invite the reader to compare the statement of the next lemma with \cref{fig-m-tuple-points}. 

\begin{lem}\label{lem:m-tuple-points}
	Fix $\kappa\in (4,8)$ and let $\eta$ be a whole-plane space-filling SLE$_\kappa$. Fix $m\geq 2$ and let $z$ be a $m$-tuple point of $\eta$. Then the following results are almost surely true. 
	
	There exist $u,v \in \BB C$, $*,\circ\in\{L,R\}$ with $\{*,\circ\}=\{L,R\}$, and $m_u^*, m_v^*\in\BB Z_{\geq 1}$ such that
	\begin{enumerate}
		\item The flow line $\beta_{u}^*$ hits the point $z$ $m_u^*$ times. 
		\item The flow line $\beta_{u}^{\circ}$ hits $z$ $m_u^\circ$ times, with $m_u^\circ\in \{m_u^*,m_u^*-1\}$.
		\item The flow line $\beta_{v}^*$ hits the point $z$ $m_v^*$ times.
		\item The flow line $\beta_{v}^{\circ}$ hits $z$ $m_v^\circ$ times, with $m_v^\circ\in \{m_v^*,m_v^*-1\}$.
		\item $m_u^*+m_u^\circ+m_v^*+m_v^\circ=m$.
		\item The flow lines $\beta_{u}^*$ and $\beta_{v}^*$ merge before hitting $z$ for the last time, but after hitting $z$ for the second-to-last time. The curve $\beta_u^*\cap \beta_v^*$ after this merge point is called \textbf{king strand}.
		\item The flow lines $\beta_{u}^{\circ}$ and $\beta_{v}^{\circ}$ merge together after hitting $z$ for the last time.
	\end{enumerate}
	We now restrict ourselves to the case\footnote{Note that a similar statement holds for all the other cases that we omitted for brevity; see for instance the right-hand side of \cref{fig-m-tuple-points}.} that $*=R$ and $m_u^\circ=m_u^*-1$ and $m_v^\circ=m_v^*-1$, as on the left-hand side of \cref{fig-m-tuple-points}. We describe how $\eta$ fills the $m+1$ regions determined by $\beta_{u}^{*},\beta_{u}^{\circ},\beta_{v}^{*},\beta_{v}^{\circ}$.
	The curve $\eta$ fills each region before moving to the next one (and so never returns to the interior of an already filled region). Moreover, $\eta$ fills the $m+1$ regions in the order indicated\footnote{Note that \cref{fig-m-tuple-points} considers the case $*=R$. If $*=L$ then $\eta$ fills the $m+1$ regions in exact opposite order with respect to the one indicated in \cref{fig-m-tuple-points}.} in \cref{fig-m-tuple-points}, that is, $\eta$ first fills the region to the right of the king strand (with respect to the natural orientation of the king strand), then --- after hitting $z$ for the first time --- $\eta$ fills the only other region that contains the continuation of the two flow lines on the boundary of the previously filled region. The curve $\eta$ continues to fill the subsequent  regions iterating the same rule, until it reaches a region containing $u$ on its boundary. After filling this region\footnote{Note that $\eta$ hits $z$ one more time while filling this region, because $z$ is on the boundary of this region.}, $\eta$ hits $u$ and then fills the only other region containing $u$ on the boundary. Then $\eta$ continues to fill the 	subsequent regions iterating the same rule (with the same rule used for $u$ applied to the two regions containing $v$ on their boundaries). The last region filled by $\eta$ is the one on the left of the king strand.

	We denote by $U(x)$ the $x$th region filled by $\eta$, for all $x\in[m+1]:=[1,m+1]\cap\BB Z$.  As a consequence of the description in the last paragraph, $\eta$ hits $z$ exactly one time while traversing two consecutive regions, except when $\eta$ is moving between the regions $U(m^*_u)$ and $U(m^*_u+1)$ --- in this case $\eta$ hits $u$ --- and between the regions $ U(m^*_u+m^\circ_u+m^\circ_v+1)$ and $U(m^*_u+m^\circ_u+m^\circ_v+2)$ --- in this case $\eta$ hits $v$. Moreover, $\eta$ hits $z$ exactly two more times, called the \textbf{exceptional hitting times}: once when it is filling the region $U(m^*_u)$ and once when is filling the region $U(m_u^*+m_u^\circ+m^\circ_v+2)$.
	\begin{figure}[ht!]
		\begin{center}
			\includegraphics[scale=.55]{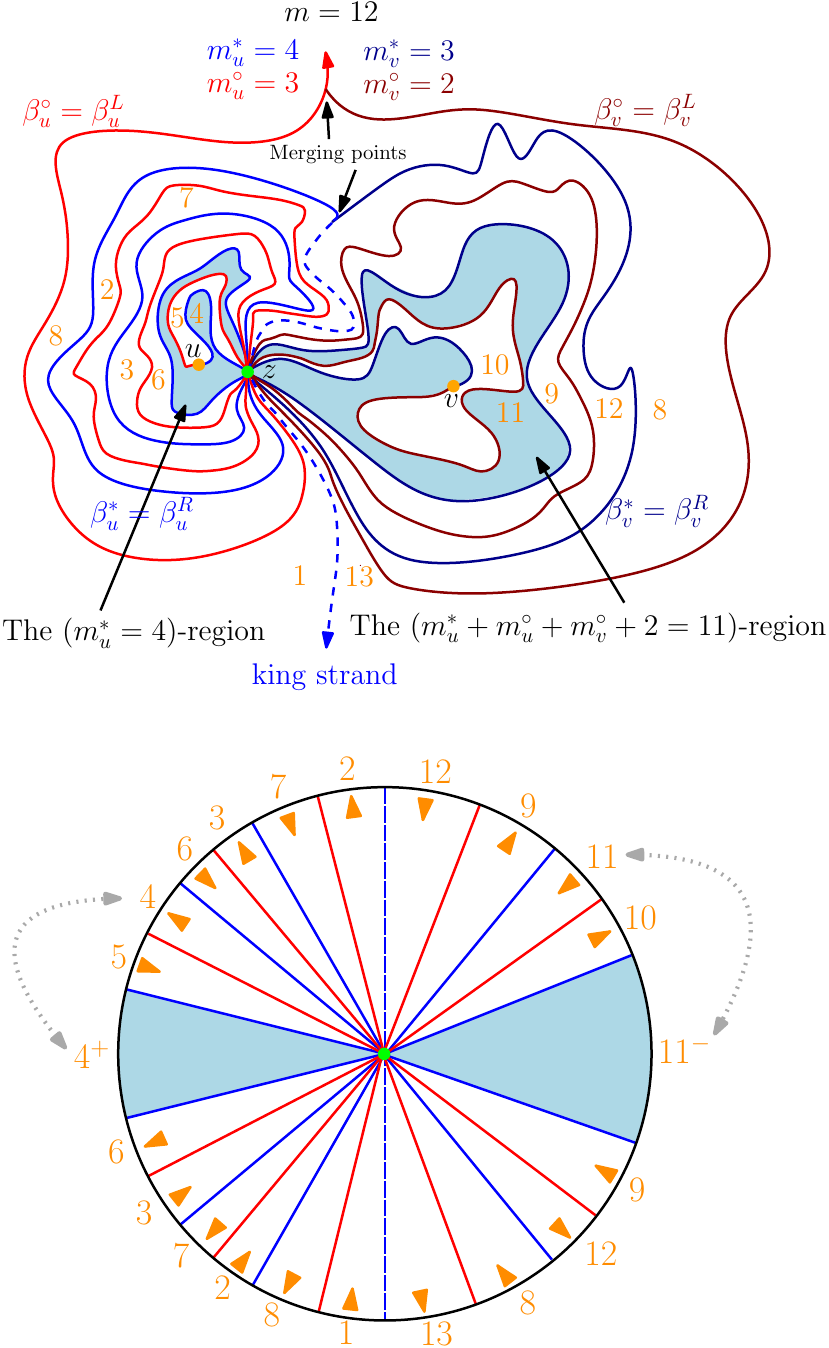}  
			\hspace{0.7cm}
			\includegraphics[scale=.55]{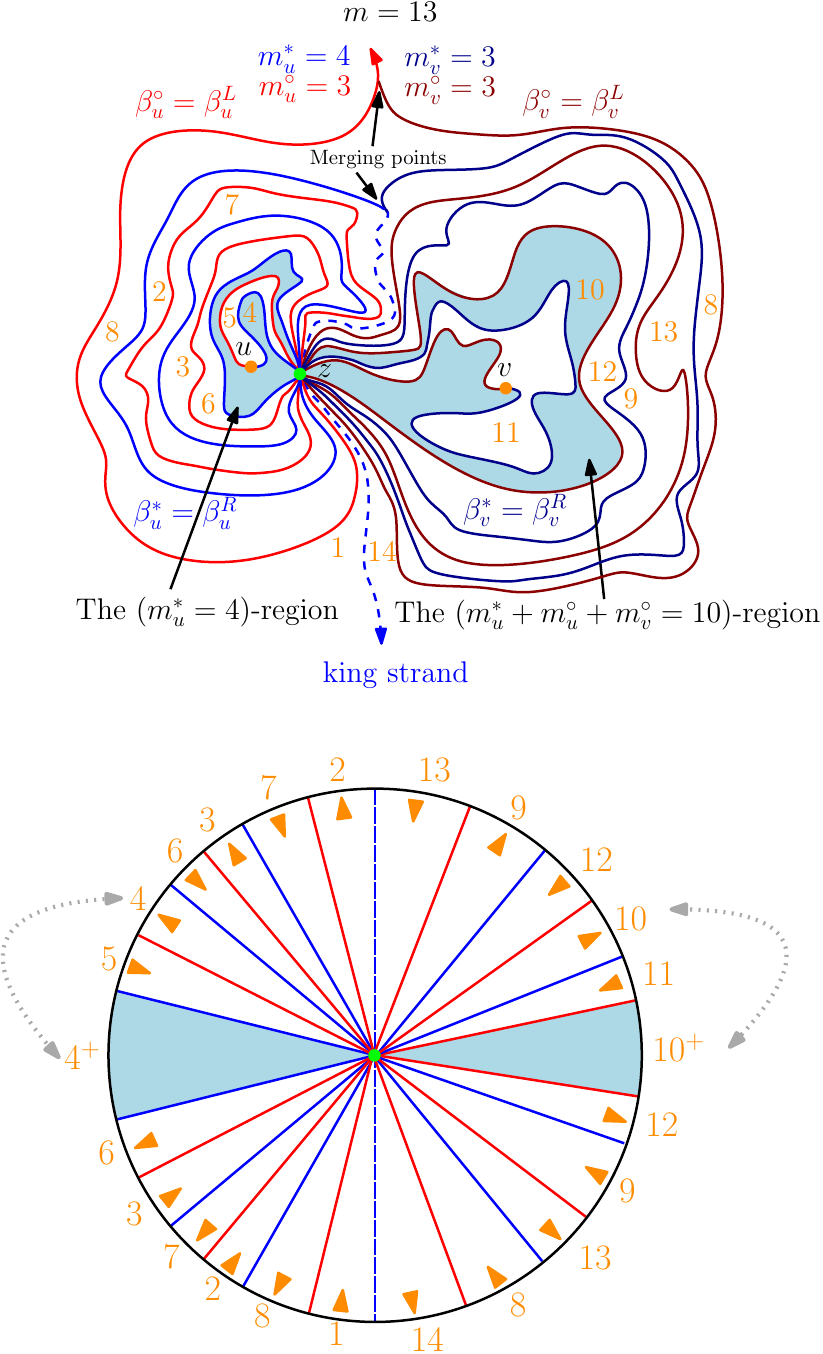}
			\caption{\label{fig-m-tuple-points} 
				\textbf{Top-Left:} The case of a $RL$-$m$-tuple point when $m^\circ_u=m^*_u-1$ and $m^\circ_v=m^*_v-1$. \textbf{Top-Right:} The case of a $RL$-$m$-tuple point when $m^\circ_u=m^*_u-1$ and $m^\circ_v=m^*_v$. \newline
				\textbf{Bottom left and right:} The local diagrams of $z$. Each diagram corresponds to the above picture. Each radius in the diagram records a piece of the flow lines locally around $z$, with respect to the order of appearance. The two radii corresponding to the king strand are dashed. The orange numbers keep track of the region numbers. The orange arrows record in which direction the SLE $\eta$ fills that region. In particular, we say that a segment is an inwards segment, if the arrow in the segment points in the direction of the middle green point, otherwise we will say that it is an outwards segment.\newline
				In each diagram, the two blue segments without arrows correspond to the instances when $\eta$ hits $z$ from the blue region without transitioning to a different region. These represent the two exceptional hitting times. Additionally, note that each blue segment has a corresponding partner segment in the diagram as indicated by the dotted double arrows. This partner segment corresponds to the moment when $\eta$ hits $z$ again from the blue region but immediately after switches to another region.
				We label the blue segments with an additional plus (resp.\ minus) if the partner segment is an outwards segment (resp.\ inwards segment). This will help in recalling which segment is hit first. \newline
				We remark that in both pictures at the top the various flow lines hit each other and themselves at many other points different than $z$, but we did not add this additional hitting points in the picture for simplicity. 
			}
		\end{center}
		\vspace{-3ex}
	\end{figure}
\end{lem}

\begin{remark} \label{remark:RL}
	Note that when $*=R$ (resp. $*=L$), the curve $\eta$ first fills the regions to the right (resp. left) of the king strand (with respect to the natural orientation of the king strand) and then the ones to the left (resp.\ right). If this is the case, we say that $z$ is a \textbf{$RL$-$m$-tuple point} (resp.\ \textbf{$LR$-$m$-tuple point}).
\end{remark}

We will often record the local behavior of the flow lines around a $m$-tuple point $z$ using some schematic diagrams, called the \textbf{local diagram of $z$}, as explained at the bottom of \cref{fig-m-tuple-points}.

\begin{remark}
	We highlight an important difference between the case where $m^\circ_x=m^*_x$ and the case when $m^\circ_x=m^*_x-1$ in \cref{lem:m-tuple-points} (here $x=u$ or $x=v$). 
	
	Locally around a $m$-tuple point $z$, the flow lines $\beta^{*}$ and $\beta^{\circ}$ always alternate with exactly two exceptions (in correspondence to the two exceptional hitting times). These two exceptions are different in the $m^\circ_x=m^*_x$ and the $m^\circ_x=m^*_x-1$ case:
	\begin{itemize}
		\item when $m^\circ_x=m^*_x$, the exception consists of two consecutive $\beta^{*}_{x}$; 
		
		\item when $m^\circ_x=m^*_x-1$, the exception consists of two consecutive $\beta^{\circ}_{x}$. 
	\end{itemize}
	We invite the reader to compare the two situations in \cref{fig-m-tuple-points}. In each case, the two exceptions are in correspondence to the regions highlighted in blue.
\end{remark}

We finally record the following useful fact (compare the statement of the lemma with \cref{fig-hitting-multiple-points}) which is again a rephrasing of the results stated and proved in~\cite[Lemma 8.13]{wedges}.

\begin{lem}\label{lem:how_hits_points}
	Let $\tau^x_z$ be the time when $\eta$ hits $z$ while traversing the two consecutive regions $U(x)$ and $U(x+1)$, for some $x\in[m]$. Then, for every (small) $\ep>0$, 
	\begin{equation*}
		z \notin \overline{U(x)\setminus\eta([\tau^x_z-\ep,\tau^x_z])}
		\quad\text{and}\quad
		z \notin \overline{U(x+1)\setminus\eta([\tau^x_z,\tau^x_z+\ep])}. 
	\end{equation*}
\end{lem}

\begin{figure}[ht!]
	\begin{center}
		\includegraphics[width=0.8\textwidth]{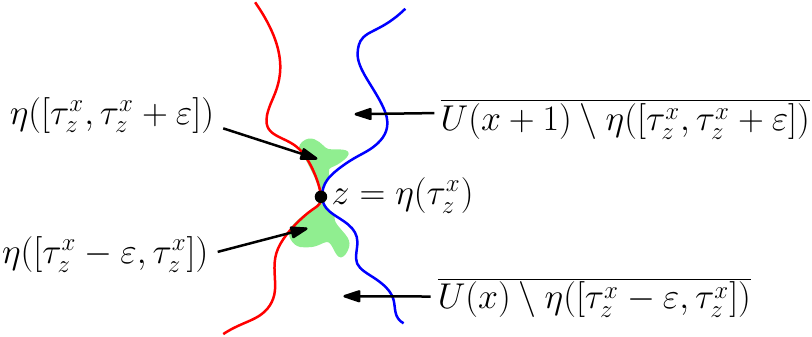}  
		\caption{\label{fig-hitting-multiple-points} A schema for the statement of \cref{lem:how_hits_points}. The point $z$ is a $m$-tuple point of $\eta$. The curve $\eta$ first fills the region $U(x)$ below $z$ included between the two flow lines and then the region $U(x+1)$ above $z$ included between the two flow lines. The green region $\eta([\tau^x_z-\ep,\tau^x_z+\ep])$ filled by $\eta$ immediately before and after hitting $z$ disconnects $z$ from $\overline{U(x)\setminus\eta([\tau^x_z-\ep,\tau^x_z])}$ and $\overline{U(x+1)\setminus\eta([\tau^x_z,\tau^x_z+\ep])}$.
		}
	\end{center}
	\vspace{-3ex}
\end{figure}

\subsubsection{Preliminary results when $\kappa\in (4,8)$: the structure of $m$-tuple points for a pair of space-filling SLEs of different angle}\label{sect:eight-flow}

We fix $z$ to be a $m_1$-tuple point of $\eta_1$. We will see that $z$ must then be a $m_2$-tuple point of $\eta_2$ for some $m_2\in\{m_1-1,m_1,m_1+1\}$.  

Recall that \cref{lem:m-tuple-points} precisely describes the behavior of $\eta_1$ around $z$. 
Let $u,v \in \BB C$ be as in the statement of \cref{lem:m-tuple-points}. We now want to describe the behavior of $\eta_2$ around $z$. 
To do that, we describe the behavior of the four flow lines $\beta^L_{u,2}$, $\beta^R_{u,2}$, $\beta^L_{v,2}$, and $\beta^R_{v,2}$ locally around $z$.

We restrict to the case when $z$ is a $RL$-$m_1$-tuple point of $\eta_1$ (recall Remark~\ref{remark:RL}) and we also assume that\footnote{Note that here we write $m_{u,1}^L$ instead of $m_{u}^L$ (as in the statement of \cref{lem:m-tuple-points}) because we are now dealing with two SLEs and so we need to distinguish the notation for the two SLEs. Similar adaptations to the notation are made for the other quantities.} $m_{u,1}^L=m_{u,1}^R-1$ and $m_{v,1}^L=m_{v,1}^R-1$. All remaining cases are treated in a similar manner and have been skipped for brevity.

\begin{figure}[ht!]
	\begin{center}
		\includegraphics[width=.9\textwidth]{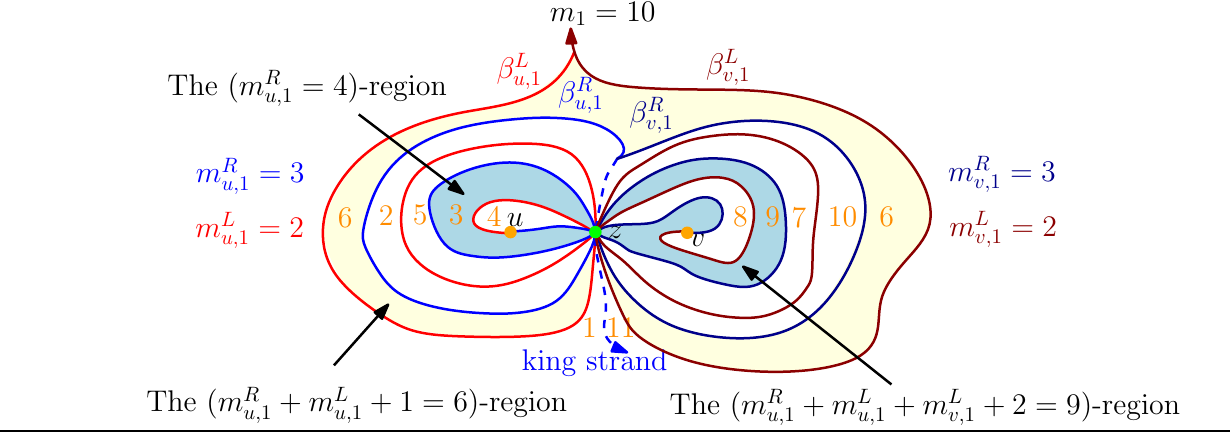}
		\includegraphics[scale=.6]{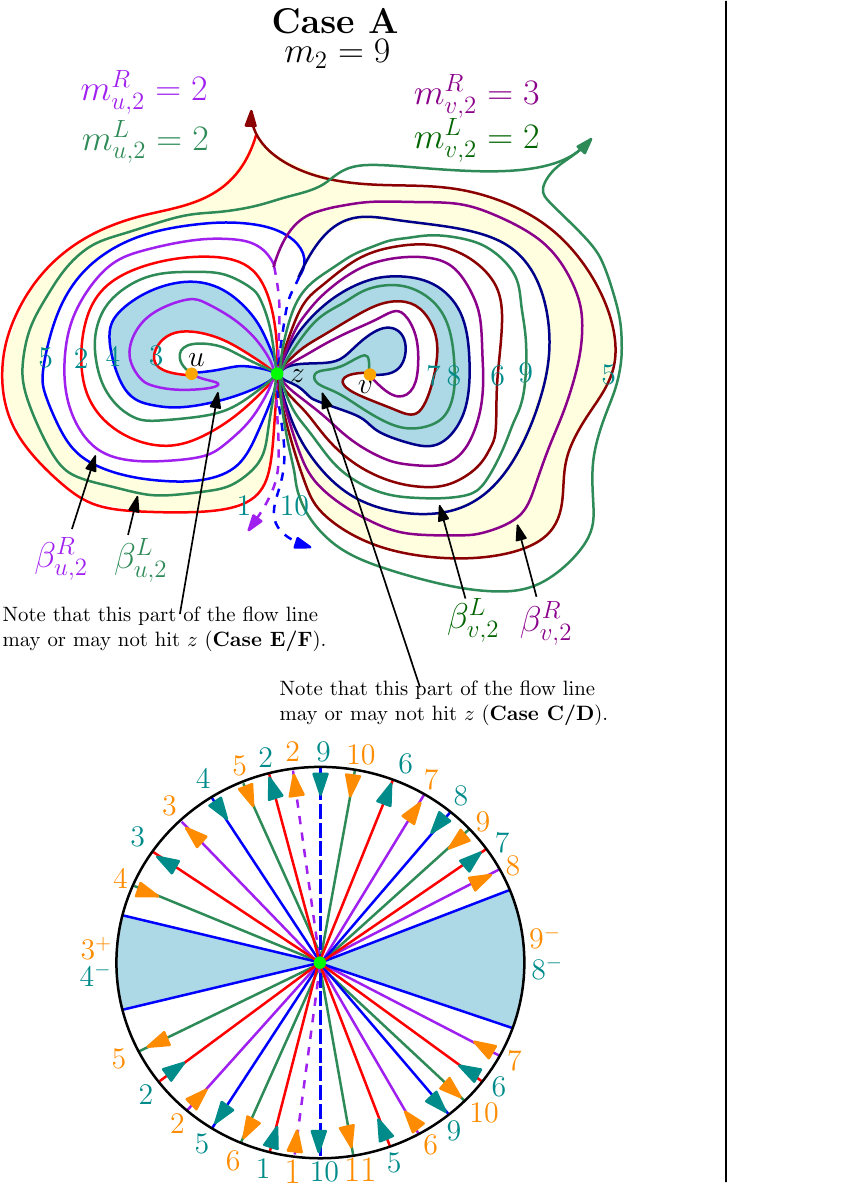}   
		\includegraphics[scale=.6]{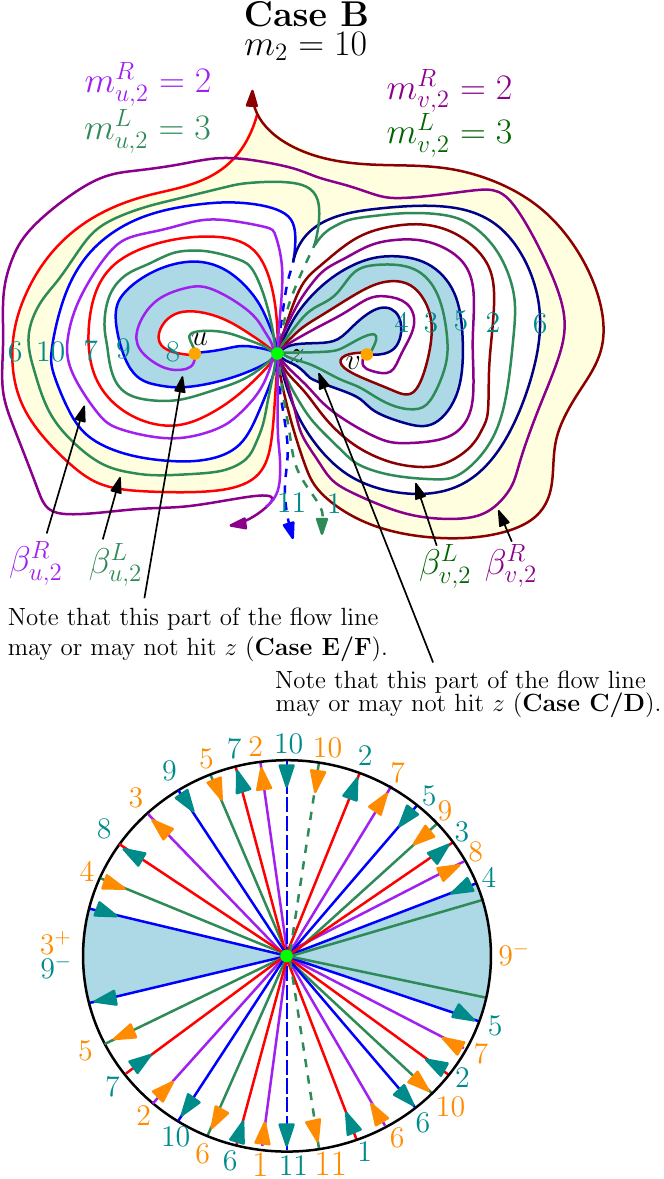} 
		\caption{\label{fig-m-tuple-points-double-flow} 
		A schema for the four steps at the beginning of \cref{sect:eight-flow}.
		\textbf{Top:} We show the local behavior of the flow lines $\beta^L_{u,1}$, $\beta^R_{u,1}$, $\beta^L_{v,1}$, and $\beta^R_{v,1}$ around $z$ when $z$ is a $RL$-$10$-tuple point of $\eta_1$. The 11 regions determined by the four flow lines are numbered in orange according to the order in which they are filled by the space-filling whole-plane SLE $\eta_1$. Some specific regions that play an important role in the four steps are highlighted in colors. \textbf{Middle:} The same scheme as on top with the addition of the four flow lines $\beta^L_{u,2}$, $\beta^R_{u,2}$, $\beta^L_{v,2}$ and $\beta^R_{v,2}$; on the left we drawn the behavior of Case A while on the right we drawn the behavior of Case B. The cyan numbers record the order in which $\eta_2$ fills the regions determined by the four flow lines $\beta^L_{u,2}$, $\beta^R_{u,2}$, $\beta^L_{v,2}$, and $\beta^R_{v,2}$. \textbf{Bottom:} The corresponding local diagrams including all eight flow lines. Colors are used in a consistent way.
		}
	\end{center}
	\vspace{-3ex}
\end{figure}


Under these assumptions, due to \cref{lem:m-tuple-points}, we know that the local behavior of the flow lines $\beta^L_{u,1}$, $\beta^R_{u,1}$, $\beta^L_{v,1}$, and $\beta^R_{v,1}$ around $z$ is as shown at the top of \cref{fig-m-tuple-points-double-flow}. We now describe, one by one, the behavior of the four flow lines $\beta^L_{u,2}$, $\beta^R_{u,2}$, $\beta^L_{v,2}$, and $\beta^R_{v,2}$ locally around $z$:
	

\medskip
	
	\noindent\underline{Step 1: The flow line $\beta^L_{u,2}$.} We start by describing the behavior of the flow line $\beta^L_{u,2}$ (this is the green curve started at $u$ in the middle of \cref{fig-m-tuple-points-double-flow}; both in Case A and Case B). $\beta^L_{u,2}$ starts at $u$ and grows inside the region $U(m_{u,1}^R+1)$, i.e.\ inside the region filled by $\eta_1$ immediately after the blue region $U(m_{u,1}^R)$. Recalling that $\beta^L_{u,2}$ cannot cross the flow lines $\beta^L_{u,1}$ and $\beta^R_{u,1}$ (\cref{lem:flow_lines_crossing}), we deduce that $\beta^L_{u,2}$ is forced to travel inside this region until when it is forced to leave forever the region through $z$ (this is the only point where $\beta^L_{u,2}$ can escape the region $U(m_{u,1}^R+1)$). Note that this is the first time that $\beta^L_{u,2}$ hits $z$. Indeed, $\beta^L_{u,2}$ cannot hit $z$ twice without spiraling around its starting point in between the two hitting times. If it did, $z$ would become an $m$-tuple point for $\eta_2$ (since it cannot be a merge point thanks to \cref{lem:triple_are_simple}) and such behavior of $\eta_2$ around an $m$-tuple point is prohibited by \cref{lem:m-tuple-points}.
	The same argument is then iterated for all the subsequent regions until when $\beta^L_{u,2}$ enters the region $U(m^R_u+m^L_u+1)$, i.e.\ the only region having all the four flow lines $\beta^R_{u,1}$, $\beta^L_{u,1}$, $\beta^L_{v,1}$, and $\beta^R_{v,1}$ appearing on the boundary (this region is highlighted in yellow in \cref{fig-m-tuple-points-double-flow}). Inside this region the flow line $\beta^L_{u,2}$ has two options:
	
	\begin{itemize}
		\item \textbf{Case A:} (See the green curve $\beta^L_{u,2}$ on the left-hand side of \cref{fig-m-tuple-points-double-flow}.) 
		The flow line $\beta^L_{u,2}$ leaves the region $U(m^R_u+m^L_u+1)$ by crossing $\beta_{v,1}^L$ (this is possible since $\beta_{u,2}^L$ hits  $\beta_{v,1}^L$ on its left-hand side, \emph{c.f.}\ \cref{lem:flow_lines_crossing}). Since $\beta_{u,2}^L$ can cross $\beta_{v,1}^L$ exactly once (and $\beta_{u,2}^L$ cannot cross $\beta^L_{u,1}$), the flow line $\beta_{u,2}^L$, after crossing $\beta_{v,1}^L$, goes to $\infty$ and never comes back to $z$. 
		
		\item \textbf{Case B:} (See the green curve $\beta^L_{u,2}$ on the right-hand side of \cref{fig-m-tuple-points-double-flow}.) The flow line $\beta^L_{u,2}$ leaves the region $U(m^R_u+m^L_u+1)$ by crossing $\beta_{v,1}^R$ (this is possible since $\beta_{v,1}^R$ hits $\beta_{u,2}^L$ on its left-hand side) and so enters the (second-to-last) region $U(m_1)=U(m^R_{u,1}+m^L_{u,1}+m^L_{v,1}+m^R_{v,1})$. Since $\beta_{u,2}^L$ can cross $\beta_{v,1}^L$ exactly once (and $\beta_{u,2}^L$ cannot cross $\beta^L_{v,1}$), the flow line $\beta_{u,2}^L$, after crossing $\beta_{v,1}^R$, is forced to stay inside the region $U(m_1)$ before leaving it through $z$. After that, $\beta_{u,2}^L$ goes to $\infty$ and never comes back to $z$.
	\end{itemize}

	\noindent\underline{Step 2: The flow line $\beta^L_{v,2}$.} We now focus our attention on the flow line $\beta^L_{v,2}$ (this is the green curve started at $v$ in the middle of \cref{fig-m-tuple-points-double-flow}; both in Case A and Case B). $\beta^L_{v,2}$ starts at $v$ and it grows inside the region $U(m^R_{u,1}+m^L_{u,1}+m^L_{v,1}+2)$ (this is the right blue region in the pictures in \cref{fig-m-tuple-points-double-flow}). Inside this region, the flow line $\beta^L_{v,2}$ has two options (note that $z$ is contained twice on the boundary of this region, once along the flow line $\beta_{v,1}^R$ and once at the merge point of the flow lines $\beta_{v,1}^R$ and $\beta^L_{v,1}$):
	\begin{itemize}
		\item \textbf{Case C:} 
		The flow line $\beta^L_{v,2}$ hits $z$ twice. 
		
		In this case, $\beta^L_{v,2}$ first hits $z$ (along the flow line $\beta_{v,1}^R$) and then hits $z$ a second time when leaving the region $U(m^R_{u,1}+m^L_{u,1}+m^L_{v,1}+2)$.
		
		\item \textbf{Case D:} 
		The flow line $\beta^L_{v,2}$ hits $z$ only once, exactly before leaving the region $U(m^R_{u,1}+m^L_{u,1}+m^L_{v,1}+2)$. 
	\end{itemize}
	In both cases, after leaving the region $U(m^R_{u,1}+m^L_{u,1}+m^L_{v,1}+2)$, the flow line $\beta^L_{v,2}$ continues to spiral around all the regions between the region $U(m^R_{u,1}+m^L_{u,1}+m^L_{v,1}+2)$ and the region $U(m_1)=U(m^R_{u,1}+m^L_{u,1}+m^L_{v,1}+m^R_{v,1})$ (because the flow line $\beta^L_{v,2}$ is forced to stay between the flow lines $\beta^R_{v,1}$ and $\beta^L_{v,1}$), hitting $z$ exactly once each time it changes region. Once $\beta^L_{v,2}$ reaches the region $U(m_1)$:
	\begin{itemize}
		\item if Case A holds, $\beta^L_{v,2}$ leaves the region $U(m_1)$ through $z$ and then merges with $\beta^L_{u,2}$ (see the left-hand side of \cref{fig-m-tuple-points-double-flow}).
		
		\item if Case B holds, $\beta^L_{v,2}$ merges with $\beta^L_{u,2}$ forming the king strand for $\eta_2$ (see the right-hand side of \cref{fig-m-tuple-points-double-flow}).
	\end{itemize} 

	\noindent\underline{Step 3: The flow line $\beta^R_{u,2}$.} We now analyze the behavior of the flow line $\beta^R_{u,2}$ (this is the purple curve started at $u$ in the middle of \cref{fig-m-tuple-points-double-flow}; both in Case A and Case B). $\beta^R_{u,2}$ starts at $u$ and it grows inside the region $U(m^R_{u,1})$ (this is the left blue region on the top of \cref{fig-m-tuple-points-double-flow}). Similarly as before, inside this region the flow line $\beta^R_{u,2}$ has two options:
	
	\begin{itemize}
		\item \textbf{Case E:} 
		The flow line $\beta^R_{u,2}$ hits $z$ twice. 
		
		\item \textbf{Case F:} 
		The flow line $\beta^R_{u,2}$ hits $z$ only once, exactly when it leaves the region $U(m^R_{u,1})$. 
	\end{itemize}
	After leaving the region $U(m^R_{u,1})$, as usual, the flow line $\beta^R_{u,2}$ spirals around (in reverse order) between the subsequent regions until it reaches the region $U(1)$. Then it goes to $\infty$ and never returns to $z$.
	
	\medskip
	
	\noindent\underline{Step 4: The flow line $\beta^R_{v,2}$.} It only remains to determine the behavior of the flow line $\beta^R_{v,2}$ (this is the purple curve started at $v$ in the middle of \cref{fig-m-tuple-points-double-flow}; both in Case A and Case B). $\beta^R_{v,2}$ starts at $v$ and spirals around (in reversed order) between the region $U(m^R_{u,1}+m^L_{u,1}+m^L_{v,1}+1)$ and the region $U(m^R_{u,1}+m^L_{u,1}+1)$, once it reaches this final region (this region is highlighted in yellow in the top of \cref{fig-m-tuple-points-double-flow}),
	
	\begin{itemize}
		\item if we are in Case A, $\beta^R_{v,2}$ first crosses $\beta^R_{u,1}$ and then merges with $\beta^R_{u,2}$ forming the king strand diagram (see the left-hand side of \cref{fig-m-tuple-points-double-flow}).
		
		\item if we are in Case B, $\beta^R_{v,2}$ first crosses $\beta^L_{u,1}$ and then merges with $\beta^R_{u,2}$ (see the right-hand side of \cref{fig-m-tuple-points-double-flow}).
	\end{itemize} 
	This completes the description of the behavior of the four flow lines $\beta^L_{u,2}$, $\beta^R_{u,2}$, $\beta^L_{v,2}$, and $\beta^R_{v,2}$.

\bigskip

We now restrict our attention to the situation in which the three cases A, D, and F hold (as in the left-hand side of \cref{fig-m-tuple-points-double-flow}). A consequence of the above discussion is that, including the new regions determined by the four flow lines $\beta^L_{u,2}$, $\beta^R_{u,2}$, $\beta^L_{v,2}$ and $\beta^R_{v,2}$ (labeled in the same order as $\eta_2$ fills them) to the local diagram of $z$ introduced in \cref{fig-m-tuple-points}, we obtain a new local diagram of $z$ as the one at the top of \cref{fig-points-in-support-even-1} (see also the bottom part of \cref{fig-m-tuple-points-double-flow}).


Let now $t^1_1,t^1_2,\dots,t^1_{m_1}$ be the $m_1$ times when $\eta_1$ hits $z$ and $t^2_1,t^2_2,\dots,t^2_{m_2}$ be the $m_2$ times when $\eta_2$ hits $z$. The next lemma clarifies which points $(t^1_i,t^2_j)$ are contained in the support of the permuton $\perm$. (\emph{Alert:} In the next lemma we write for completeness exactly what points $(t^1_i,t^2_j)$ are contained in the support of the permuton $\perm$, but what will be more important to us is the relative position between these points; see, in particular, the blue curve connecting these points at the bottom of \cref{fig-points-in-support-even-1}.)

\begin{figure}[ht!]
	\begin{center}
		\includegraphics[width=0.6\textwidth]{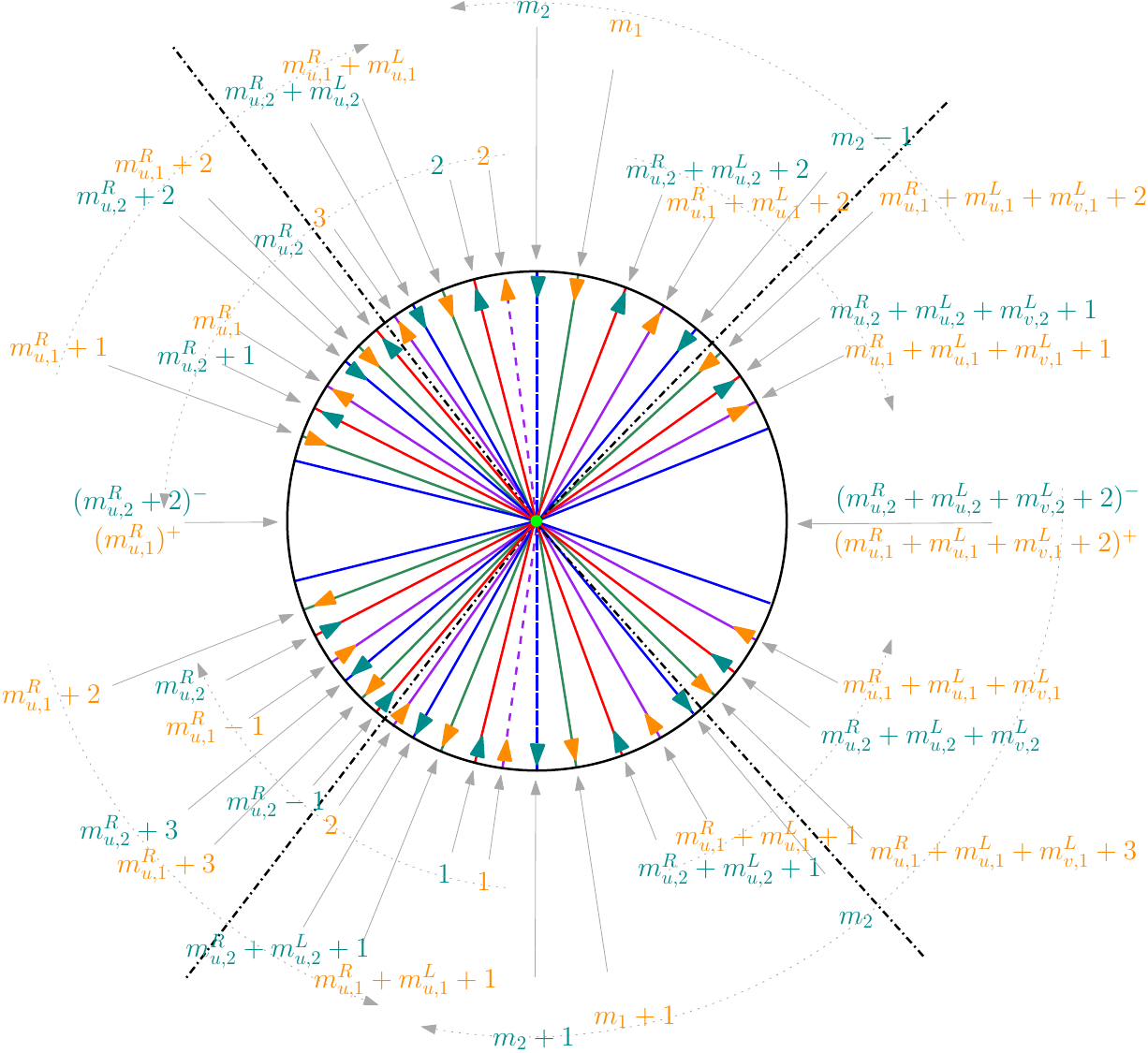}
		
		\vspace{0.1cm} 
		
		\includegraphics[width=0.6\textwidth]{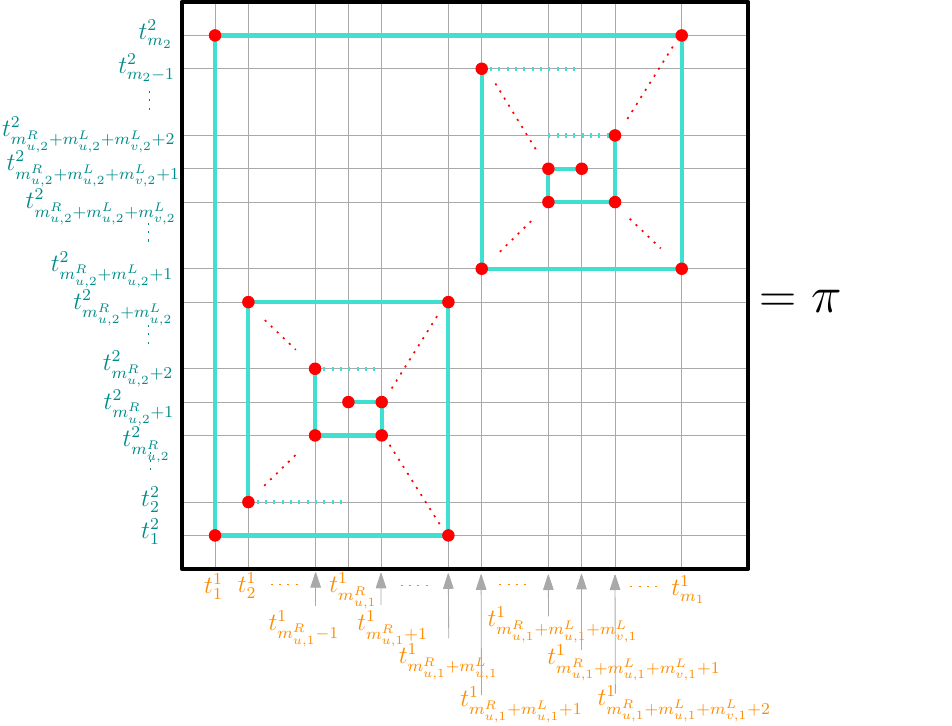}  
		\caption{\label{fig-points-in-support-even-1} 
			\textbf{Top:} The local diagram of a point $z$ that satisfies the assumptions of \cref{lem:points_in_support} (Cases A, D, and F hold). It includes the eight flow lines of $\eta_1$ and $\eta_2$ starting at the two points $u$ and $v$ (four for each point). This diagram is obtained starting from the diagram introduce in \cref{fig-m-tuple-points} and following the four steps at the beginning of \cref{sect:eight-flow} describing the behavior of the four flow lines $\beta^L_{u,2}$, $\beta^R_{u,2}$, $\beta^L_{v,2}$, and $\beta^R_{v,2}$ locally around $z$. The four dashed-dotted black rays indicate that a piece of the diagram has been removed at that location (which can be easily reconstructed by following the natural alternation of colored rays locally around that area). \textbf{Bottom:} The corresponding points $(t^1_i,t^2_j)$ contained in the support of the permuton $\perm$, as described in \cref{lem:points_in_support}. These points are highlighted in red. The blue curve is the only curve connecting them using only vertical/horizontal segments. This blue curve will play an important role in the proof of \cref{prop-counterflow-dichotomy2} in \cref{sect:proof-remaining}.
		}
	\end{center}
	\vspace{-3ex}
\end{figure}

\begin{lem}\label{lem:points_in_support}
	Fix $m_1=m_2+1\geq 2$ and let $z\in\BB C$ be (recall the definitions form \cref{remark:RL}) 
	\begin{itemize}
		\item a $RL$-$m_1$-tuple point of $\eta_1$ such that $m_{u,1}^L=m_{u,1}^R-1$ and $m_{v,1}^L=m_{v,1}^R-1$;
		\item a $RL$-$m_2$-tuple point of $\eta_1$ such that $m_{u,2}^L=m_{u,2}^R=m_{u,1}^L$ and $m_{v,2}^L=m_{v,2}^R-1=m_{v,1}^L$ (i.e., Cases A, D, F above holds).
	\end{itemize}
	Then the following points $(t^1_i,t^2_j)$ are contained in the support of the permuton $\perm$ (see the bottom of \cref{fig-points-in-support-even-1}):
	
	\begin{itemize}
		\item $(t^1_1,t^2_1)$ and $(t^1_1,t^2_{m_2})$;
		\item $(t^1_{1+i},t^2_{1+i})$ and $(t^1_{1+i},t^2_{m^R_{u,2}+m^L_{u,2}+1-i})$ for all $1\leq i\leq m^R_{u,1}-2$;
		\item $(t^1_{m^R_{u,1}},t^2_{m^R_{u,2}+1})$;
		\item $(t^1_{m^R_{u,1}+i},t^2_{m^R_{u,2}+i})$ and $(t^1_{m^R_{u,1}+i},t^2_{m^R_{u,2}+1-i})$ for all $1\leq i\leq m^L_{u,1}$;
		\item $(t^1_{m^R_{u,1}+m^L_{u,1}+i},t^2_{m^R_{u,2}+m^L_{u,2}+i})$ and $(t^1_{m^R_{u,1}+m^L_{u,1}+i},t^2_{m_2-i})$ for all $1\leq i\leq m^L_{v,1}$;
		\item  $(t^1_{m^R_{u,1}+m^L_{u,1}+m^L_{v,1}+1},t^2_{m^R_{u,2}+m^L_{u,2}+m^L_{v,2}+1})$ 
		\item $(t^1_{m^R_{u,1}+m^L_{u,1}+m^L_{v,1}+1+i},t^2_{m^R_{u,2}+m^L_{u,2}+m^L_{v,2}+1+i})$ and $(t^1_{m^R_{u,1}+m^L_{u,1}+m^L_{v,1}+1+i},t^2_{m^R_{u,2}+m^L_{u,2}+m^L_{v,2}+1-i})$ for all $1\leq i\leq m^R_{v,1}-1$.
	\end{itemize}
	Moreover, all the points $(t^1_i,t^2_j)$ not mentioned above are not in the support of the permuton $\perm$.
\end{lem} 

\begin{proof}
	We first show that $(t^1_1,t^2_1)$ and $(t^1_1,t^2_{m_2})$ are in the support of the permuton $\perm$. Note that in the local diagram of $z$ at the top of \cref{fig-points-in-support-even-1}, we have that the inward segment labeled by an orange $1$ (recall from the definition of local diagram in \cref{fig-m-tuple-points} that this segment includes all the area between the adjacent red and blue radii) intersects on the left the inward segment labeled by a blue $1$ (this segment includes all the area between the adjacent green and purple radii) and on the right the outward segment labeled by a blue $m_2+1$ (this segment includes all the area between the adjacent purple and green radii). We now focus on the inward segment labeled by an orange $1$ intersecting the inward segment labeled by a blue $1$, and deduce that $(t^1_1,t^2_1)\in \op{supp} \perm$. The proof for $(t^1_1,t^2_{m_2})$ is very similar.
	
	By construction of the local diagram of $z$ and the discussion in the previous paragraph, we have that locally around $z$, $\eta_1$ and $\eta_2$ behave as shown in \cref{fig-hitting-multiple-points2} when they hit $z$ at time $t^1_1$ and $t^2_1$, respectively. In particular, this local picture combined with  \cref{lem:how_hits_points}, has the following consequence: for every $\ep>0$, 
	\begin{equation*}
		\eta_1\left([t^1_1-\ep,t^1_1+\ep]\right)\cap \eta_2\left([t^2_1-\ep,t^2_1+\ep]\right)
	\end{equation*}
	contains an open non-empty set. Indeed, if this is false then one can immediately deduce that either $z\in \overline{U_1(1)\setminus \eta_1([t^1_1-\ep,t^1_1])}$ or $z\in \overline{U_2(1)\setminus \eta_2([t^2_1-\ep,t^2_1])}$, but both cases would contradict \cref{lem:how_hits_points}. Hence $(t^1_1,t^2_1)\in \op{supp} \perm$ and, similarly, $(t^1_1,t^2_{m_2})\in \op{supp} \perm$.

	\begin{figure}[ht!]
		\begin{center}
			\includegraphics[width=1\textwidth]{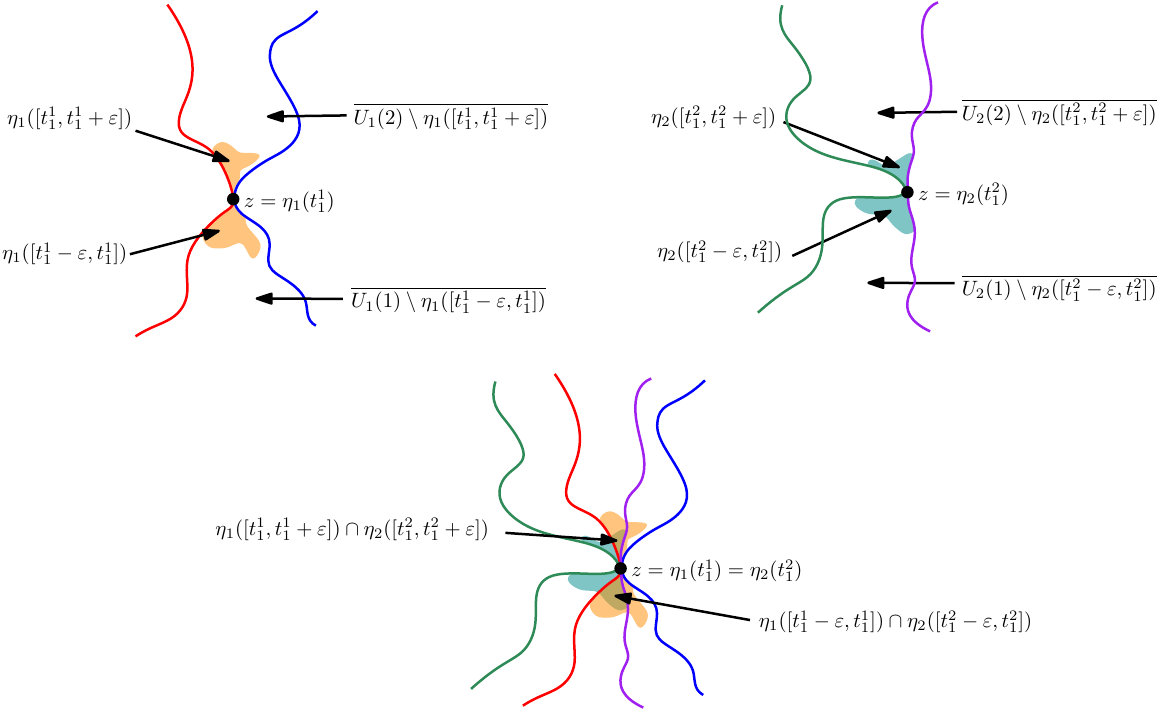}  
			\caption{\label{fig-hitting-multiple-points2} \textbf{Top-left:} The local behavior of $\eta_1$ around $z$ when $\eta_1$ hits $z$ at time $t^1_1$. (Recall  \cref{lem:how_hits_points}.) \textbf{Top-right:} The local behavior of $\eta_2$ around $z$ when $\eta_2$ hits $z$ at time $t^2_1$. (Recall  \cref{lem:how_hits_points}.) \textbf{Bottom:} The local diagram of $z$ at the top of \cref{fig-points-in-support-even-1} guarantees that the boundary flow lines of $\eta_1$ and $\eta_2$ alternates as displayed.
			}
		\end{center}
		\vspace{-3ex}
	\end{figure} 
	
	\medskip
	
	We now show that $(t^1_{m^R_{u,1}},t^2_{m^R_{u,2}+1})$ is in the support of the permuton $\perm$. Note that in the local diagram of $z$ at the top of \cref{fig-points-in-support-even-1}, the segment labeled by the orange $(m^R_{u,1})^+$ is contained inside the segment labeled by the blue $(m^R_{u,2}+2)^-$. Since these segments correspond to the two exceptional hitting times $t^1_{m^R_{u,1}}$ and $t^2_{m^R_{u,2}+1}$, respectively, we conclude, combining the above observation with similar arguments as before, that for every $\ep>0$, 
	\[\eta_1\left([t^1_{m^R_{u,1}}-\ep,t^1_{m^R_{u,1}}+\ep]\right)\cap \eta_2\left([t^2_{m^R_{u,2}+1}-\ep,t^2_{m^R_{u,2}+1}+\ep]\right)\]
	contains an open nonempty set. Hence $(t^1_{m^R_{u,1}},t^2_{m^R_{u,2}+1})\in \op{supp} \perm$. 
	
	\medskip
	
	One can prove that all the remaining points in the list of the  lemma statement are in the support of the permuton by using one of the two arguments above. Finally, noting that for every pair  $(t^1_i,t^2_j)$ not in the list, the corresponding segments in the local diagram of $z$ are disjoint, one can conclude (with a proof similar to the one used for \cref{lem-counterflow-bad}) that these points are not in the support of the permuton.
\end{proof}

\subsubsection{Proof of Proposition~\ref{prop-counterflow-dichotomy2}}\label{sect:proof-remaining}

We can conclude this section proving Proposition~\ref{prop-counterflow-dichotomy2}.

\begin{proof}[Proof of Proposition~\ref{prop-counterflow-dichotomy2}]	
	The first part of the Proposition statement follows from the final claim in \cref{lem:points_in_support}. We prove the second part of the statement.
	
	\medskip
	
	If for all $m_1\geq 2$ and $m_2\geq 2$, $z\in \BB C$ is not simultaneously a $m_1$-tuple point of $\eta_1$ and a $m_2$-tuple point of $\eta_2$, then thanks to \cref{lem:triple_are_simple}, $z\in \BB C$ is a simple point for at least one among $\eta_1$ and $\eta_2$. Therefore, $(t,\psi_{-}(t))\in \op{supp} \perm$ by Assertion 3 in Lemma~\ref{lem-permuton-defined}.
	
	\medskip
	
	We finally address the case when $z\in \BB C$ is simultaneously an $m_1$-tuple point of $\eta_1$ and an $m_2$-tuple point of $\eta_2$ for some $m_1\geq 2$ and $m_2\geq 2$. We restrict ourselves to the case where $z$ is as in the statement of \cref{lem:points_in_support}; the other cases can be treated similarly. Fix $(s,t) \in [0,1]^2$ such that $\eta_1(s) =\eta_1(t)=z$.
	Recall the definitions of $t^1_1,t^1_2,\dots,t^1_{m_1}$ and $t^2_1,t^2_2,\dots,t^2_{m_2}$ from the discussion above the statement of \cref{lem:points_in_support}.
	We need to prove that there exists a sequence of times $(t^1_{i_\ell}, t^2_{j_\ell})_{\ell=0}^M$ such that  
	\begin{itemize}
		\item $t^1_{i_0}=s$ and $t^1_{i_M}=t$;
		\item for every $\ell\in\{1,2,\dots,M\}$, $(t^1_{i_\ell}, t^2_{j_\ell}) \in \op{supp} \perm $;
		\item for every $\ell\in\{2,\dots,M\}$, either $t^1_{i_\ell}=t^1_{i_{\ell-1}}$ or $t^2_{j_\ell}=t^2_{j_{\ell-1}}$.
	\end{itemize} 
	To find this sequence, it is enough to note that \cref{lem:points_in_support} guarantees that there is always at least one point $(t^1_{i_0},t^2_{j_0})\in \op{supp} \perm $ such that $t^1_{i_0}=s$ and one point $(t^1_{i_M},t^2_{j_M})\in \op{supp} \perm $ such that $t^1_{i_M}=t$. Then, the remaining points are exactly the red points along the blue path at the bottom of \cref{fig-points-in-support-even-1} between $(t^1_{i_0},t^2_{j_0})$ and $(t^1_{i_M},t^2_{j_M})$. 
\end{proof}

\section{The permuton determines the two SLEs and the LQG}
\label{sect:perm_det_LQGSLE}

The goal of this section is to prove Theorems~\ref{thm-permuton-determ-lqg-sles}~and~\ref{thm-sle-msrble}. In \cref{sect:deduction}, we will easily deduce \cref{thm-permuton-determ-lqg-sles} from \cref{thm-permuton-multi-points}, assuming \cref{thm-sle-msrble}. Then, the remaining sections will be devoted to the proof of \cref{thm-sle-msrble}.

\subsection{Proof of \cref{thm-permuton-determ-lqg-sles} from \cref{thm-permuton-multi-points}, assuming \cref{thm-sle-msrble}}\label{sect:deduction}

Theorem~\ref{thm-sle-msrble} immediately implies the following corollary in the case of an LQG sphere decorated by two space-filling SLEs. 

\begin{cor} \label{cor-permuton-msrble0}
Let $\gamma \in (0,2)$ and $\kappa > 4$. 
Let $(\BB C , h , \infty)$ be a singly marked unit area $\gamma$-Liouville quantum sphere.
Let $(\eta_1,\eta_2)$ be a pair of whole-plane space-filling SLE$_\kappa$ processes from $\infty$ to $\infty$, coupled together in an arbitrary manner, which are sampled independently from $h$ and then parametrized by $\gamma$-LQG mass with respect to $h$.  
Let $\mcl T \mcl M_1= \mcl T \mcl M(\eta_1)$ be the associated intersection set for $\eta_1$ as in~\eqref{eqn-intersect-set} and let $\perm$ be the permuton associated with $(\eta_1,\eta_2)$ as in~\eqref{eqn-permuton-def}.  

Almost surely, $\mcl T \mcl M_1$ and the support of $\perm$ together determine the curve-decorated quantum surface $(\BB C , h ,  \infty , \eta_1 , \eta_2 )$ up to complex conjugation.
\end{cor}
\begin{proof}
By Theorem~\ref{thm-sle-msrble}, a.s.\ $\mcl T \mcl M_1$ determines the curve-decorated quantum surface $(\BB C , h ,  \infty , \eta_1 )$ up to complex conjugation. 

By Assertions 2 and 3 in Lemma~\ref{lem-permuton-defined}, a.s.\ for each $t \in [0,1]$ such that $\eta_1(t)$ is hit only once by $\eta_2$, the support $\op{supp} \perm$ intersects the line $\{t\} \times [0,1]$ only at the point $(t,s)$, where $s$ is the unique time such that $\eta_2(s) =\eta_1(t)$.\footnote{We are not claiming that if $t \in [0,1]$  is such that $\op{supp} \perm$ intersects the line $\{t\} \times [0,1]$ only at the point $(t,s)$, then $\eta_1(t)$ is non-multiple points of $\eta_2$. Indeed, in general, the latter claim would be false by \cref{lem-counterflow-bad}.}
Therefore, a.s., the curve-decorated quantum surface $(\BB C , h ,  \infty , \eta_1 )$, viewed modulo complex conjugation, and $\op{supp} \perm$ determine the order in which the non-multiple points of $\eta_2$ are hit by $\eta_2$.

The set of times corresponding to non-multiple points of $\eta_2$ is a.s.\ dense in $[0,1]$. Indeed, the set of non-multiple points of $\eta_2$ has full $\mu_h$-mass \cite{ghm-kpz}, and so its preimage has full Lebesgue measure. Since $\eta_2$ is continuous, it follows that $\op{supp} \perm$ together with the curve-decorated quantum surface $(\BB C , h ,  \infty , \eta_1 )$, viewed modulo complex conjugation, a.s.\ determine the curve-decorated quantum surface $(\BB C , h ,  \infty , \eta_1 , \eta_2 )$, viewed modulo complex conjugation.  
\end{proof}

\begin{proof}[Proof of \cref{thm-permuton-determ-lqg-sles}]
In light of Corollary~\ref{cor-permuton-msrble0}, it suffices to show that the support of $\perm$ a.s.\ determines the intersection set $\mcl T \mcl M_1 = \mcl T \mcl M(\eta_1)$ from~\eqref{eqn-intersect-set}. But this is exactly the statement of \cref{thm-permuton-multi-points}.
\end{proof}

%
%
%

\subsection{The set of hitting-times of multiple points of an SLE determines the curve-decorated quantum sphere} \label{sec-multi-det-lqg-andsles}

In this section, we fix $\gamma\in (0,2)$ and $\kappa >4$ and consider  a singly marked unit area $\gamma$-Liouville quantum sphere $(\BB C ,h, \infty)$ and a whole-plane space-filling SLE$_\kappa$ $\eta$ from $\infty$ to $\infty$ sampled independently from $h$ and then parametrized by $\gamma$-LQG mass with respect to $h$.
We will show that the set $\mcl T \mcl M (\eta)\subset [0,1]^2$ defined in~\eqref{eqn-intersect-set} determines the curve-decorated quantum surface $(\BB C , h ,  \infty , \eta )$ up to complex conjugation, proving Theorem~\ref{thm-sle-msrble}.

\subsubsection{The adjacency graph of SLE cells}
\label{sec-cells}

For $n\in\BB N$, let $\mcl G^n$ be the graph with vertex set $\mcl V\mcl G^n := (0,1]\cap \frac{1}{n} \BB Z$, with two distinct vertices $x,y\in (0,1]\cap \frac{1}{n}\BB Z$ joined by an edge if and only if $\eta ([x-1/n,x]) \cap \eta ([y-1/n,y])$ contains a non-singleton connected set. We refer to $\eta ([x-1/n,x])$ as the \textbf{cell} associated with $x\in\mcl V\mcl G^n$.

In the special case when $\kappa = 16/\gamma^2$, the graph $\mcl G^n$ is precisely the mated-CRT map with the sphere topology~\cite{ghs-dist-exponent,gms-tutte}, but viewed as a graph rather than a planar map. We view $\mcl G^n$ as a graph rather than a planar map because $\mcl T \mcl M$ determines $\mcl G^n$ only as a graph, not as a planar map (this is related to the fact that the curve-decorated LQG sphere is determined by $\mcl T \mcl M$ only up to complex conjugation; recall the discussion below the statement of Theorem~\ref{thm-sle-msrble}). Nevertheless, $\mcl G^n$ viewed as a graph is sufficient to define the random walk on $\mcl G^n$, and hence to determine the Tutte embedding of $\mcl G^n$, as explained below \eqref{eq:tutte}.

When $\kappa \geq 8$, a.s.\ two cells intersect if and only if their intersection contains a non-singleton connected set, so we get the same graph if we instead require only that $\eta ([x-1/n,x])\cap \eta ([y-1/n,y]) \not=\emptyset$. When $\kappa \in (4,8)$, it is possible for the intersection of two cells to be a totally disconnected cantor-like set. The vertices of $\mcl G^n$ corresponding to cells which intersect in this way are not joined by an edge of $\mcl G^n$. See Figure~\ref{fig-weird-cell} for an illustration.

\begin{figure}[t!]
 \begin{center}
\includegraphics[scale=.95]{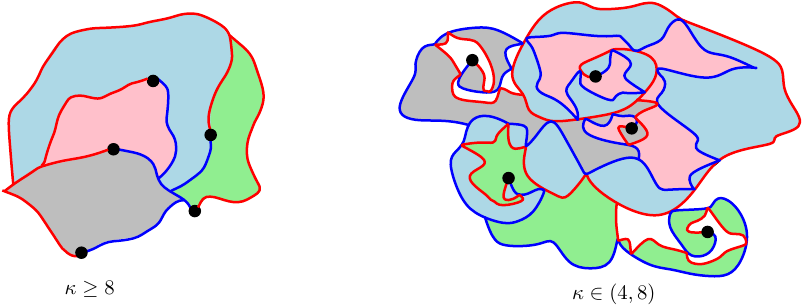}
\vspace{-0.01\textheight}
\caption{Four typical space-filling SLE$_\kappa$ cells for $\kappa\geq 8$ (left) and for $\kappa\in (4,8)$ (right). The points where $\eta $ starts and finishes filling in each cell are shown with black dots. $\eta$ fills the cells in the order of gray, pink, blue, and green.  The right picture is slightly misleading since the set of ``pinch points'' where the left (red) and right (blue) boundaries of each cell meet is actually uncountable and totally disconnected, with no isolated points, and has Hausdorff dimension greater than 1. Note that in the right figure, the grey and green cells intersect at several points, but do not share a connected boundary arc so the corresponding vertices of $\mcl G^n$ are \emph{not} joined by an edge. This is natural since one can think of the blue cell as lying in between the grey and green cells. In fact, two cells which intersect, but do not share a connected boundary arc, will always be separated by one or more other cells in this manner. 
}\label{fig-weird-cell}
\end{center}
\vspace{-1em}
\end{figure}

The relevance of the graph $\mcl G^n$ for the proof of Theorem~\ref{thm-sle-msrble} is that it is a measurable function of the intersection set $\mcl T \mcl M=\mcl T \mcl M (\eta)$ of~\eqref{eqn-intersect-set}. 
  
\begin{lem} \label{lem-graph-permuton}
Almost surely, two distinct vertices $x,y\in \mcl V\mcl G^n$ are joined by an edge if and only if 
\begin{enumerate}
	\item there exists $u \in [x-1/n,x]$ and $v \in [y-1/n,y]$ such that $(u,v) \in \mcl T \mcl M$;
	\item and there does not exist $q \in [0,1] \setminus [(x-1/n,x) \cup (y-1/n,y)]$ such that $(u,q) \in \mcl T \mcl M$.
\end{enumerate}
In particular, the graph $\mcl G^n$ is a.s.\ determined by $\mcl T \mcl M$. 
\end{lem}
\begin{proof}
Almost surely, the intersection $\eta ([x-1/n,x]) \cap \eta ([y-1/n,y])$ contains a non-trivial connected set (i.e., $x$ and $y$ are joined by an edge in $\mcl G^n$) if and only if there is a point in this intersection which belongs to the interior of $\eta ([x-1/n,x]) \cup \eta ([y-1/n,y])$. The boundary of $\eta ([x-1/n,x]) \cup \eta ([y-1/n,y])$ is exactly the set of points of $\eta ([x-1/n,x]) \cup \eta ([y-1/n,y])$ which are hit by $\eta$ at a time which is not in $(x-1/n,x) \cup (y-1/n,y)$.
Hence, $x$ and $y$ are joined by an edge in $\mcl G^n$ if and only if there exists $u \in [x-1/n,x] $ and $v\in [y-1/n,y]$ such that $\eta(u) = \eta(v)$ and 
\eqbn
\eta(u) \notin \eta \left( [0,1]\setminus [(x-1/n,x)\cup (y-1/n,y)] \right) .
\eqen  
By recalling the definition~\eqref{eqn-intersect-set} of $\mcl T \mcl M$, we see that this is equivalent to the existence of $u,v$ as in the lemma statement. 
\end{proof}

We will have occasion to consider various special subgraphs of $\mcl G^n$.

\begin{defn} \label{def-graph-subset} 
For a set $A\subset \BB C$, we define $\mcl G^n(A)$ to be the subgraph of $\mcl G^n$ induced by the set of vertices $x\in\mcl V\mcl G^n$ such that the interior of the cell $\eta ([x-1/n,x])$ intersects $A$. For $0\leq a < b \leq 1$, we also define 
\eqb \label{eqn-graph-interval}
\mcl G^n_{[a,b]}  
:= \mcl G^n(\eta ([a,b]))
= \left(\text{subgraph of $\mcl G^n$ induced by $\left\{x\in (0,1] \cap \frac{1}{n}\BB Z : [x-1/n,x] \cap (a,b) \not=\emptyset   \right\}$}\right) .
\eqe
\end{defn}

We note that Lemma~\ref{lem-graph-permuton} implies that if $a$ and $b$ are chosen in a $\sigma(\mcl T \mcl M)$-measurable manner, then $\mcl G^n_{[a,b]}$ is a.s.\ determined by $\mcl T \mcl M$. 

\medskip
  
The following proposition tells us that random walk on $\mcl G^n$, under the a priori embedding $x\mapsto \eta (x)$, converges to Brownian motion modulo time parametrization. It has, effectively, already been proven in~\cite{gms-tutte}.

\begin{prop} \label{prop-rw-conv}
Let $K\subset \BB C$ be a deterministic compact set. The the Prokhorov distance between the laws of following two random paths converges to zero in probability with respect to the topology on curves viewed modulo time parametrization, uniformly over all $x \in \mcl V\mcl G^n(K)$ (notation as in Definition~\ref{def-graph-subset}):
\begin{itemize}
\item Let $X^x$ denote random walk on $\mcl G^n$ started from $x$ and stopped at the first time it hits a vertex of $\mcl V\mcl G^n\setminus \mcl V\mcl G^n(K)$. Consider the path $j \mapsto \eta(X^x_j)$, extended from $\BB N_0$ to $[0,\infty)$ by piecewise linear interpolation.
\item Let $\mcl B^{\eta(x)}$ denote Brownian motion started from $\eta(x)$ and stopped when it exits $K$. 
\end{itemize} 
\end{prop}
\begin{proof}
The analog of the proposition statement for the mated-CRT map with the sphere topology under its SLE / LQG embedding (i.e., the case when $\kappa = 16/\gamma^2$) is established in~\cite[Theorem 3.4]{gms-tutte}, building on a general invariance principle for random walk in random planar environments from~\cite{gms-random-walk}. The proof of~\cite[Theorem 3.4]{gms-tutte} in the case of the quantum sphere does not use any special features of the case when $\kappa = 16/\gamma^2$. In particular, as alluded to at the beginning of~\cite[Section 5]{gms-tutte}, the following parts of the proof all carry over verbatim to the case when $\kappa \in (4,\infty)\setminus \{ 16/\gamma^2 \}$:
\begin{itemize}
\item the proofs of the moment bounds in~\cite[Section 4.1]{gms-tutte} (one still gets moments up to order $4/\gamma^2$ for the degree);
\item the verification of the hypotheses of the invariance principle from~\cite{gms-random-walk} when $h$ is replaced by an embedding of a 0-quantum cone in~\cite[Section 3.1]{gms-tutte};
\item  the arguments to transfer from the 0-quantum cone to the $\gamma$-quantum cone and then to the quantum sphere in~\cite[Sections 3.2 and 3.3]{gms-tutte}. 
\end{itemize} 
Hence, one gets the proposition statement from the proofs in~\cite{gms-tutte}.
\end{proof}

To prove Theorem~\ref{thm-sle-msrble}, we want to define an embedding function $\Phi^n : \mcl V\mcl G^n \rta \BB C$ in terms of observables related to the random walk on $\mcl G^n$ (in particular, we will consider a version of the Tutte embedding of $\mcl G^n$ as in~\cite{gms-tutte}). We will then use Proposition~\ref{prop-rw-conv} to show that $\Phi^n$ is in some sense close to the a priori embedding $x\mapsto \eta (x)$ when $n$ is large. Since $\mcl G^n$ is determined by $\mcl T \mcl M$, this will lead to a proof of Theorem~\ref{thm-sle-msrble}. 
This idea has previously been used in~\cite{gms-tutte} to give an explicit way to recover an SLE-decorated quantum disk from its associated mating of trees Brownian motion and in~\cite{gms-poisson-voronoi,afs-metric-ball} to give an explicit way to recover an LQG surface from its metric measure space structure. 

In our setting, the proof that Proposition~\ref{prop-rw-conv} implies Theorem~\ref{thm-sle-msrble} is more involved than the analogous proofs in~\cite{gms-tutte,gms-poisson-voronoi,afs-metric-ball}. The reason is that we are only able to show that $\mcl T \mcl M$ determines $\mcl G^n$ as a \emph{graph}, not as a planar map. In particular, $\mcl T \mcl M$ does not determine the cyclic ordering of the vertices on the boundary of a connected subgraph of $\mcl G^n$. The usual definition of the Tutte embedding uses this cyclic ordering, so we will need to do some work to determine exactly to what extent $\mcl T \mcl M$ fails to determine this cyclic ordering.

\subsubsection{Recovering the geometry of $\eta $ from $\mcl T \mcl M$}\label{sect:rec-geom}

In this subsection, we will explain how to recover various geometric features of the curve $\eta $ from $\mcl T \mcl M$. 
For each $a,b\in [0,1]$ with $a<b$, the boundary $\bdy \eta ([a,b])$ is the union of two curves from $\eta (a)$ to $\eta (b)$. We refer to these curves as the \textbf{left boundary} and the \textbf{right boundary} of $\eta ([a,b])$, where the left (resp.\ right) boundary is the curve which lies to the left (resp.\ right) of $\eta$; recall also the explanations in \cref{sec:SLE_as_flow}.

It is easy to see from the definition of $\mcl G^n$ that a.s.\ any two consecutive vertices of $\mcl G^n|_{[a,b]}$ whose corresponding cells intersect the left (resp.\ right) boundary of $\mcl G^n$ are joined by an edge in $\mcl G^n$. We will use this fact without comment in what follows. 

\begin{lem} \label{lem-bdy-determined}
Almost surely, for each $a,b\in [0,1]$ with $a<b$ and each $t\in (a,b)$, we have $\eta (t) \in \bdy\eta ([a,b])$ if and only if there exists $s \in [0,1]\setminus [a,b]$ such that $(t,s) \in \mcl T \mcl M$. 

In particular, if $a$ and $b$ are chosen in a $\sigma(\mcl T \mcl M)$-measurable manner, then the set 
\begin{equation*}
	\{t\in [a,b] : \eta (t) \in \bdy\eta ([a,b])\}
\end{equation*}
is a.s.\ determined by $\mcl T \mcl M$.
\end{lem}
\begin{proof}
We have $\eta (t) \in \bdy\eta ([a,b])$ if and only if there is a time $s\in [0,1]\setminus (a,b)$ such that $\eta (t) = \eta (s)$. The lemma statement is immediate from this and the definition~\eqref{eqn-intersect-set} of $\mcl T \mcl M$. 
\end{proof}

We note that $\mcl T \mcl M$ does \emph{not} determine the set of $t\in [a,b]$ such that $\eta (t)$ belongs to the left boundary of $\eta ([a,b])$, and similarly with ``left'' and ``right'' interchanged. Indeed, if we were to replace $\eta $ by its complex conjugate $\ol\eta$, we would swap ``left'' and ``right'' but we would not change $\mcl T \mcl M$. See, however, Lemma~\ref{lem-bdy-arc-determined} below. 

When $\kappa\geq 8$, it holds for each $0\leq a < b \leq 1$ that the left and right boundaries of $\eta ([a,b])$ meet only at their endpoints and $\eta ([a,b])$ is simply connected. When $\kappa \in (4,8)$, the two boundary curves can intersect (but not cross) each other. In this case, $\eta ([a,b])$ looks like an ordered string of ``beads'', each of which is simply connected, which meet at the intersection points of the left and right boundaries. See Figure~\ref{fig-cell-bdy}, left, for an illustration.

\begin{figure}[t!]
 \begin{center}
\includegraphics[scale=.85]{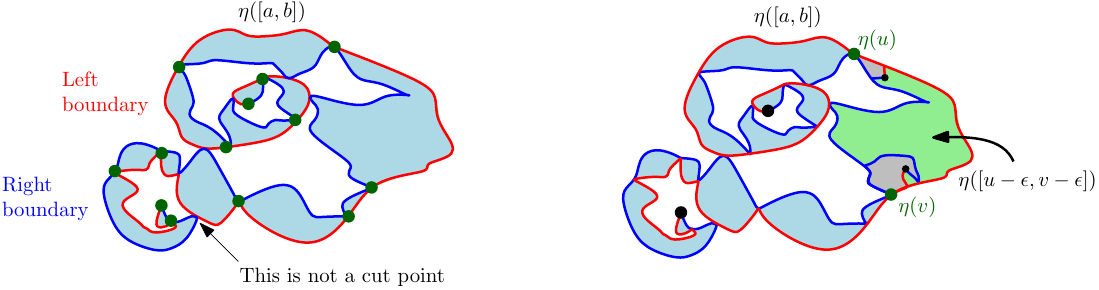}
\vspace{-0.01\textheight}
\caption{ \textbf{Left:} The segment $\eta ([a,b])$ with its left and right boundaries shown in red and blue, respectively, and the images under $\eta$ of the cut times (Definition~\ref{def-cut}) shown as dark green dotes. \textbf{Right:} Illustration of the statement and proof of Lemma~\ref{lem-bdy-arc-determined}. 
}\label{fig-cell-bdy}
\end{center}
\vspace{-1em}
\end{figure} 

\begin{defn} \label{def-cut}
We say that $t\in [a,b]$ is a \textbf{cut time} for $\eta |_{[a,b]}$ if either $t \in \{a,b\}$ or $\eta ([a,t]) \cap \eta ([t,b])$ does not contain a non-singleton connected set. If $t$ is a cut time for $\eta |_{[a,b]}$, then $\eta(t)$ is a \textbf{cut point} of $\eta |_{[a,b]}$.
\end{defn}

Cut times for $\eta |_{[a,b]}$ correspond to points where the right side of the left boundary meets the left side of the right boundary. Points where the left side of the left boundary meets the right side of the right boundary do not give rise to cut points. However, such points can give rise to a totally disconnected set of intersection points of $\eta ([a,t]) $ and $\eta ([t,b])$ even when $t$ is a cut time for $\eta|_{[a,b]}$. As noted above, for $\kappa\geq 8$ the only cut times for $\eta|_{[a,b]}$ are $a$ and $b$. 

\begin{lem} \label{lem-cut-determined}
Almost surely, for each $a,b\in [0,1]$ with $a<b$ and each $t\in (a,b)$, the time $t$ is a cut time for $\eta |_{[a,b]}$ if and only if the following is true. 
For each $\ep > 0$, it holds for each large enough $n\in\BB N$ that the vertex sets $\mcl V\mcl G^n_{[a,t-\ep]}$ and $\mcl V\mcl G^n_{[t+\ep,b]}$ are not joined by an edge of $\mcl G^n$. 

In particular, if $a$ and $b$ are chosen in a $\sigma(\mcl T \mcl M)$-measurable manner then the set 
\begin{equation*}
	\{t\in [a,b] : \text{$t$ is a cut time for $\eta |_{[a,b]}$}\}
\end{equation*} 
is a.s.\ determined by $\mcl T \mcl M$. 
\end{lem}
\begin{proof}
By Definition~\ref{def-cut}, if $t$ is a cut time for $\eta |_{[a,b]}$, then for each $\ep > 0$ the set $\eta ([a,t-\ep]) \cap \eta ([t+\ep,b])$ is totally disconnected. The definition of $\mcl G^n$ therefore implies that if $n$ is large enough, then $\mcl V\mcl G^n_{[a,t-\ep]}$ and $\mcl V\mcl G^n_{[t+\ep,b]}$ are not joined by an edge of $\mcl G^n$. Conversely, if the condition in the lemma statement holds for each $\ep > 0$ and each large enough $n$, then the definition of $\mcl G^n$ shows that for each $\ep > 0$, the set $\eta ([a,t-\ep]) \cap \eta ([t+\ep,b])$ is totally disconnected. Sending $\ep \rta 0$ then gives that $t$ is a cut time for $\eta |_{[a,b]}$. 
The last assertion is immediate from our above characterization of cut times and Lemma~\ref{lem-graph-permuton}.
\end{proof}

Although $\mcl T \mcl M$ does not distinguish between the left and right boundaries of $\eta ([a,b])$, the following lemma tells us that for certain special segments $\eta ([u,v])$, we can see from $\mcl T \mcl M$ the division of $\bdy\eta ([u,v])$ into two distinguished subsets (we just can't see which one is ``left'' and which one is ``right''). 

\begin{lem} \label{lem-bdy-arc-determined}
Let $a,b\in [0,1]$ with $a<b$ and let $u,v\in [a,b]$ be cut times for $\eta |_{[a,b]}$ such that $u <v$ and there are no cut times for $\eta |_{[a,b]}$ in $(u,v)$ (note that for $\kappa \geq 8$, we can take $u =a$ and $v = b$). Assume that $a,b,u,v$ are chosen in a $\sigma(\mcl T \mcl M)$-measurable manner. Then $\mcl T \mcl M$ determines the \emph{unordered} pair of sets
\allb \label{eqn-bdy-arc-determined} 
&\left\{t \in [u,v] : \text{$\eta (t)$ is in the left boundary of $\eta ([u,v])$}\right\} \quad \text{and} \notag\\
&\left\{t \in [u,v] : \text{$\eta (t)$ is in the right boundary of $\eta ([u,v])$}\right\}
\alle
but not which of these sets corresponds to ``left'' and which corresponds to ``right''. 
\end{lem}
\begin{proof}
See Figure~\ref{fig-cell-bdy}, right, for an illustration.
Let $a,b,u,v$ be as in the lemma statement. Let us first note that since there are no cut times for $\eta |_{[a,b]}$ in $(u,v)$, the left and right boundaries of $\eta([u,v])$ are subsets of the left and right boundaries of $\eta ([a,b])$ which meet only at their endpoints. In particular, $\eta ([u,v])$ is bounded by a Jordan curve and so is simply connected. 

For each $\ep > 0$, the set $\bdy \eta ([u,v]) \cap \bdy \eta ([u+\ep ,v - \ep])$ has two (possibly empty) connected components, one of which is a subset of the left boundary of $\eta ([u,v])$ and the other of which is a subset of the right boundary. By the continuity of $\eta $, a.s.\ as $n\rta\infty$ the maximal size of the cells $\eta ([x-1/n,x])$ for $x\in \mcl V\mcl G^n$ tends to zero. From this and the definition of $\mcl G^n$, we see that for each large enough $n\in\BB N$ (depending on $\ep$), none of the cells $\eta ([x-1/n,x])$ for $x\in \mcl V\mcl G^n_{[u+\ep ,v-\ep]}$ shares non-trivial boundary arcs with both the left and right boundaries of $\eta ([u,v])$. Therefore, the graph
\eqb \label{eqn-bdy-arc-split}
\mcl G^n\left( \bdy \eta ([u,v]) \cap \bdy \eta ([u+\ep ,v - \ep]) \right) ,
\eqe 
defined as in Definition~\ref{def-graph-subset} has two connected components, one of which consists of the vertices whose corresponding cells intersect the left boundary of $\eta ([u,v])$ and the other of which consists of the vertices whose corresponding cells intersect the right boundary. Call the vertex sets of these two connected components $V^n$ and $\wt V^n$. We note that $\mcl G^n$ (and hence also $\mcl T \mcl M$, by Lemma~\ref{lem-graph-permuton}) determines $V^n$ and $\wt V^n$, but not which of $V^n$ and $\wt V^n$ corresponds to the left boundary and which corresponds to the right boundary. 

Each time $t \in [u+\ep , v-\ep]$ such that $\eta (t) \in \bdy \eta ([u,v])$ is contained in $[x-1/n,x]$ for one of the vertices of the graph~\eqref{eqn-bdy-arc-split}. From this and the preceding paragraph, we get that for each large enough $n\in\BB N$, the unordered pair of sets
\allb
&[u+\ep , v-\ep] \cap \bigcup_{x\in V^n} [x-1/n,x] \cap \eta ^{-1}(\bdy \eta ([u,v]) ) \quad \text{and} \notag\\
&\qquad\qquad [u+\ep , v-\ep] \cap \bigcup_{x\in \wt V^n} [x-1/n,x] \cap \eta ^{-1}(\bdy \eta ([u,v]) )
\alle
coincides with the two sets in~\eqref{eqn-bdy-arc-determined} intersected with $[u+\ep,v-\ep]$. By Lemmas~\ref{lem-graph-permuton} and~\ref{lem-bdy-determined}, we get that for each $\ep > 0$, the unordered pair of sets consisting of the intersections of the sets in~\eqref{eqn-bdy-arc-determined} with $[u+\ep , v-\ep]$ is determined by $\mcl T \mcl M$. Sending $\ep \rta 0$ concludes the proof. 
\end{proof}

\subsubsection{Convergence of the Tutte embedding} \label{sec-tutte}

\begin{figure}[t!]
 \begin{center}
\includegraphics[scale=.85]{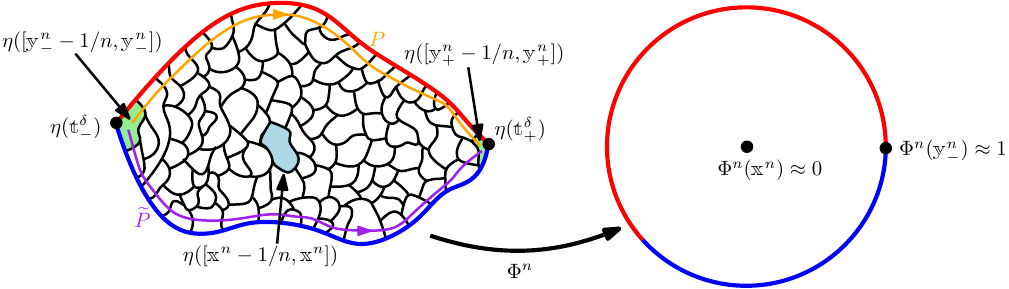}
\vspace{-0.01\textheight}
\caption{ Illustration of the cells corresponding to vertices of $\mcl G^n|_{[\BB t_-^\delta,\BB t_+^\delta]}$ and the associated Tutte embedding map $\Phi^n$. The figure correspond to the case when $\kappa\geq 8$ since we have drawn the cells as simply connected regions, but the set $\eta([\BB t_-^\delta,\BB t_+^\delta])$ is simply connected even when $\kappa \in (4,8)$. The paths $P$ and $\wt P$ visit, in numerical order, the vertices of $\mcl G^n_{[\BB t_-^\delta,\BB t_+^\delta]}$ whose corresponding cells intersect the left and right boundaries of $\eta([\BB t_-^\delta,\BB t_+^\delta])$, respectively. Note that some vertices can be hit by both $P$ and $\wt P$. The intersection set $\mcl T \mcl M$ determines the unordered pair $\{P,\wt P\}$, but not which path corresponds to ``left'' and which corresponds to ``right''. For this reason, the Tutte embedding function $\Phi^n$ can potentially approximate either a conformal map or an anticonformal map from $\eta([\BB t_-^\delta,\BB t_+^\delta])$ to $\BB D$. 
}\label{fig-tutte-domain}
\end{center}
\vspace{-1em}
\end{figure} 

We will now define a version of the Tutte embedding of $\mcl G^n$. 
More precisely, the Tutte embedding is most easily defined for graphs with boundary, so we will consider the Tutte embedding of a large subgraph of $\mcl G^n$ and eventually take a limit as this subgraph increases to all of $\mcl G^n$. See Figure~\ref{fig-tutte-domain} for an illustration of the objects defined in this section. 

Let $\delta  \in (0,1/2)$, which we will eventually send to zero. In the terminology of Definition~\ref{def-cut}, let
\allb
\BB t_-^\delta &:= \left(\text{last cut time for $\eta |_{[\delta,1-\delta]}$ before time $1/2$}\right) \quad \text{and} \notag\\
\BB t_+^\delta &:= \left(\text{first cut time for $\eta |_{[\delta,1-\delta]}$ after time $1/2$}\right) .
\alle
Then $1/2 \in [\BB t_-^\delta , \BB t_+^\delta]$ and $\eta ([\BB t_-^\delta ,\BB t_+^\delta])$ has no cut times, so is bounded by a Jordan curve and hence is simply connected.
Note that if $\kappa\geq 8$, then $\BB t_-^\delta  =\delta$ and $\BB t_+^\delta =1-\delta$, but this is typically not the case for $\kappa \in (4,8)$. As we will see in \cref{lem-endpoints}, it almost surely holds that $\BB t_-^\delta \to 0$ and $\BB t_+^\delta \to 1$ as $\delta \to 0$.
 
For $n\in\BB N$, let $\BB y_-^n , \BB y_+^n\in \mcl V\mcl G^n$ be chosen so that
\eqb \label{eqn-endpt-vertices}
\BB t_-^\delta \in [\BB y_-^n -1/n, \BB y_-^n] \quad \text{and} \quad \BB t_+^\delta \in [\BB y_+^n -1/n, \BB y_+^n] .
\eqe
By Lemma~\ref{lem-cut-determined}, the times $\BB t_-^\delta$ and $\BB t_+^\delta$ are measurable functions of $\mcl T \mcl M$. By Lemma~\ref{lem-bdy-arc-determined}, the same is true of the unordered pair of sets
\allb \label{eqn-delta-sets}
&\left\{t \in [u,v] : \text{$\eta (t)$ is in the left boundary of $\eta ([\BB t_-^\delta , \BB t_+^\delta ])$}\right\} \quad \text{and} \notag\\
&\left\{t \in [u,v] : \text{$\eta (t)$ is in the right boundary of $\eta ([\BB t_-^\delta , \BB t_+^\delta])$}\right\} .
\alle 
 
Let $P : [0,K]_{\BB Z} \rta\mcl V\mcl G^n_{[\BB t_-^\delta , \BB t_+^\delta ]}$ be the function defined so that $P(k)$ is the $k$th smallest vertex $x\in \mcl V\mcl G^n_{[\BB t_-^\delta , \BB t_+^\delta ]}$ (in numerical order) with the property that $[x-1/n,x]$ intersects the first set in~\eqref{eqn-delta-sets}. Similarly define $\wt P : [0,\wt K]_{\BB Z} \rta \mcl V\mcl G^n_{[\BB t_-^\delta , \BB t_+^\delta ]}$, but with the second set in place of the first set. Then each of $P$ and $\wt P$ is a path in $\mcl G^n$ from $\BB y_-^n$ to $\BB y_+^n$ (notation as in~\eqref{eqn-endpt-vertices}). 

\begin{defn} \label{def-bdy-path}
For each $n\in\BB N$, we define the \textbf{boundary path} $P_* = P_*^n : [0,K+\wt K]_{\BB Z} \rta   \mcl V \mcl G^n_{[\BB t_-^\delta ,\BB t_+^\delta]}  $ to be a path chosen in a $\sigma(\mcl T \mcl M)$-measurable manner which is equal to either the concatenation of $P$ followed by the time reversal of $\wt P$ or the concatenation of $\wt P$ followed by the time reversal of $P$.
\end{defn}

An example of a $\sigma(\mcl T \mcl M)$-measurable choice for $P_*$ is to declare that $P_*$ traces $P$ first if and only if the first vertex $y\in\mcl V\mcl G^n|_{[\BB t_-^\delta,t_+^\delta]}$ hit by $P$ but not $\wt P$ is numerically smaller than the first vertex $\wt y \in\mcl V\mcl G^n|_{[\BB t_-^\delta,t_+^\delta]}$ which is hit by $\wt P$ but not $P$. An example of a choice of $P_*$ which is \emph{not} $\sigma(\mcl T \mcl M)$-measurable is to declare that $P_*$ traces $P$ before $\wt P$. The reason for this is that we cannot tell from $\mcl T \mcl M$ which of the sets in~\eqref{eqn-delta-sets} corresponds to the left boundary and which corresponds to the right boundary.

The path $P_*$ visits all of the vertices of $\mcl G^n|_{[\BB t_-^\delta, \BB t_+^\delta]}$ whose corresponding cells intersect $\bdy \eta ([u,v])$. It visits these vertices in either counterclockwise order or clockwise order, but we cannot tell which just from seeing $\mcl T \mcl M$. 

We are now ready to define the \textbf{Tutte embedding} of $\mcl G^n_{[\BB t_-^\delta ,\BB t_+^\delta]}$, which will be a function 
\eqb\label{eq:tutte}
\Phi^n : \mcl V \mcl G^n_{[\BB t_-^\delta ,\BB t_+^\delta]} \rta \ol{\BB D} .
\eqe
The definition is essentially the same as in~\cite[Section 1.1.2]{gms-tutte}. We will need to consider the marked interior vertex
\eqb
\BB x^n := \begin{cases}
1/2 \quad &\text{$n$ is even} \\
(n+1)/2n \quad &\text{$n$ is odd}   
\end{cases} \in\mcl V\mcl G^n_{[\BB t_-^\delta ,\BB t_+^\delta]} .
\eqe
Our map $\Phi^n$ will take $\BB x^n$ to (approximately) 0 and $\BB y_-^n$ to (approximately) 1. 

We first define $\Phi^n$ for the vertices hit by $P_*$. For $k \in [0,K+\wt K]_{\BB Z}$, let $\frk p(k)$ be the conditional probability given $\mcl G^n_{[\BB t_-^\delta ,\BB t_+^\delta]}$ that the simple random walk on $\mcl G^n_{[\BB t_-^\delta ,\BB t_+^\delta]}$ started from $\BB x^n$ first hits the set of boundary vertices $P_*([0,K+\wt K]_{\BB Z})$ at a vertex in $P_*([0,k]_{\BB Z})$. We then set\footnote{
If $y \in\mcl V\mcl G^n_{[\BB t_-^\delta,\BB t_+^\delta]}$ and $\eta ([y-1/n,y])$ has a non-trivial intersection with both the left and right boundaries of $\eta ([\BB t_-^\delta,\BB t_+^\delta])$, then $P_*$ will hit $y$ twice, say at times $k_1$ and $k_2$. An example of a cell $\eta ([y-1/n,y])$ of this type is shown in the left panel of Figure~\ref{fig-tutte-domain}. This does not cause a problem for the definition of $\Phi^n(y)$ in~\eqref{eqn-tutte-bdy}, however, for the following reason. By looking that the cell $\eta ([y-1/n,y])$, one sees that the single vertex $y$ disconnects either $P_*([k_1+1,k_2-1]_{\BB Z})$ or $P_*([0,K+\wt K]_{\BB Z}\setminus [k_1+1,k_2-1]_{\BB Z})$ from $\BB x^n$ in $\mcl G^n_{[\BB t_-^\delta,\BB t_+^\delta]}$. In particular, the harmonic measure from $\BB x^n$ of one of these two boundary arcs is zero, which implies that $\Phi^n(k_1) = \Phi^n(k_2)$.
}
\eqb \label{eqn-tutte-bdy}
\Phi^n(P_*(k)) := e^{2\pi i \frk p(k)} .
\eqe
This makes it so that $\Phi^n$ pushes forward the discrete harmonic measure from $\BB x^n$ on the boundary path $P_*([0,K+\wt K]_{\BB Z})$ to (approximately) one-dimensional Lebesgue measure on $\bdy\BB D$. 

We then extend $\Phi^n$ to all of $\mcl V \mcl G^n_{[\BB t_-^\delta ,\BB t_+^\delta]}$ in such a way that $\Phi^n$ is discrete harmonic on $\mcl V \mcl G^n_{[\BB t_-^\delta ,\BB t_+^\delta]} \setminus P_*([0,K+\wt K]_{\BB Z})$. 

\begin{lem} \label{lem-tutte-determined}
The function $\Phi^n$ is a.s.\ determined by $\mcl T \mcl M$. 
\end{lem}
\begin{proof}
By definition, $\Phi^n$ is determined by $\mcl G^n_{[\BB t_-^\delta,\BB t_+^\delta]}$ and $P_*$. The graph $\mcl G^n_{[\BB t_-^\delta,\BB t_+^\delta]}$ is a.s.\ determined by $\mcl T \mcl M$ by Lemma~\ref{lem-graph-permuton} and the measurability statements just before~\eqref{eqn-delta-sets}. The path $P_*$ is a.s.\ determined by $\mcl T \mcl M$ by Definition~\ref{def-bdy-path}. 
\end{proof}

Recall that $\mu_h$ denotes the $\gamma$-LQG measure associated with a singly marked unit area $\gamma$-Liouville quantum sphere $(\BB C, h, \infty)$.

\begin{prop} \label{prop-tutte-conv}
Let $\Phi^n$ be the Tutte embedding function as above. Let $\Psi_\delta$ be the unique conformal map from the simply connected domain $ \eta ([\BB t_-^\delta , \BB t_+^\delta])$ to $ \ol{\BB D}$ which takes $\eta (1/2)$ to 0 and $\eta (\BB t_-^\delta)$ to 1. There is a sequence of maps $\{\Psi_\delta^n\}_{n\in\BB N}$, each of which is equal to either $\Psi_\delta$ or its complex conjugate $\ol\Psi_\delta$, such that as $n\rta\infty$, 
\eqb \label{eqn-tutte-conv}
\max_{x\in\mcl V\mcl G^n_{[\BB t_-^\delta,\BB t_+^\delta]}} \left| \Phi^n(x) - \Psi_\delta^n(\eta (x)) \right| \rta 0 , \quad \text{in probability}. 
\eqe 
In particular, 
\begin{itemize}
\item The Prokhorov distance between $(\Psi_\delta^n)_* \mu_h |_{\eta ([\BB t_-^\delta,\BB t_+^\delta])}$ and the measure which assigns mass $1/n$ to each point of $\Phi^n\left(\mcl V\mcl G^n_{[\BB t_-^\delta,\BB t_+^\delta]}\right)$ converges in probability to zero. 
\item The uniform distance between $\Psi_\delta^n \circ \eta |_{[\BB t_-^\delta,\BB t_+^\delta]}$ and the function $[\BB t_-^\delta,\BB t_+^\delta] \ni t \mapsto \Phi^n( \lfloor t n \rfloor /n )$ converges in probability to zero. 
\end{itemize} 
\end{prop}
\begin{proof}
We set $\Psi_\delta^n = \Psi_\delta$ (resp.\ $\Psi_\delta^n = \ol\Psi_\delta$) if $P_*^n$ visits the vertices of $\mcl V\mcl G^n_{[\BB t_-^\delta,\BB t_+^\delta]}$ which intersect $\bdy\eta ([\BB t_-^\delta,\BB t_+^\delta])$ in counterclockwise (resp.\ clockwise) order. 
With this choice, the convergence~\eqref{eqn-tutte-conv} is immediate from the definition of $\Phi^n$ and the convergence of random walk on $\mcl G^n$ to Brownian motion (Proposition~\ref{prop-rw-conv}). The other convergence statements follow from~\eqref{eqn-tutte-conv} and the fact that (by the continuity of $\eta $) the maximal Euclidean diameter of the cells $\eta ([x-1/n,x])$ tends to zero in probability as $n\rta\infty$. 
\end{proof}

In order to recover the whole curve-decorated quantum surface $(\BB C , h , \infty,\eta )$ from $\mcl T \mcl M$ via Proposition~\ref{prop-tutte-conv}, we need to know that $\eta ([\BB t_-^\delta,\BB t_+^\delta])$ covers most of $\BB C$ when $\delta$ is small. This is the purpose of the following lemma.

\begin{lem} \label{lem-endpoints}
Almost surely, $\BB t_-^\delta \rta 0$ and $\BB t_+^\delta \rta 1$ as $\delta \rta 0$.
\end{lem}
\begin{proof}
Almost surely, $\eta $ is a continuous path on the Riemann sphere which starts and ends at $\infty$, and each point of $\BB C$ is hit by $\eta $ at some time. Hence, a.s.\ for each $R>0$, it holds for each small enough $\delta \in (0,1/2)$ that $B_R(0) \subset \eta ([\delta,1-\delta])$. Since $\eta ([\BB t_-^\delta,\BB t_+^\delta])$ is the largest simply connected subset of $\eta ([\delta,1-\delta])$ which contains $\eta (1/2)$, it follows that a.s.\ $B_R(0) \subset \eta ([\BB t_-^\delta,\BB t_+^\delta])$ for each small enough $\delta>0$. As $R\rta\infty$, we have $\mu_h(B_R(0)) \rta 1$. Therefore, $\BB t_+^\delta - \BB t_-^\delta$ must converge to 1 as $\delta \rta 0$, which gives the lemma statement. 
\end{proof}

\begin{proof}[Proof of Theorem~\ref{thm-sle-msrble}]
By Lemma~\ref{lem-tutte-determined} and Proposition~\ref{prop-tutte-conv}, if $\Psi_\delta$ is the conformal map as in Proposition~\ref{prop-tutte-conv}, then a.s.\ for each $\delta \in (0,1/2)$ the set $\mcl T \mcl M$ determines the pair 
\begin{equation*}
	\left((\Psi_\delta)_*  \mu_h |_{\eta ([\BB t_-^\delta,\BB t_+^\delta])} , \Psi_\delta\circ \eta |_{[\BB t_-^\delta,\BB t_+^\delta]} \right)
\end{equation*}
 modulo complex conjugation. By~\cite[Theorem 1.1]{bss-lqg-gff}, a GFF-type distribution is a.s.\ determined by its associated LQG area measure. Consequently, a.s.\ $\mcl T \mcl M$ determines the pair
\eqb \label{eqn-determined-gff-sle}
 \left( h\circ \Psi_\delta^{-1} + Q\log |(\Psi_\delta^{-1})'| , \Psi_\delta\circ \eta |_{[\BB t_-^\delta,\BB t_+^\delta]} \right) 
\eqe
modulo complex conjugation. 

By Lemma~\ref{lem-endpoints}, a.s.\ $\eta ([\BB t_-^\delta , \BB t_+^\delta])$ increases to all of $\BB C$ as $\delta \rta 0$. By basic distortion estimates for conformal maps, for each $R > 0$ it holds for each small enough $\delta > 0$ that $\Psi_\delta$ is approximately complex affine on $B_R(0)$. Consequently, there are random complex affine maps $f_\delta$ such that $f_\delta \circ \Psi_\delta$ converges to the identity map uniformly on compact subsets of $\BB C$. Therefore, a.s.\ 
\eqb
 \left( h\circ (f_\delta \circ \Psi_\delta)^{-1} + Q\log |( (f_\delta \circ \Psi_\delta )^{-1})'| , f_\delta\circ \Psi_\delta\circ \eta |_{[\BB t_-^\delta,\BB t_+^\delta]} \right) 
\eqe
converges to $(h,\eta )$ with respect to the distributional topology on the first coordinate and the appropriate localized uniform topology on the second coordinate. Since $\mcl T \mcl M$ a.s.\ determines the pair~\eqref{eqn-determined-gff-sle} modulo complex conjugation, we get that $\mcl T \mcl M$ a.s.\ determines the curve-decorated quantum surface $(\BB C , h, \infty,\eta )$ modulo complex affine maps and complex conjugation.  
\end{proof}

\subsection{Summary on how the permuton determines the curve-decorated quantum sphere and proofs of the corollaries}\label{sect:explicit-function}

We first detail the function $F$ in \cref{eq:function-det}, that is, the function from \cref{thm-permuton-determ-lqg-sles} that, given the support of $\perm$, reconstructs the curve-decorated quantum sphere $(\BB C , h ,  \infty , \eta_1 , \eta_2 )_{\mathrm{cc}}$ (recall the notation $\mathrm{cc}$ from \eqref{eq:cc}). Then we give the proofs of \cref{cor:mutual-sing1} and \cref{cor:meandric-perm}. 

\medskip

We fix $\gamma \in (0,2)$, a singly marked unit area $\gamma$-Liouville quantum sphere $(\BB C , h , \infty)$, and a pair $(\eta_1,\eta_2)$ of whole-plane space-filling SLEs from $\infty$ to $\infty$, sampled independently from $h$ and then parametrized by $\gamma$-LQG mass with respect to $h$.  
We let $\perm$ be the permuton associated with $(\eta_1,\eta_2)$ as in~\eqref{eqn-permuton-def} and assume that either Assumption~\ref{item-indep} or Assumption~\ref{item-ig-gff} holds. Recall also the definition of $\mcl T \mcl M (\eta_1)$ from \eqref{eqn-intersect-set}. We describe $F$ in four steps:

\begin{enumerate}
	\item \label{ybwvwfei} \textbf{Reconstruction of $\mcl T \mcl M (\eta_1)$ from the support of $\perm$}: Given the support of $\perm$, we first create the \textbf{augmented support} of $\perm$. First of all, the augmented support of $\perm$ contains all the points in  the  support of $\perm$. Then, for every pair of points $(s,q),(t,q)\in \op{supp} \perm$, run the following operations (see \cref{fig-reconstr}): 
	\begin{enumerate}
		\item Draw an horizontal line on $[0,1]^2$ passing through $(s,q)$ and $(t,q)$, a vertical line on $[0,1]^2$ passing through $(s,q)$, and a vertical line on $[0,1]^2$ passing through $(t,q)$.
		\item For each added vertical line, check if there is a point in $\op{supp} \perm$ belonging to the line other than $(t,q)$ or $(s,q)$. If there is, add an horizontal line on $[0,1]^2$ passing through that point and iterate $(1.b)$ on this line (with ``vertical'' and ``horizontal'' swapped). If there isn't, do nothing.
		\item Look at all the lines added in $(1.a)$ and $(1.b)$ and add to the augmented support of $\perm$ all the points that lie at the intersection of two of these lines.\footnote{If $z$ is the point $\eta_1(s)=\eta_1(t)=\eta_2(q)$, then all the points that lie at the intersection of two of these lines are exactly the points  $(t^1_i,t^2_j)$, where $t^1_1,t^1_2,\dots,t^1_{m_1}$ are  the $m_1\geq 2$ times  when $\eta_1$ hits $z$ and   $t^2_1,t^2_2,\dots,t^2_{m_2}$ are the $m_2\geq 1$ times when $\eta_2$ hits $z$.}
	\end{enumerate}
	After running the analogue operation for every pair of points $(q,s),(q,t)\in \op{supp} \perm$, add to $\mcl T \mcl M (\eta_1)$ all the points $(s,t)\in[0,1]^2$ for which there exists $q\in[0,1]$ such that $(s,q),(t,q)$ are both in the augmented closed support of $\perm$. This operation works thanks to Propositions~\ref{prop-permuton-full},~\ref{prop-counterflow-dichotomy}~and~\ref{prop-counterflow-dichotomy2}.
	
	\item \textbf{Reconstruction of $\mu_h$ and $\eta_1$ from $\mcl T \mcl M (\eta_1)$}: To reconstruct $\mu_h$ and $\eta_1$ from $\mcl T \mcl M (\eta_1)$ it is sufficient to run the following operations:
	\begin{enumerate}
		\item Construct the graph $\mcl G^n$ from $\mcl T \mcl M (\eta_1)$ as explained in \cref{lem-graph-permuton};
		\item Define its Tutte embedding $\Phi^n : \mcl V \mcl G^n_{[\BB t_-^\delta ,\BB t_+^\delta]} \rta \ol{\BB D}$ as explained in the text around \eqref{eqn-tutte-bdy}; 
		\item Send $n \to \infty$ to recover $\mu_h$ and $\eta_1$ parametrized by $\mu_h$-mass (both modulo scaling, translation, rotation and reflection) as explained in the proof of Theorem~\ref{thm-sle-msrble}.
	\end{enumerate}
	
	\item \textbf{Reconstruction of $\eta_2$ from $\mcl T \mcl M (\eta_1)$ and the support of $\perm$}: Run the operations described in the proof of \cref{cor-permuton-msrble0}.
	
	\item \textbf{Reconstruction of $h$ from $\mu_h$}: Run the strategy explained in \cite{bss-lqg-gff}: For $\varepsilon > 0$ and a point $z$ in the plane, let $h^\varepsilon(z) := (1/\gamma) \log (\mu_h(B_\varepsilon(z)))$, where $B_\varepsilon(z)$ is the open ball of radius $\varepsilon$ centered at $z$. In \cite{bss-lqg-gff}, the authors show that $h^\varepsilon - \mathbb E[ h^\varepsilon ]$ converges to $h$ in probability as $\varepsilon \to 0$. 
\end{enumerate}

Note that the first three steps of the above procedure do not depend either on whether Assumption~\ref{item-indep}~or~\ref{item-ig-gff} is adopted, or on the choice of the parameters $(\gamma,\kappa_1,\kappa_2)$ in Assumption~\ref{item-indep} or $(\gamma,\kappa,\theta)$ in Assumption~\ref{item-ig-gff}. In other words, the function that, given the support of $\perm$, determines $(\mu_h,\eta_1,\eta_2)$ up to scaling, translation, rotation and reflection, is always the same. We point out that the procedure in Step 4 depends on the parameter $\gamma$ of the Liouville quantum gravity measure $\mu_h$.

\begin{figure}[ht!]
	\begin{center}
		\includegraphics[width=.99\textwidth]{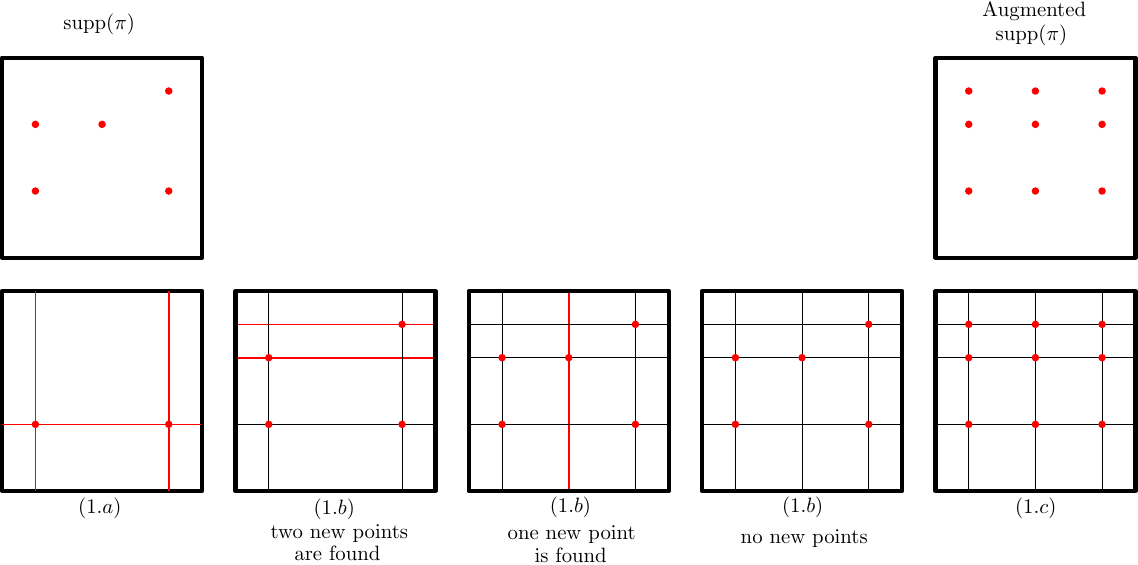}  
		\caption{\label{fig-reconstr}
			Let $z$ be a multiple point both for $\eta_1$ and $\eta_2$.
			Let  $t^1_1,t^1_2,\dots,t^1_{m_1}$ be the $m_1\geq 2$ times when $\eta_1$ hits $z$ and $t^2_1,t^2_2,\dots,t^2_{m_2}$ be the $m_2\geq 2$ times when $\eta_2$ hits $z$. In general, as shown on the top-left unit square, the support of the permuton $\perm$ contains only some of the points $(t^1_i,t^1_j)$. By running the steps (1.a-b-c) described in Item~\ref{ybwvwfei} and shown in the second line of the picture above (the new lines added at each step are shown in red), we can recover all these points $(t^1_i,t^1_j)$ obtaining the augmented support of the permuton $\perm$ shown on the top-right unit square.}
	\end{center}
	\vspace{-3ex}
\end{figure}

\medskip

We finally give the proofs of \cref{cor:mutual-sing1} and \cref{cor:meandric-perm}. 
	
	\begin{proof}[Proof of \cref{cor:mutual-sing1}]
		Given the support of the permuton $\perm$, we can recover $(\mu_h, \eta_1, \eta_2)$ modulo scaling, translation, rotation and reflection via the first three steps described above, which we recall do not depend either on whether Assumption~\ref{item-indep}~or~\ref{item-ig-gff} is adopted, or on the choice of the parameters $(\gamma,\kappa_1,\kappa_2)$ in Assumption~\ref{item-indep} or $(\gamma,\kappa,\theta)$ in Assumption~\ref{item-ig-gff}.
		
		In the case of Assumption~\ref{item-indep}, the law of the equivalence classes $(\mu_h, \eta_1 , \eta_2 )_{\mathrm{cc}}$ for different vales of $(\gamma,\kappa_1,\kappa_2)$ are mutually singular, for instance for the following two reasons:
		\begin{itemize}
			\item Recall from \cref{sec:SLE_as_flow} that the boundary of a space-filling SLE$_\kappa$ curve stopped at any given time is an SLE$_{16/\kappa}$-type curve, which has Hausdorff dimension $1+\frac{16/\kappa}{8}$ by \cite{beffara-dim}. This gives the mutual singularity for different values of $\kappa_1$ and $\kappa_2$.
			
			\item The $\gamma$-LQG measure is supported on the $\gamma$-thick points of $h$, which is a set of Hausdorff dimension $2-\frac{\gamma^2}{2}$ (see, e.g., \cite[Theorem 2.4]{bp-lqg-notes}). This gives the mutual singularity for different values of $\gamma$.
		\end{itemize}
		
		In the case of Assumption~\ref{item-ig-gff}, the law of the equivalence classes $(\mu_h, \eta_1 , \eta_2 )_{\mathrm{cc}}$ for different vales of $(\gamma,\kappa, \theta)$ are mutually singular, for instance for the following reason (note that we only need to argue for $\theta$, since the mutual singularity for different values of $\gamma,\kappa$ is already proved above).

		Let $\beta^L_1$ and $\beta^L_2$ be the $\pi/2$-angle and $\theta$-angle flow lines of the same whole-plane GFF $\hat h$ started at a point $Z$ sampled uniformly from the LQG-measure $\mu_h$. In this way,  $\beta^L_1$ and $\beta^L_2$ are the left outer boundaries of $\eta_1$ and $\eta_2$, respectively, stopped upon hitting $Z$.
		
		If $\theta$ is such that $\beta^L_1$ and $\beta^L_2$ intersect somewhere besides at their starting point, then the Hausdorff dimension of the intersection depends on $\theta$ by \cite[Theorem 1.5]{miller-wu-dim}. This gives the mutual singularity in this case.

		If $\theta$ is such that $\beta^L_1$ and $\beta^L_2$ do not intersect, let $I$ be the set of angles $\phi$ for which the $\phi$-angle flow lines started from $Z$ intersects $\beta^L_1$ but not $\beta^L_2$, and let $J$ be the set of angles $\phi$ for which the $\phi$-angle flow line started from $Z$ intersects neither $\beta^L_1$ nor $\beta^L_2$ (note that the $\phi$ angle flow line is determined by $\hat h$, which is determined by $\eta_1$ by \cite[Theorem 1.16]{ig1}). Then the pair of sets $(I,J)$ depends on on $\theta$ (e.g., by the same argument as in \cref{lem-angle-condition}). This gives the mutual singularity in this case.
		
		\medskip
		
		Finally, for the statement on the skew Brownian permutons $\mu_{\rho,q}$, it is enough to note that the permutons $\mu_{\rho,q}$ when $(\rho,q) \in (-1,1) \times (0,1)$ are exactly the permutons in Assumption~\ref{item-ig-gff} when $\gamma^2=16/\kappa$, as shown in \cite[Theorem 1.17]{borga-skew-permuton}. Therefore in this regime, $\mu_{\rho,q}$ are mutually singular thanks to the previous paragraphs. 
	\end{proof}

\begin{remark}
	When $\rho=1$, $\mu_{1,q}$ coincides -- by definition -- with the Brownian separable permuton of parameter $1-q$ introduced in~\cite{bbfgp-universal}. Combining \cite[Theorem 1.10]{bhsy-baxter-permuton} and \cite[Equation (20)]{bbfgp-universal}, we have that $\mu_{1,q}$ and $\mu_{\rho,q}$ for $\rho\neq 1$ are also mutually singular. Finally, $\mu_{1,q}$ and $\mu_{1,q'}$ are also mutually singular thanks to \cite[Proposition 1.3]{frl-recursiveperm}.
\end{remark}

We conclude with the proof of \cref{cor:meandric-perm}.

\begin{proof}[Proof of \cref{cor:meandric-perm}]
Suppose that we are in the setting of Theorem~\ref{thm-permuton-determ-lqg-sles} and that the permuton $\perm$ satisfies the re-rooting invariance $\perm_T \eqD \perm$ when $T$ is sampled uniformly from $[0,1]$, independently from $\perm$. We need to show that $\kappa_1=\kappa_2=8$.

For $i\in \{1,2\}$, we define the re-rooting of $\eta_i$ at $z\in \BB C$ to be the curve $\eta_i^z : [0,1]\rta \BB C$ obtained by concatenating $\eta_i|_{[\tau_i^z , 1]}$ followed by $\eta_i|_{[0,\tau_i^z]}$ where $\tau_i^z$ is the first hitting time of $z$. That is,
			\begin{equation*}
				\eta_i^z(t) := \begin{cases}
					\eta_i(t + \tau_i^z) ,\quad &\text{if $t\in [0, 1-\tau_i^z]$} ,\\
					\eta_i(t - (1-\tau_i^z) ) ,\quad &\text{if $t \in [ 1-\tau_i^z ,1]$}  .
				\end{cases}
			\end{equation*}
It is shown in~\cite[Proposition 5.2]{bgs-meander} that if $\kappa_i = 8$, then for any $z\in\BB C$ and any fractional linear transformation $\phi_z : \BB C \cup\{\infty\} \to \BB C \cup\{\infty\}$ taking $z$ to $\infty$, we have $\phi_z\circ \eta_i \eqD \eta_i$ modulo time parametrization; and that this property is not true for any $z\in\BB C$ if $\kappa_i \not= 8$. The proof will follow easily from this fact together with Theorem~\ref{thm-permuton-determ-lqg-sles}. 
		
Let us now give the details. To avoid complications with choices of conformal maps, it is convenient to consider a triply marked quantum sphere $(\BB C , h , 0, 1 , \infty)$ (recall Definition~\ref{def-sphere}). Note that the single marked quantum sphere $(\BB C ,  h , \infty)$ used to construct $\perm$ can be recovered from the triply marked quantum sphere $(\BB C , h , 0, 1 , \infty)$ by forgetting two of the three marked points. 

Conditional on $(h,\eta_1,\eta_2)$, let $Z\in \BB C$ be sampled from the probability measure $\mu_h$. Consider the curve-decorated quantum surface $(\BB C , h ,  0, 1,   Z , \eta_1^Z , \eta_2^Z )$ obtained from $(\BB C , h ,  0, 1,  \infty , \eta_1 , \eta_2 )$ by re-rooting at $Z$. Let $\widetilde\perm$ be the permuton associated with $(\eta_1^Z, \eta_2^Z)$ as in~\eqref{eqn-permuton-def}.

Since $\eta_1$ is parametrized by $\mu_h$-mass, the time $\tau_1^Z$ is uniformly distributed on $[0,1]$ and is independent from $(h,\eta_1,\eta_2)$, and hence also from $\perm$. 
By definition of $\perm_t $ (recall \eqref{eq:rerooting}) and Assertion 1 in \cref{lem-permuton-defined} we have that, almost surely, for each rectangle $[a,b]\times[c,d] \subset [0,1]^2$, 
\begin{equation*}
	\perm_{\tau_1^Z}\left([a,b]\times[c,d]\right) = \mu_h\left( \eta_1^Z([a,b]) \cap \eta_2^Z([c,d]) \right),
\end{equation*}
where we used that $\psi(\tau_Z^i) = \tau_Z^2$ is the a.s.\ unique time at which $\eta_2$ hits $Z$ (\cite[Lemma 2.10]{bgs-meander}). By construction of $\widetilde \perm$, this implies that almost surely $\wt\perm = \perm_{\tau_1^Z}$. 
Since $\tau_1^Z$ is independent from $\perm$, our re-rooting invariance assumption implies that 
$$\wt\perm \eqD \perm.$$ 
Now, by \cref{thm-permuton-determ-lqg-sles} and \cref{eq:function-det}, we have that
\begin{equation*}
	(\BB C , h ,  \infty , \eta_1 , \eta_2 )_{\mathrm{cc}}=F(\op{supp} \perm)\quad\text{ and }\quad(\BB C , h ,  Z , \eta_1^Z , \eta_2^Z )_{\mathrm{cc}} = F(\op{supp} \widetilde\perm).
\end{equation*} 
Note that we can use the same function $F$ for both permutons thanks to the description of $F$ at the beginning of \cref{sect:explicit-function}. Hence $(\BB C , h ,  \infty , \eta_1 , \eta_2 )_{\mathrm{cc}} \eqD (\BB C , h ,  Z , \eta_1^Z , \eta_2^Z )$, and so also
\begin{equation*}
	(\BB C , h , 0, 1, \infty , \eta_1 , \eta_2 )_{\mathrm{cc}} \eqD (\BB C , h , 0 ,1,  Z , \eta_1^Z , \eta_2^Z )_{\mathrm{cc}}.
\end{equation*} 
To convert this to a statement about $\eta_1$ and $\eta_2$, let $\phi_Z(w) = w(1-Z)/(w-Z)$ be the conformal which takes $Z$ to $\infty$ and fixes 0 and 1. Let $\xi , \wt \xi$ be random functions which are each equal to the identity with probability $1/2$, and equal to $w\mapsto \ol w$ with probability $1/2$, chosen independently from each other and everything else. The previous paragraph implies that for each $i\in \{1,2\}$,
\begin{equation*}
	\wt\xi \circ \phi_Z \circ \eta_i^Z \eqD \xi\circ \eta_i.
\end{equation*} 
Since the law of SLE$_\kappa$ is invariant under complex conjugation, this implies that $\phi_Z\circ \eta_i^Z \eqD \eta_i$. Since $Z$ is sampled independently from $\eta_i$, viewed modulo time parametrization, we get from~\cite[Proposition 5.2]{bgs-meander} that $\kappa_1=\kappa_2=8$. 
\end{proof}

\bibliography{cibib,cibib2}

\newcommand{\etalchar}[1]{$^{#1}$}
\def\cprime{$'$}
\begin{thebibliography}{GKMW18}

\bibitem[AFS20]{afs-metric-ball}
M.~Ang, H.~Falconet, and X.~Sun.
\newblock Volume of metric balls in {L}iouville quantum gravity.
\newblock {\em Electron. J. Probab.}, 25:Paper No. 160, 50, 2020,
  \arxiv{2001.11467}. \MR{4193901}

\bibitem[AHS17]{ahs-sphere}
J.~Aru, Y.~Huang, and X.~Sun.
\newblock Two perspectives of the 2{D} unit area quantum sphere and their
  equivalence.
\newblock {\em Comm. Math. Phys.}, 356(1):261--283, 2017, \arxiv{1512.06190}.
  \MR{3694028}

\bibitem[BBF{\etalchar{+}}20]{bbfgp-universal}
F.~Bassino, M.~Bouvel, V.~F\'{e}ray, L.~Gerin, M.~Maazoun, and A.~Pierrot.
\newblock Universal limits of substitution-closed permutation classes.
\newblock {\em J. Eur. Math. Soc. (JEMS)}, 22(11):3565--3639, 2020,
  \arxiv{1706.08333}. \MR{4167015}

\bibitem[BBFS20]{bbfs-tree-sep}
J.~Borga, M.~Bouvel, V.~F\'{e}ray, and B.~Stufler.
\newblock A decorated tree approach to random permutations in
  substitution-closed classes.
\newblock {\em Electron. J. Probab.}, 25:Paper No. 67, 52, 2020,
  \arxiv{1904.07135}. \MR{4115736}

\bibitem[Bef08]{beffara-dim}
V.~Beffara.
\newblock The dimension of the {SLE} curves.
\newblock {\em Ann. Probab.}, 36(4):1421--1452, 2008, \arxiv{math/0211322}.
  \MR{2435854 (2009e:60026)}

\bibitem[BG22]{bg-lbm}
N.~Berestycki and E.~Gwynne.
\newblock Random walks on mated-{CRT} planar maps and {L}iouville {B}rownian
  motion.
\newblock {\em Comm. Math. Phys.}, 395(2):773--857, 2022, \arxiv{2003.10320}.
  \MR{4487526}

\bibitem[BGS22]{bgs-meander}
J.~{Borga}, E.~{Gwynne}, and X.~{Sun}.
\newblock {Permutons, meanders, and SLE-decorated Liouville quantum gravity}.
\newblock {\em ArXiv e-prints}, July 2022, \arxiv{2207.02319}.

\bibitem[BHS23]{bhs-site-perc}
O.~Bernardi, N.~Holden, and X.~Sun.
\newblock Percolation on triangulations: a bijective path to {L}iouville
  quantum gravity.
\newblock {\em Mem. Amer. Math. Soc.}, 289(1440):v+176, 2023,
  \arxiv{1807.01684}. \MR{4651497}

\bibitem[BHSY23]{bhsy-baxter-permuton}
J.~Borga, N.~Holden, X.~Sun, and P.~Yu.
\newblock Baxter permuton and {Liouville} quantum gravity.
\newblock {\em Probab. Theory Relat. Fields}, 186(3-4):1225--1273, 2023,
  \arxiv{2203.12176}.

\bibitem[BM22]{bm-baxter-permutation}
J.~Borga and M.~Maazoun.
\newblock Scaling and local limits of {B}axter permutations and bipolar
  orientations through coalescent-walk processes.
\newblock {\em Ann. Probab.}, 50(4):1359--1417, 2022, \arxiv{2008.09086}.
  \MR{4420422}

\bibitem[Bor22]{borga-strong-baxter}
J.~Borga.
\newblock The permuton limit of strong-{Baxter} and semi-{Baxter} permutations
  is the skew {Brownian} permuton.
\newblock {\em Electron. J. Probab.}, 27:53, 2022, \arxiv{2112.00159}.
\newblock Id/No 158.

\bibitem[Bor23]{borga-skew-permuton}
J.~Borga.
\newblock The skew {Brownian} permuton: {A} new universality class for random
  constrained permutations.
\newblock {\em Proc. Lond. Math. Soc. (3)}, 126(6):1842--1883, 2023,
  \arxiv{2112.00156}.

\bibitem[BP24]{bp-lqg-notes}
N.~{Berestycki} and E.~{Powell}.
\newblock {\em {Gaussian free field and Liouville quantum gravity}}.
\newblock Cambridge University Press, 2024, \arxiv{2404.16642}.
\newblock To appear.

\bibitem[BSS14]{bss-lqg-gff}
N.~{Berestycki}, S.~{Sheffield}, and X.~{Sun}.
\newblock {Equivalence of Liouville measure and Gaussian free field}.
\newblock {\em ArXiv e-prints}, October 2014, \arxiv{1410.5407}.

\bibitem[DMS21]{wedges}
B.~Duplantier, J.~Miller, and S.~Sheffield.
\newblock Liouville quantum gravity as a mating of trees.
\newblock {\em Ast\'{e}risque}, (427):viii+257, 2021, \arxiv{1409.7055}.
  \MR{4340069}

\bibitem[DS11]{shef-kpz}
B.~Duplantier and S.~Sheffield.
\newblock Liouville quantum gravity and {KPZ}.
\newblock {\em Invent. Math.}, 185(2):333--393, 2011, \arxiv{1206.0212}.
  \MR{2819163 (2012f:81251)}

\bibitem[Dub09]{dubedat-duality}
J.~Dub{\'e}dat.
\newblock Duality of {S}chramm-{L}oewner evolutions.
\newblock {\em Ann. Sci. \'Ec. Norm. Sup\'er. (4)}, 42(5):697--724, 2009,
  \arxiv{0711.1884}. \MR{2571956 (2011g:60151)}

\bibitem[FPS09]{fps-counting-bipolar}
{\'E}.~Fusy, D.~Poulalhon, and G.~Schaeffer.
\newblock Bijective counting of plane bipolar orientations and {S}chnyder
  woods.
\newblock {\em European J. Combin.}, 30(7):1646--1658, 2009,
  \arxiv{arXiv:0803.0400}. \MR{2548656 (2010j:05047)}

\bibitem[FRL23]{frl-recursiveperm}
V.~Féray and K.~Rivera-Lopez.
\newblock The permuton limit of random recursive separable permutations.
\newblock {\em Confluentes Mathematici}, 15:45--82, 2023, \arxiv{2306.04278}.

\bibitem[GHM20]{ghm-kpz}
E.~Gwynne, N.~Holden, and J.~Miller.
\newblock An almost sure {KPZ} relation for {SLE} and {B}rownian motion.
\newblock {\em Ann. Probab.}, 48(2):527--573, 2020, \arxiv{1512.01223}.
  \MR{4089487}

\bibitem[GHS19]{ghs-dist-exponent}
E.~{Gwynne}, N.~{Holden}, and X.~{Sun}.
\newblock {A distance exponent for Liouville quantum gravity}.
\newblock {\em {Probability Theory and Related Fields}}, 173(3):931--997, 2019,
  \arxiv{1606.01214}.

\bibitem[GHS20]{ghs-map-dist}
E.~Gwynne, N.~Holden, and X.~Sun.
\newblock A mating-of-trees approach for graph distances in random planar maps.
\newblock {\em Probab. Theory Related Fields}, 177(3-4):1043--1102, 2020,
  \arxiv{1711.00723}. \MR{4126936}

\bibitem[GHS23]{ghs-mating-survey}
E.~Gwynne, N.~Holden, and X.~Sun.
\newblock Mating of trees for random planar maps and {L}iouville quantum
  gravity: a survey.
\newblock In {\em Topics in statistical mechanics}, volume~59 of {\em Panor.
  Synth\`eses}, pages 41--120. Soc. Math. France, Paris, 2023,
  \arxiv{1910.04713}. \MR{4619311}

\bibitem[GKMW18]{gkmw-burger}
E.~Gwynne, A.~Kassel, J.~Miller, and D.~B. Wilson.
\newblock Active {S}panning {T}rees with {B}ending {E}nergy on {P}lanar {M}aps
  and {SLE}-{D}ecorated {L}iouville {Q}uantum {G}ravity for {$\kappa > 8$}.
\newblock {\em Comm. Math. Phys.}, 358(3):1065--1115, 2018, \arxiv{1603.09722}.
  \MR{3778352}

\bibitem[GMS20]{gms-poisson-voronoi}
E.~Gwynne, J.~Miller, and S.~Sheffield.
\newblock The {T}utte {E}mbedding of the {P}oisson--{V}oronoi {T}essellation of
  the {B}rownian {D}isk {C}onverges to {$\sqrt{8/3}$}-{L}iouville {Q}uantum
  {G}ravity.
\newblock {\em Comm. Math. Phys.}, 374(2):735--784, 2020, \arxiv{1809.02091}.
  \MR{4072229}

\bibitem[GMS21]{gms-tutte}
E.~Gwynne, J.~Miller, and S.~Sheffield.
\newblock The {T}utte embedding of the mated-{CRT} map converges to {L}iouville
  quantum gravity.
\newblock {\em Ann. Probab.}, 49(4):1677--1717, 2021, \arxiv{1705.11161}.

\bibitem[GMS22]{gms-random-walk}
E.~Gwynne, J.~Miller, and S.~Sheffield.
\newblock An invariance principle for ergodic scale-free random environments.
\newblock {\em Acta Math.}, 228(2):303--384, 2022, \arxiv{1807.07515}.
  \MR{4448682}

\bibitem[Gwy20]{gwynne-ams-survey}
E.~Gwynne.
\newblock Random surfaces and {L}iouville quantum gravity.
\newblock {\em Notices Amer. Math. Soc.}, 67(4):484--491, 2020,
  \arxiv{1908.05573}. \MR{4186266}

\bibitem[HS23]{hs-cardy-embedding}
N.~Holden and X.~Sun.
\newblock Convergence of uniform triangulations under the {C}ardy embedding.
\newblock {\em Acta Math.}, 230(1):93--203, 2023, \arxiv{1905.13207}.
  \MR{4567714}

\bibitem[Kah85]{kahane}
J.-P. Kahane.
\newblock Sur le chaos multiplicatif.
\newblock {\em Ann. Sci. Math. Qu\'ebec}, 9(2):105--150, 1985. \MR{829798
  (88h:60099a)}

\bibitem[KMSW19]{kmsw-bipolar}
R.~Kenyon, J.~Miller, S.~Sheffield, and D.~B. Wilson.
\newblock Bipolar orientations on planar maps and {${\mathrm SLE}_{12}$}.
\newblock {\em Ann. Probab.}, 47(3):1240--1269, 2019, \arxiv{1511.04068}.
  \MR{3945746}

\bibitem[{La }03]{lacroix-meander-survey}
M.~{La Croix}.
\newblock {A}pproaches to the enumerative theory of meanders.
\newblock {A}vailable at
  \url{https://math.mit.edu/~malacroi/Latex/Meanders.pdf}, 2003.

\bibitem[Law05]{lawler-book}
G.~F. Lawler.
\newblock {\em Conformally invariant processes in the plane}, volume 114 of
  {\em Mathematical Surveys and Monographs}.
\newblock American Mathematical Society, Providence, RI, 2005. \MR{2129588
  (2006i:60003)}

\bibitem[{Le }19]{legall-brownian-geometry}
J.-F. {Le Gall}.
\newblock Brownian geometry.
\newblock {\em Jpn. J. Math.}, 14(2):135--174, 2019, \arxiv{1810.02664}.
  \MR{4007665}

\bibitem[LSW24]{lsw-schnyder-wood}
Y.~Li, X.~Sun, and S.~S. Watson.
\newblock Schnyder woods, {{\(\mathrm{SLE}_{16}\)}}, and {Liouville} quantum
  gravity.
\newblock {\em Trans. Am. Math. Soc.}, 377(4):2439--2493, 2024,
  \arxiv{1705.03573}.

\bibitem[MS16]{ig1}
J.~Miller and S.~Sheffield.
\newblock Imaginary geometry {I}: interacting {SLE}s.
\newblock {\em Probab. Theory Related Fields}, 164(3-4):553--705, 2016,
  \arxiv{1201.1496}. \MR{3477777}

\bibitem[MS17]{ig4}
J.~Miller and S.~Sheffield.
\newblock Imaginary geometry {IV}: interior rays, whole-plane reversibility,
  and space-filling trees.
\newblock {\em Probab. Theory Related Fields}, 169(3-4):729--869, 2017,
  \arxiv{1302.4738}. \MR{3719057}

\bibitem[MS19]{sphere-constructions}
J.~Miller and S.~Sheffield.
\newblock Liouville quantum gravity spheres as matings of finite-diameter
  trees.
\newblock {\em Ann. Inst. Henri Poincar\'{e} Probab. Stat.}, 55(3):1712--1750,
  2019, \arxiv{1506.03804}. \MR{4010949}

\bibitem[MW17]{miller-wu-dim}
J.~Miller and H.~Wu.
\newblock Intersections of {SLE} {P}aths: the double and cut point dimension of
  {SLE}.
\newblock {\em Probab. Theory Related Fields}, 167(1-2):45--105, 2017,
  \arxiv{1303.4725}. \MR{3602842}

\bibitem[Pom92]{pom-book}
C.~Pommerenke.
\newblock {\em Boundary behaviour of conformal maps}, volume 299 of {\em
  Grundlehren der Mathematischen Wissenschaften [Fundamental Principles of
  Mathematical Sciences]}.
\newblock Springer-Verlag, Berlin, 1992. \MR{1217706 (95b:30008)}

\bibitem[RV11]{rhodes-vargas-log-kpz}
R.~Rhodes and V.~Vargas.
\newblock K{PZ} formula for log-infinitely divisible multifractal random
  measures.
\newblock {\em ESAIM Probab. Stat.}, 15:358--371, 2011, \arxiv{0807.1036}.
  \MR{2870520}

\bibitem[She16]{shef-burger}
S.~Sheffield.
\newblock Quantum gravity and inventory accumulation.
\newblock {\em Ann. Probab.}, 44(6):3804--3848, 2016, \arxiv{1108.2241}.
  \MR{3572324}

\bibitem[She23]{sheffield-icm}
S.~Sheffield.
\newblock What is a random surface?
\newblock In {\em I{CM}---{I}nternational {C}ongress of {M}athematicians.
  {V}ol. 2. {P}lenary lectures}, pages 1202--1258. EMS Press, Berlin, [2023]
  \copyright 2023, \arxiv{2203.02470}. \MR{4680280}

\bibitem[SS09]{ss-dgff}
O.~Schramm and S.~Sheffield.
\newblock Contour lines of the two-dimensional discrete {G}aussian free field.
\newblock {\em Acta Math.}, 202(1):21--137, 2009, \arxiv{math/0605337}.
  \MR{2486487 (2010f:60238)}

\bibitem[Zha08]{zhan-duality1}
D.~Zhan.
\newblock Duality of chordal {SLE}.
\newblock {\em Invent. Math.}, 174(2):309--353, 2008, \arxiv{0712.0332}.
  \MR{2439609 (2010f:60239)}

\bibitem[Zha10]{zhan-duality2}
D.~Zhan.
\newblock Duality of chordal {SLE}, {II}.
\newblock {\em Ann. Inst. Henri Poincar\'e Probab. Stat.}, 46(3):740--759,
  2010, \arxiv{0803.2223}. \MR{2682265 (2011i:60155)}

\end{thebibliography}
\bibliographystyle{hmralphaabbrv}

\end{document}